\documentclass{amsart}
\usepackage{graphicx}
\usepackage[all]{xy}
\usepackage{latexsym}\usepackage[usenames,dvipsnames]{color}
\usepackage{amssymb}
\usepackage{amsmath}
\usepackage{color}
\xyoption{arc}

\def\id{{\mbox{1 \hskip -7pt 1}}}
\newcommand{\sgn}{{\mathit s  \mathit g\mathit  n}}
 \newcommand{\lon}{\longrightarrow}
 \newcommand{\bu}{\bullet}
 
 \newcommand{\rar}{\rightarrow}

 \newcommand{\Der}{\mathrm{Der}}
\newcommand{\p}{{\partial}}
\newcommand{\Id}{{\mathrm{Id}}}

\newcommand{\bubu}{{_{\xy
 (0,0)*{_\bullet},
(2.5,0)*{_\bu},
 (-0.5,-0.05)*{}="a",
(2.8,-0.05)*{}="b",
\ar @{-} "a";"b" <0pt>
\endxy}}}
\newcommand{\wibu}{{_{\xy
 (0,0)*{_\bu},
(2.5,0)*{_\circ},
(-2.5,0)*{_\circ},
 (0,0)*{}="a",
(2.0,0)*{}="b",
(-2.25,0.)*{}="c",
\ar @{-} "a";"b" <0pt>
\ar @{-} "a";"c" <0pt>
\endxy}}}

 \newcommand{\Z}{{\mathbb Z}}
 \newcommand{\bS}{{\mathbb S}}

 \newcommand{\R}{{\mathbb R}}
 
 \newcommand{\K}{{\mathbb K}}

 \newcommand{\ot}{\otimes}

\newcommand{\sG}{{\mathsf G}}

\newcommand{\sP}{{\mathsf P}}

\newcommand{\Def}{\mathsf{Def}}
 \newcommand{\Beq}{\begin{equation}}
 \newcommand{\Eeq}{\end{equation}}
 \newcommand{\Beqr}{\begin{eqnarray}}
 \newcommand{\Eeqr}{\end{eqnarray}}
 \newcommand{\Beqrn}{\begin{eqnarray*}}
 \newcommand{\Eeqrn}{\end{eqnarray*}}
 \newcommand{\Ba}{\begin{array}}
 \newcommand{\Ea}{\end{array}}
 \newcommand{\Bi}{\begin{itemize}}
 \newcommand{\Ei}{\end{itemize}}
 \newcommand{\Bc}{\begin{center}}
 \newcommand{\Ec}{\end{center}}

 \newcommand{\fg}{{\mathfrak g}}

\newcommand{\fr}{{\mathfrak r}}
\newcommand{\fs}{{\mathfrak s}}

\newcommand{\ft}{{\mathfrak t}}

\newcommand{\fA}{{\mathfrak A}}

 \newcommand{\f}{{\mathcal O}}
 \newcommand{\cA}{{\mathcal A}}
 \newcommand{\cB}{{\mathcal B}}
 \newcommand{\cC}{{\mathcal C}}
 \newcommand{\caD}{{\mathcal D}}
 \newcommand{\cE}{{\mathcal E}}
 \newcommand{\cF}{{\mathcal F}}
 \newcommand{\cG}{{\mathcal G}}
 \newcommand{\caH}{{\mathcal H}}

 \newcommand{\caL}{{\mathcal L}}
 \newcommand{\cM}{{\mathcal M}}
 
 \newcommand{\cP}{{\mathcal P}}
 \newcommand{\cQ}{{\mathcal Q}}
 \newcommand{\cR}{{\mathcal R}}

 \newcommand{\cU}{{\mathcal U}}

 \newcommand{\al}{\alpha}
 \newcommand{\be}{\beta}
 \newcommand{\ga}{\gamma}
 
 \newcommand{\Ga}{\Gamma}

 \newcommand{\Hom}{{\mathrm H\mathrm o\mathrm m}}
 
 \newcommand{\sip}{\smallskip}
 \newcommand{\bip}{\bigskip}
 \newcommand{\mip}{\vspace{2.5mm}}

\newcommand{\GC}{\mathsf{GC}}

\newcommand{\grt}{\fg\fr\ft}
\newcommand{\CB}{\mathcal{C}\mathit{omb}}
\newcommand{\Assb}{\mathcal{A}\mathit{ssb}}
  \newcommand{\LB}{\mathcal{L}\mathit{ieb}}

\newcommand{\wLB}{\widehat{\mathcal{L}\mathit{ieb}}}

\theoremstyle{plain}
\swapnumbers
\newtheorem{theorem}{Theorem}[subsection]

\newtheorem{lemma}[theorem]{Lemma}

\newtheorem{prop-def}[theorem]{Proposition-definition}

\newtheorem{f-theorem}{Formality Theorem}[section]
\newtheorem{main-theorem}{Main~Theorem}[section]
\newtheorem{section-theorem}{Theorem}[section]

\theoremstyle{definition}

\begin{document}

 \sloppy

 \newenvironment{proo}{\begin{trivlist} \item{\sc {Proof.}}}
  {\hfill $\square$ \end{trivlist}}

\long\def\symbolfootnote[#1]#2{\begingroup%
\def\thefootnote{\fnsymbol{footnote}}\footnote[#1]{#2}\endgroup}

\title{Classification of universal formality maps\\
for quantizations of Lie bialgebras}


 \author{Sergei Merkulov}
\address{Sergei~Merkulov:  Mathematics Research Unit, University of Luxembourg,  Grand Duchy of Luxembourg}
\email{sergei.merkulov@uni.lu}

\author{Thomas~Willwacher}
\address{Thomas~Willwacher: Department of Mathematics, ETH Zurich, Zurich, Switzerland}
\email{thomas.willwacher@math.ethz.ch}

\thanks{S.M. has been partially supported by the Swedish Vetenskapr\aa det, grant 2012-5478. T.W. has been partially supported by the Swiss National Science foundation, grant 200021\_150012, and the SwissMAP NCCR funded by the Swiss National Science foundation. }

\begin{abstract} We settle several fundamental  questions about the theory of universal deformation quantization of Lie bialgebras by giving their complete classification up to homotopy equivalence. Moreover, we settle these questions in a greater generality --- we give a complete classification of the associated universal formality maps.

\sip

An important new technical ingredient introduced in this paper  is a {\em polydifferential}\, endofunctor $\caD$ in the category of augmented props with the property
that for any representation of a prop $\cP$ in a vector space $V$ the associated prop $\caD\cP$
admits an induced representation on the graded commutative algebra $\odot^\bu V$ given in terms of polydifferential operators.


\sip

Applying this functor to the minimal resolution $\wLB_\infty$ of the genus completed prop $\wLB$ of Lie bialgebras we show that universal formality
maps for quantizations of Lie bialgebras are in in 1-1 correspondence with morphisms of dg props
$$
F: \Assb_\infty \lon \caD\wLB_\infty
$$
satisfying certain boundary conditions,
where $\Assb_\infty$ is a minimal resolution
 of the prop of associative bialgebras. We prove that the set of such formality morphisms is non-empty. The latter result is used in turn to
give a short proof of the formality theorem for universal quantizations of arbitrary Lie bialgebras
 which says that for any Drinfeld associator $\fA$  there is an associated $\caL ie_\infty$ quasi-isomorphism between the $\caL ie_\infty$ algebras $\Def(\cA ss\cB_\infty\rar \cE nd_{\odot^\bu V})$ and $\Def(\caL ie\cB\rar \cE nd_V)$ controlling, respectively,  deformations of the standard
bialgebra structure in $\odot V$ and deformations of any given Lie bialgebra structure in $V$.

\sip

We study the deformation complex of an arbitrary universal formality morphism $\Def(\Assb_\infty \stackrel{F}{\rar} \caD\wLB_\infty)$
 and prove that it is quasi-isomorphic to the full (i.e.\ not necessary connected) version of the graph complex introduced  Maxim Kontsevich in the context of the theory of deformation quantizations of Poisson manifolds. This result gives a complete classification of the set  $\{F_\fA\}$  of gauge equivalence classes of universal Lie connected formality
maps --- it is a torsor over the Grothendieck-Teichm\"uller group $GRT=GRT_1\rtimes \K^*$ and can hence can be identified with the set $\{\fA\}$ of Drinfeld associators.

\sip

\sip

\end{abstract}
 \maketitle
\markboth{}{Formality theorem}

{\small
{\small
\tableofcontents
}
}

{\Large
\section{\bf Introduction}
}

\subsection{Open questions settled} The theory of universal deformation quantizations of Lie bialgebras started with the paper by Vladimir Drinfeld \cite{D} and played a very important role in several different areas of mathematics including the theory of quantum groups, the theory of Hopf algebras, the number theory, the theory of braid groups, the Poisson geometry, the theory of the Grothendieck-Teichm\"uller group and others. The first proof by Pavel Etingof and David Kazhdan \cite{EK} of the Main Theorem in this theory --- the existence theorem of universal quantizations of Lie bialgebras --- exhibited an important role which
Drinfeld's associators play in the story: every such an associator gives us an associated universal quantization. This leads to the questions: Do different associators give us really different (i.e.\ homotopy inequivalent) universal quantizations? Do Drinfeld associators give us {\em all}\, possible universal quantizations, or not? We settle these questions in the present paper by proving that the set of homotopy classes of universal quantizations of Lie bialgebras can be identified with the set of Drinfeld's associators. In particular, the Grothendieck-Teichm\"uller group acts transitively and faithfully on this set.
In fact, we settle these questions in a greater generality --- we give a complete classification of the associated universal formality maps (of which universal quantizations
of Lie bialgebras form a special case), and prove that the deformation complex of any such a formality map is quasi-isomorphic to the full version of the graph complex introduced  Maxim Kontsevich in the context of the theory of deformation quantizations of Poisson manifolds.

\subsection{Polydifferential functor}  We introduce an endofunctor $\caD$ in the category of (augmented) props with the property
that for any representation of a prop $\cP$ in a vector space $V$ the associated {\em polydifferential}\, prop $\caD\cP$
admits an induced representation on the graded commutative algebra $\f_V:=\odot^\bu V$ given in terms of polydifferential operators. For any $\cP$ the polydifferential prop $\caD\cP$ comes equipped with a canonical injection
$\cC omb \rar \caD \cP$ from the prop $\cC omb$ of commutative and cocommutative bialgebras.

\sip

 We apply this functor to the minimal resolution $\wLB_\infty$ of the genus completed version of the prop of Lie bialgebras $\LB$ introduced by Vladimir Drinfeld and obtain the following results.

\subsection{Main Theorem}\label{1: Main Theorem} (i) {\em For any choice of a Drinfeld associator $\fA$ there is an associated highly non-trivial\footnote{The precise meaning of what we mean by {\em highy non-trivial}\, is given in \S {\ref{5: Existence Theorem for formalities}} below. It says essentially that the morphism $F_\fA$ is non-zero on {\em each}\, member of the infinite family of generators of $\Assb_\infty$  (see formula (\ref{5: Boundary cond for formality map})).}
morphism of dg props,
\Beq\label{1: formality map F_A}
F_\fA: \Assb_\infty \lon \caD\wLB_\infty.
\Eeq
}
where $\Assb_\infty$ stands for a minimal resolution of the prop of associative bialgebras.

\sip
(ii) {\em  For any graded vector space $V$, each morphism
$F_\fA$ induces a $\caL ie_\infty$ quasi-isomorphism (called a {\em formality map})
between the dg $\caL ie_\infty$ algebra\footnote{As a complex this is precisely the Gerstenhaber-Schack complex \cite{GS} of $\f_V$.}
$$
\Def(\cA ssb_\infty\stackrel{\rho_0}{\lon} \cE nd_{\f_V})
$$
controlling deformations of the standard graded commutative and co-commutative bialgebra structure $\rho_0$ in $\f_V$, and the Lie algebra\footnote{Maurer-Cartan
elements of this Lie algebra are precisely strongly homotopy Lie bialgebra structures in $V$.}
$$
\Def(\LB\stackrel{0}{\lon} \cE nd_V)
$$
controlling deformations of the trivial morphism $0: \LB\rar \cE nd_V$. 
}
\mip

(iii) {\em Let\, $\Def\left(\Assb_\infty \stackrel{F_\fA}{\lon} \caD \wLB_\infty\right)$ be the deformation complex of any given  formality map $F_\fA$. Then there is a canonical morphism of complexes
 $$
 \mathsf{fGC}_3^\uparrow \lon \Def\left(\Assb_\infty \stackrel{F_\fA}{\lon} \caD \wLB_\infty\right)
 $$
 which is a quasi-isomorphism.
 Here  $\mathsf{fGC}_3^\uparrow$ is the
  oriented ``3-dimensional" incarnation (with the same cohomology \cite{Wi2})  of the full ``2-dimensional" Kontsevich graph complex introduced  in \cite{Ko-f}.
 \mip

 (iv)
 The set of homotopy classes of universal formality maps as in (\ref{1: formality map F_A}) can be identified with the set of Drinfeld associators. In particular,
the Grothendieck-Teichm\"uller group $GRT=GRT_1\rtimes \K^*$ acts faithfully and transitively on such
universal formality maps.}

\mip

In the proof of item (i) we used the Etingof-Kazhdan theorem \cite{EK} which says that any Lie bialgebra
can deformation quantized in the sense introduced by Drinfeld in \cite{D}. Later different proofs of this theorem have been given
by D.\ Tamarkin \cite{Ta}, P.\ Severa \cite{Se} and the authors \cite{MW3} (the latest proof in \cite{MW3} gives an explicit formula for a universal quantization of Lie bialgebras).
Item (ii) in the above Theorem is a Lie bialgebra analogue of the famous Kontsevich
formality theorem for deformation quantizations of Poisson manifolds (it was first proved in \cite{Me2} with the help of a more complicated method). In the proof of item (iii) we used an earlier result \cite{MW2} which established a quasi-isomorphism between the
deformation complex of the properad of Lie bialgebras and the oriented graph complex. In the proof of item
(iv) we used a computation \cite{Wi2} of the zero-th cohomology group of the graph complex $\mathsf{fGC}_3^\uparrow$,
$$
H^0(\mathsf{fGC}_3^\uparrow)=\grt
$$
where $\grt$ is the Lie algebra of the Grothendieck-Teichm\"uller group $GRT$ introduced by
Drinfeld in \cite{D2}.

\sip

Given any continuous morphism of topological dg props
$$
i: \wLB_\infty \lon \cP
$$
there is, by the Main Theorem, an associated {\em quantized}\, morphism given as a composition
$$
i_\fA: \Assb_\infty \stackrel{F_\fA}{\lon} \caD\wLB_\infty \stackrel{\caD(i)}{\lon}\caD \cP
$$
which depends in general on the choice of an associator $\fA$. An important example of such a non-trivial morphism $i$ was found recently in \cite{MW},
\Beq\label{1: from LB to RGra}
i: \wLB_\infty \lon \widehat{\cR\cG ra},
\Eeq
where $\cR\cG ra$ stands for the prop of {\em ribbon}\, graphs (see \cite{MW}) and  $\widehat{\cR\cG ra}$ for its genus competition; this morphism $i$ provides us with many new algebraic structures on the totality
$$
\prod_{g,n} H^\bu(\cM_{g,n})\ot_{\bS_n} \sgn_n
$$
of cohomology groups of moduli spaces of algebraic curves of genus $g$ and with $n$ (skewsymmetrized) punctures,
and establishes a surprisingly strong link between the latter and the cohomology theory of ordinary
(i.e.\ non-ribbon) graph complexes \cite{MW}. The quantized version of the morphism (\ref{1: from LB to RGra}) will be studied elsewhere.


\subsection{Some notation}
 The set $\{1,2, \ldots, n\}$ is abbreviated to $[n]$;  its group of automorphisms is
denoted by $\bS_n$;
the trivial one-dimensional representation of
 $\bS_n$ is denoted by $\id_n$, while its one dimensional sign representation is
 denoted by $\sgn_n$; we often abbreviate $\sgn_n^d:= \sgn_n^{\ot |d|}$.
The cardinality of a finite set $A$ is denoted by $\# A$.

\sip

We work throughout in the category of $\Z$-graded vector spaces over a field $\K$
of characteristic zero.
If $V=\oplus_{i\in \Z} V^i$ is a graded vector space, then
$V[k]$ stands for the graded vector space with $V[k]^i:=V^{i+k}$ and
and $s^k$ for the associated isomorphism $V\rar V[k]$; for $v\in V^i$ we set $|v|:=i$.
For a pair of graded vector spaces $V_1$ and $V_2$, the symbol $\Hom_i(V_1,V_2)$ stands
for the space of homogeneous linear maps of degree $i$, and
$\Hom(V_1,V_2):=\bigoplus_{i\in \Z}\Hom_i(V_1,V_2)$; for example, $s^k\in \Hom_{-k}(V,V[k])$.
Furthermore, we use the notation $\odot^k V$ for the k-fold symmetric product of the vector space $V$.
\sip

For a
prop(erad) $\cP$ we denote by $\cP\{k\}$ a prop(erad) which is uniquely defined by
 the following property:
for any graded vector space $V$ a representation
of $\cP\{k\}$ in $V$ is identical to a representation of  $\cP$ in $V[k]$.
 The degree shifted operad of Lie algebras $\caL \mathit{ie}\{d\}$  is denoted by $\caL ie_{d+1}$ while its minimal resolution by $\caH olie_{d+1}$; representations of $\caL ie_{d+1}$ are vector spaces equipped with Lie brackets of degree $-d$.

\sip

For a right (resp., left) module $V$ over a group $G$ we denote by $V_G$ (resp.\
$_G\hspace{-0.5mm}V$)
 the $\K$-vector space of coinvariants:
$V/\{g(v) - v\ |\ v\in V, g\in G\}$ and by $V^G$ (resp.\ $^GV$) the subspace
of invariants: $\{\forall g\in  G\ :\  g(v)=v,\ v\in V\}$. If $G$ is finite, then these
spaces are canonically isomorphic as $char(\K)=0$.

\bip

\bip

{\Large
\section{\bf A polydifferential endofunctor}
}

\mip

\subsection{Polydifferential operators.} For an arbitrary vector space $V$ the associated graded commutative tensor
algebra $\f_V:=\odot^\bu V$ has a canonical structure of a commutative and co-commutative bialgebra with the
multiplication denoted by central dot (or just by juxtaposition)
$$
\cdot: \f_V\ot \f_V \rar \f_V
$$
and the comultiplication denoted by
$$
\Delta: \f_V \lon \f_V\ot \f_V
$$
This comultiplication is uniquely determined by its value on an arbitrary element $x\in V$ which is given by
$$
\Delta(x)= 1\ot x+ x\ot 1.
$$
The bialgebra $\f_V$ is (co)augmented, with the (co)augmentation (co-)ideal denoted by $\bar{\f}_V:=\odot^{\bu \geq 1} V$. The reduced diagonal $\bar{\f}_V \rar \bar{\f}_V\ot \bar{\f}_V$ is denoted by $\bar{\Delta}$.  The multiplication in $\f_V$ induces naturally a multiplication in $\f_V^{\ot m}$ for any $m\geq 2$,
$$
\cdot: \f_V^{\ot m} \ot \f_V^{\ot m} \ot \f_V^{\ot m}
$$
which is also denoted by the central dot.

\sip

As $\f_V$ is a bialgebra, one can consider an associated Gerstenhaber-Schack complex
$$
 C_{GS}(\f_V^{\ot \bullet}, \f_V^{\ot \bullet})= \bigoplus_{m,n\geq 1} \Hom(\f_V^{\ot n}, \f_V^{\ot m})[m+n-2]
$$
equipped with the differential
\cite{GS},
$$
d_{\fg\fs}=d_1 \oplus d_2:  \Hom(\f_V^{\ot n}, \f_V^{\ot m}) \lon  \Hom(\f_V^{\ot n+1}, \f_V^{\ot m}) \oplus
\Hom(\f_V^{\ot n}, \f_V^{\ot m+1}),
$$
with $d_1$ given on an arbitrary $f\in \Hom(V^{\ot n}, V^{\ot m})$  by
\Beqrn
(d_1f)(v_0, v_1, \ldots, v_n)&:=&\Delta^{m-1}(v_0)\cdot f(v_1, v_2, \ldots, v_n) -
\sum_{i=0}^{n-1} (-1)^if(v_1, \dots, v_iv_{i+1}, \ldots, v_n)\\
&& + (-1)^{n+1}f(v_1, v_2, \ldots, v_n)\cdot\Delta^{m-1}(v_n)
\forall\ \ v_0, v_1,\ldots, v_n\in V,
\Eeqrn
where
$$
\Delta^{m-1}:= (\Delta\ot \Id^{\ot m-2 })\circ (\Delta\ot \Id^{\ot m-3})\circ \ldots \circ \Delta: V \rar V^{\ot m},
$$
for $m\geq 2$, we set $\Delta^0:=\Id$.
The expression for $d_2$ is an obvious ``dual" analogue of the one for $d_1$.

\mip

One can consider a subspace,
\Beq\label{2: C_poly in C_GS}
C_{poly}(\f_V^{\ot \bullet}, \f_V^{\ot \bullet}) \subset
 C_{GS}(\f_V^{\ot \bullet}, \f_V^{\ot \bullet}),
\Eeq
spanned by so called {\em polydifferential}\, operators,
$$
\Ba{rccc}
\Phi: & \f_V^{\ot m} &  \lon &  \f_V^{\ot n}\\
& f_1\ot \ldots \ot f_m & \lon & \Phi(f_1, \ldots, f_m),
\Ea
$$
which are linear combinations of operators of the form,
\Beq\label{polyoperator}
\Phi(f_1, \ldots, f_m)= 
x^{\be_{J_1}} \ot \ldots \ot x^{\be_{J_n}}\cdot \Delta^{n-1}\left( \frac{\p^{|\al_{I_1}|}f_1}{\p x^{\al_{I_1}}}\right)\cdot\ldots \cdot
 \Delta^{n-1}\left( \frac{\p^{|\al_{I_m}|}f_m}{\p x^{\al_{I_m}}}\right).
\Eeq
where $\{x^\al\}_{\al\in S}$ is some basis in $V$, $\al_I$ and $\be_{J}$ stand for multi-indices (say, for $\al_1\al_2\ldots \al_p$,
and $\be_1\be_2\ldots \be_q$) from the set $S$,
$$
x^{\be_J}:= x^{\be_1}x^{\be_2}\ldots x^{\be_q}, \ \ \ \ \ \
 \frac{\p^{|\al_I|}}{\p x^{\al_I}}:= \frac{\p^{p}}{\p x^{\al_1}\p x^{\al_2}\ldots \p x^{\al_p}}.
$$
In fact this subspace is a subcomplex of the Gerstenhaber-Shack complex with the differential given explicitly
by
\Beqrn
d_{\fg\fs}\Phi &=&
\sum_{i=1}^{n}(-1)^{i+1}
x^{\be_{J_1}} \ot  \mbox{...} \ot \bar{\Delta}(x^{\be_{J_i}})\ot   \mbox{...}  \ot x^{\be_{J_n}}\cdot \Delta^{n}
\left( \frac{\p^{|\al_{I_1}|}}{\p x^{\al_{I_1}}}\right)\cdot\ldots \cdot
 \Delta^{n}\left( \frac{\p^{|\al_{I_m}|}}{\p x^{\al_{I_m}}}\right)
\nonumber \\
&&
+\ 
\sum_{i=1}^{m}(-1)^{i+1}\sum_{\al_{I_i}=\al_{I_i'}\sqcup
\al_{I_i''}}
\nonumber \\
&&
\ \ \ \
x^{\be_{J_1}} \ot \mbox{...} \ot x^{\be_{J_m}}\cdot \Delta^{n-1}\left( \frac{\p^{|\al_{I_1}|}}{\p x^{\al_{I_1}}}\right)\cdot\mbox{...} \cdot
 \Delta^{n-1}\left( \frac{\p^{|\al_{I'_i}|}}{\p x^{\al_{I'_i}}}\right)\cdot  \Delta^{n-1}\left( \frac{\p^{|\al_{I_i''}|}}{\p x^{\al_{I''_i}}}\right)
\cdot\mbox{...}\cdot
 \Delta^{n-1}\left( \frac{\p^{|\al_{I_m}|}}{\p x^{\al_{I_m}}}\right),
\Eeqrn
where  the second summation goes over all
possible splittings, ${I_i=I_i'\sqcup I_i''}$, into  disjoint subsets.
Moreover, the inclusion (\ref{2: C_poly in C_GS}) is a quasi-isomorphism \cite{Me2}.
Note that polydifferential operators (\ref{polyoperator}) are allowed to have sets of multi-indices $I_i$ and $J_i$ with
zero cardinalities $|I_i|=0$ and/or $|J_i|=0$ (so that the standard multiplication and comultiplication in $\f_V$
belong to the class of polydifferential operators.

\subsection{The polydifferential prop associated to a prop} In this section we construct
an endofunctor in the category of (augmented) props which has the property
that for any prop $\cP=\{\cP(m,n)\}_{m,n\geq 1}$ and its representation
$$
\rho: \cP \lon \cE nd_V
$$
in a vector space $V$, the associated prop $\caD\cP=\{\caD\cP(k)\}_{k\geq 1}$
admits an associated representation,
$$
\rho^{poly}: \caD\cP \lon \cE nd_{\odot^\bu V}
$$
in the graded commutative algebra $\odot^\bu V$ on which elements $p\in\cP$ act
as  polydifferential operators.

\mip

The idea is simple, and is best expressed in some basis $\{x^\al\}_{\al\in S}$ in $V$
 (so that $\odot^\bu V\simeq \K[x^\al]$).
Any element $p\in \cP(m,n)$ gets transformed by $\rho$ into a linear map
$$
\Ba{rccc}
\rho(p): & \ot^n V & \lon & \ot^m V\\
         & x^{\al_1}\ot x^{\al_2} \ot\ldots \ot x^{\al_n} & \lon &
 \displaystyle \sum_{\be_1,\be_2,...,\be_m} A^{\al_1\al_2...\al_n}_{\be_1\be_2...,\be_m} x^{\be_1}\ot x^{\be_2} \ot\ldots \ot x^{\be_n}
\Ea
$$
for some $A^{\al_1\al_2...\al_n}_{\be_1\be_2...,\be_m}\in \K$.

\sip

Then, for any partitions
$$
[m]=J_1\sqcup \ldots  \sqcup J_l\ \ \ \ \  \mathrm{and}\ \ \ \ \
[n]=I_1\sqcup \ldots  \sqcup I_k
$$
 of the sets $[m]$ and  $[n]$ into a disjoint union of (possibly, not all non-empty) subsets,
$$
J_j=\{s_{a_1}, s_{a_2}, \ldots, s_{a_{\# J_j}}\},\ 1\leq j\leq l,\ \ \ \ \ \mbox{and}\ \ \ \ \
 I_i=\{s_{b_1}, s_{b_2}, \ldots, s_{b_{\# I_i}}\},\  1\leq i\leq k,
$$
 we can associate to $p\in \cP(m,n)$ a polydifferential operator
$$
\Ba{rccc}
p^{poly}: & \ot^k (\odot^\bu V) & \lon & \ot^l(\odot^\bu V)\\
         & f_1(x)\ot f_2(x) \ot\ldots \ot f_n(x) & \lon &
p^{poly}(f_1,\ldots,f_k)
\Ea
$$
where
$$
p^{poly}(f_1,\ldots,f_k):=
\displaystyle \sum_{\al_1,...,\al_n\atop
\be_1,..., \be_m}  A^{\al_{I_1}... \al_{I_k}}_{\be_{J_1}...,\be_{J_l}}\,
x^{\be_{J_1}} \ot \ldots \ot x^{\be_{J_l}}\cdot \Delta^{l-1}\left( \frac{\p^{|\al_{I_1}|}f_1}{\p x^{\al_{I_1}}}\right)\cdot\ldots \cdot
 \Delta^{l-1}\left( \frac{\p^{|\al_{I_k}|}f_k}{\p x^{\al_{I_k}}}\right).
$$
This association $p\rar p^{poly}$ (for any fixed partition of $[m]$ and $[n]$) is independent of the choice of a basis used in the construction. Our purpose is to construct a prop $\caD\cP$
 out of the prop $\cP$ together with a linear map
$$
F_{[n]=I_1\sqcup...\sqcup I_k}^{[m]=J_1\sqcup...\sqcup J_l}: \cP(m,n) \lon \f\cP(l,k)
$$
such that, for any representation $\rho:\cP\rar \cE nd_V$,  the prop $\caD\cP$ admits a natural representation $\rho^{poly}$ in
 $\odot^\bu V$ and
$$
p^{poly}= \rho^{poly}\left(F_{[n]=I_1\sqcup...\sqcup I_k}^{[m]=J_1\sqcup...\sqcup J_l}(p)\right).
$$
Note that $\odot^\bu V$ carries a natural representation of the prop $\CB$ controlling commutative and cocommutative bialgebras so that
the latter  can also be incorporated into  $\caD\cP$ in the form of operators corresponding, in the above notation,
to the case when all the sets $I_i$ and $J_j$ are empty.

\subsubsection{\bf Remark}  The functor
$\caD$ (see below) can be defined for an arbitrary unital prop. However, some of our constructions in this paper
become nicer if we assume that the props we consider are augmented,
$$
\cP= \K \oplus \bar{\cP},
$$
and apply the functor $\caD$ to the augmentation ideal $\bar{\cP}$ only\footnote{The restriction of the functor $\caD$ to the augmentation ideal comes, however, at the price of loosing  some {\em rescaling}\, operators in $\caD\cP$ which  can be  useful in applications.}.

\subsubsection{\bf Definition}\label{2: Definition of DP prop} Let $\cP$ be an augmented prop. Define a collection of (completed with respect to the filtration by the number of outputs and inputs) $\bS$-modules,
$$
\caD\cP(l,k) :=\CB(l,k)\ \oplus \prod_{m,n\geq 1}\bigoplus_{[m]=\ J_1\sqcup ...\sqcup J_l\atop
{[n]=\ I_1\sqcup ...\sqcup I_k \atop
\# I_i,..., \# J_j\geq 0}} \caD\cP^{J_1,...,J_l}_{I_1,...,I_k}
$$
where
\Beqrn
\CB(l,k)&:=&\id_l\ot \id_k\\
\caD\cP^{J_1,...,J_l}_{I_1,...,I_k} &:=&
   (\id_{S_{J_1}} \ot \ldots \ot \id_{S_{J_l}}) \ot_{S_{J_1}\times ...\times S_{J_l}}  \ot \bar{\cP}(m,n)\ot_{S_{I_1}\times ...\times S_{I_k}} (\id_{S_{I_1}} \ot \ldots \ot \id_{S_{I_k}})
\Eeqrn
where $\id_I$ stands for the trivial one-dimensional representation of the permutation group $S_I$.
Thus an element of the summand $\caD\cP^{J_1,...,J_l}_{I_1,...,I_k}$ is an element of $\cP(\# I_1+ \ldots + \# I_l,\# J_1+ \ldots + \# J_k)$ with  inputs (resp., outputs) belonging to each bunch $I_i$ (resp., $J_j$)
 symmetrized. We assume from now on that all legs in each bunch $I_i$ (resp., $J_j$) are labelled by the same integer $i$ (resp., $j$); this defines an action of the group $\bS_l^{op}\times \bS_k$ on $\caD\cP(l,k)$.
There is a canonical linear projection,
\Beq\label{2: projection from P to DP}
\pi^{J_1,...,J_l}_{I_1,...,I_k}: \bar{\cP}(m,n) \lon \caD\cP^{J_1,...,J_l}_{I_1,...,I_k}.
\Eeq

\mip

One can represent elements $p$ of the (non-unital) prop  $\bar{\cP}$ as (decorated) corollas,
$$
p\ \ \sim
\Ba{c}\resizebox{10mm}{!}{ \xy
(0,4.5)*+{...},
(0,-4.5)*+{...},
(0,0)*{\circ}="o",
(-5,5)*{}="1",
(-3,5)*{}="2",
(3,5)*{}="3",
(5,5)*{}="4",
(-3,-5)*{}="5",
(3,-5)*{}="6",
(5,-5)*{}="7",
(-5,-5)*{}="8",
(-5.5,7)*{_1},
(-3,7)*{_2},
(3,6)*{},
(5.9,7)*{m},
(-3,-7)*{_2},
(3,-7)*+{},
(5.9,-7)*{n},
(-5.5,-7)*{_1},
\ar @{-} "o";"1" <0pt>
\ar @{-} "o";"2" <0pt>
\ar @{-} "o";"3" <0pt>
\ar @{-} "o";"4" <0pt>
\ar @{-} "o";"5" <0pt>
\ar @{-} "o";"6" <0pt>
\ar @{-} "o";"7" <0pt>
\ar @{-} "o";"8" <0pt>
\endxy}\Ea \in \bar{\cP}(m,n)
$$
The image of $p$  under the projection (\ref{2: projection from P to DP}) is represented pictorially as the same  corolla but with output and input legs decorated by {\em not necessarily distinct}\, numbers,
$$
\pi^{J_1,...,J_l}_{I_1,...,I_k}(p) \sim
\Ba{c}\resizebox{15mm}{!}{\xy
(-9,-6)*{};
(0,0)*{\circ}
**\dir{-};
(-7.5,-6)*{};
(0,0)*{\circ }
**\dir{-};
(-6,-6)*{};
(0,0)*{\circ }
**\dir{-};
(-1,-6)*{};
(0,0)*{\circ }
**\dir{-};
(0,-6)*{};
(0,0)*{\circ }
**\dir{-};
(1,-6)*{};
(0,0)*{\circ }
**\dir{-};
(9,-6)*{};
(0,0)*{\circ }
**\dir{-};
(7.5,-6)*{};
(0,0)*{\circ }
**\dir{-};
(6,-6)*{};
(0,0)*{\circ }
**\dir{-};
(-3,-5)*{...};
(3,-5)*{...};
(-8,-9)*{\underbrace{\ \  }_{I_1}};
(0,-9)*{\underbrace{\ \  }_{I_i}};
(8,-9)*{\underbrace{\ \  }_{I_k}};
(-9,6)*{};
(0,0)*{\circ}
**\dir{-};
(-7.5,6)*{};
(0,0)*{\circ }
**\dir{-};
(-6,6)*{};
(0,0)*{\circ }
**\dir{-};
(-1,6)*{};
(0,0)*{\circ }
**\dir{-};
(0,6)*{};
(0,0)*{\circ }
**\dir{-};
(1,6)*{};
(0,0)*{\circ }
**\dir{-};
(9,6)*{};
(0,0)*{\circ }
**\dir{-};
(7.5,6)*{};
(0,0)*{\circ }
**\dir{-};
(6,6)*{};
(0,0)*{\circ }
**\dir{-};
(-3,5)*{...};
(3,5)*{...};
(-8,9)*{\overbrace{\ \  }^{J_1}};
(0,9)*{\overbrace{\ \  }^{J_j}};
(8,9)*{\overbrace{\ \  }^{J_l}};
\endxy}\Ea
\sim
\Ba{c}\resizebox{15mm}{!}{\xy
(-9,-6)*{};
(0,0)*{\circ }
**\dir{-};
(-7.5,-6)*{};
(0,0)*{\circ }
**\dir{-};
(-6,-6)*{};
(0,0)*{\circ }
**\dir{-};
(-1,-6)*{};
(0,0)*{\circ }
**\dir{-};
(0,-6)*{};
(0,0)*{\circ }
**\dir{-};
(1,-6)*{};
(0,0)*{\circ }
**\dir{-};
(9,-6)*{};
(0,0)*{\circ }
**\dir{-};
(7.5,-6)*{};
(0,0)*{\circ }
**\dir{-};
(6,-6)*{};
(0,0)*{\circ }
**\dir{-};
(-3,-5)*{...};
(3,-5)*{...};
(-9,-7.5)*{_1};
(-7.5,-7.5)*{_1};
(-6,-7.5)*{_1};
(-1.1,-7.5)*{_i};
(1.1,-7.5)*{_i};
(0,-7.5)*{_i};
(7.8,-7.5)*{_k};
(6.3,-7.5)*{_k};
(9.6,-7.5)*{_k};
(-9,6)*{};
(0,0)*{\circ }
**\dir{-};
(-7.5,6)*{};
(0,0)*{\circ }
**\dir{-};
(-6,6)*{};
(0,0)*{\circ }
**\dir{-};
(-1,6)*{};
(0,0)*{\circ }
**\dir{-};
(0,6)*{};
(0,0)*{\circ }
**\dir{-};
(1,6)*{};
(0,0)*{\circ }
**\dir{-};
(9,6)*{};
(0,0)*{\circ }
**\dir{-};
(7.5,6)*{};
(0,0)*{\circ }
**\dir{-};
(6,6)*{};
(0,0)*{\circ }
**\dir{-};
(-3,5)*{...};
(3,5)*{...};
(-9,7.5)*{^1};
(-7.5,7.5)*{^1};
(-6,7.5)*{^1};
(-1.1,7.5)*{^j};
(1.1,7.5)*{^j};
(0,7.5)*{^j};
(7.8,7.5)*{^l};
(6.3,7.5)*{^l};
(9.6,7.5)*{^l};
\endxy}\Ea \ \ \ \ \  1\leq i\leq k,\  1\leq j\leq l.
$$
Note that some of the sets $I_i$ and/or $J_j$ can be empty so that some of the numbers decorating inputs and/or outputs can have no
legs attached at all. For example, one and the same element
$$
q= \Ba{c}\resizebox{8mm}{!}{\xy
(0,0)*{\circ}="o",
(-4,5)*{}="1",
(-1.5,5)*{}="2",
(1.5,5)*{}="3",
(4,5)*{}="4",
(-3,-5)*{}="5",
(3,-5)*{}="6",
(5,-5)*{}="7",
(-5,-5)*{}="8",
(0,-5)*{}="9",
(-4.5,7)*{_1},
(-1.7,7)*{_2},
(1.7,7)*{_3},
(4.5,7)*{_4},
(-3,-7)*{_2},
(3.5,-7)*{_4},
(5.9,-7)*{_5},
(-5.5,-7)*{_1},
(0,-7)*{_3},
\ar @{-} "o";"1" <0pt>
\ar @{-} "o";"2" <0pt>
\ar @{-} "o";"3" <0pt>
\ar @{-} "o";"4" <0pt>
\ar @{-} "o";"5" <0pt>
\ar @{-} "o";"6" <0pt>
\ar @{-} "o";"7" <0pt>
\ar @{-} "o";"8" <0pt>
\ar @{-} "o";"9" <0pt>
\endxy}\Ea\in \bar{\cP}(4,5)
$$
can generate, for different partitions, several different elements in $\caD\cP$,
\Beq\label{2: examples of DP}
\Ba{c}\resizebox{8mm}{!}{\xy
(0,0)*{\circ}="o",
(-4,5)*{}="1",
(-1.5,5)*{}="2",
(1.5,5)*{}="3",
(4,5)*{}="4",
(-3,-5)*{}="5",
(3,-5)*{}="6",
(5,-5)*{}="7",
(-5,-5)*{}="8",
(0,-5)*{}="9",
(-4.5,6.9)*{_1},
(-1.7,6.9)*{_2},
(1.8,6.9)*{_1},
(4.6,6.9)*{_2},
(-3,-7)*{_1},
(3.5,-7)*{_1},
(5.9,-7)*{_2},
(-5.5,-7)*{_1},
(0,-7)*{_2},
\ar @{-} "o";"1" <0pt>
\ar @{-} "o";"2" <0pt>
\ar @{-} "o";"3" <0pt>
\ar @{-} "o";"4" <0pt>
\ar @{-} "o";"5" <0pt>
\ar @{-} "o";"6" <0pt>
\ar @{-} "o";"7" <0pt>
\ar @{-} "o";"8" <0pt>
\ar @{-} "o";"9" <0pt>
\endxy}\Ea
 \in \caD\cP(2,2) \ \ \ \ , \ \ \ \
\Ba{c}\resizebox{14mm}{!}{\xy
(0,0)*{\circ}="o",
(-4,5)*{}="1",
(-1.5,5)*{}="2",
(1.5,5)*{}="3",
(4,5)*{}="4",
(-3,-5)*{}="5",
(3,-5)*{}="6",
(5,-5)*{}="7",
(-5,-5)*{}="8",
(0,-5)*{}="9",
(-4.5,6.8)*{_1},
(-1.7,6.8)*{_1},
(1.8,6.8)*{_1},
(4.6,6.8)*{_1},
(-3,-7)*{_1},
(3.5,-7)*{_1},
(5.9,-7)*{_2},
(-5.5,-7)*{_1},
(0,-7)*{_2},
(8,6.8)*{_2},
(10,-7)*{_3},
(12,-7)*{_4},
\ar @{-} "o";"1" <0pt>
\ar @{-} "o";"2" <0pt>
\ar @{-} "o";"3" <0pt>
\ar @{-} "o";"4" <0pt>
\ar @{-} "o";"5" <0pt>
\ar @{-} "o";"6" <0pt>
\ar @{-} "o";"7" <0pt>
\ar @{-} "o";"8" <0pt>
\ar @{-} "o";"9" <0pt>
\endxy}\Ea \in \f\cP(2,4)\ \ , \ \ etc.
\Eeq
We shall often represent elements of $\caD\cP$
as graphs having two types of vertices:  the small ones which are decorated by  elements of $\bar{\cP}$, and new big ones corresponding to those input, respectively, output legs of $\caD\cP$ which having the same  numerical labels,
$$
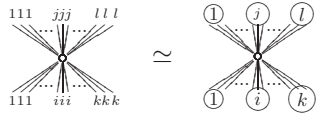
\Ba{c}\resizebox{15mm}{!}{\xy
(-9,-6)*{};
(0,0)*{\circ }
**\dir{-};
(-7.5,-6)*{};
(0,0)*{\circ }
**\dir{-};
(-6,-6)*{};
(0,0)*{\circ }
**\dir{-};
(-1,-6)*{};
(0,0)*{\circ }
**\dir{-};
(0,-6)*{};
(0,0)*{\circ }
**\dir{-};
(1,-6)*{};
(0,0)*{\circ }
**\dir{-};
(9,-6)*{};
(0,0)*{\circ }
**\dir{-};
(7.5,-6)*{};
(0,0)*{\circ }
**\dir{-};
(6,-6)*{};
(0,0)*{\circ }
**\dir{-};
(-3,-5)*{...};
(3,-5)*{...};
(-9,-7.5)*{_1};
(-7.5,-7.5)*{_1};
(-6,-7.5)*{_1};
(-1.1,-7.5)*{_i};
(1.1,-7.5)*{_i};
(0,-7.5)*{_i};
(7.8,-7.5)*{_k};
(6.3,-7.5)*{_k};
(9.6,-7.5)*{_k};
(-9,6)*{};
(0,0)*{\circ }
**\dir{-};
(-7.5,6)*{};
(0,0)*{\circ }
**\dir{-};
(-6,6)*{};
(0,0)*{\circ }
**\dir{-};
(-1,6)*{};
(0,0)*{\circ }
**\dir{-};
(0,6)*{};
(0,0)*{\circ }
**\dir{-};
(1,6)*{};
(0,0)*{\circ }
**\dir{-};
(9,6)*{};
(0,0)*{\circ }
**\dir{-};
(7.5,6)*{};
(0,0)*{\circ }
**\dir{-};
(6,6)*{};
(0,0)*{\circ }
**\dir{-};
(-3,5)*{...};
(3,5)*{...};
(-9,7.5)*{^1};
(-7.5,7.5)*{^1};
(-6,7.5)*{^1};
(-1.1,7.5)*{^j};
(1.1,7.5)*{^j};
(0,7.5)*{^j};
(7.8,7.5)*{^l};
(6.3,7.5)*{^l};
(9.6,7.5)*{^l};
\endxy}\Ea
\ \  {\simeq}
 \ \
\Ba{c}\resizebox{15mm}{!}{\xy
(-9,-6)*{};
(0,0)*{\circ }
**\dir{-};
(-7.5,-6)*{};
(0,0)*{\circ }
**\dir{-};
(-6,-6)*{};
(0,0)*{\circ }
**\dir{-};
(-1,-6)*{};
(0,0)*{\circ }
**\dir{-};
(0,-6)*{};
(0,0)*{\circ }
**\dir{-};
(1,-6)*{};
(0,0)*{\circ }
**\dir{-};
(9,-6)*{};
(0,0)*{\circ }
**\dir{-};
(7.5,-6)*{};
(0,0)*{\circ }
**\dir{-};
(6,-6)*{};
(0,0)*{\circ }
**\dir{-};
(-3,-5)*{...};
(3,-5)*{...};
(-9,6)*{};
(0,0)*{\circ }
**\dir{-};
(-7.5,6)*{};
(0,0)*{\circ }
**\dir{-};
(-6,6)*{};
(0,0)*{\circ }
**\dir{-};
(-1,6)*{};
(0,0)*{\circ }
**\dir{-};
(0,6)*{};
(0,0)*{\circ }
**\dir{-};
(1,6)*{};
(0,0)*{\circ }
**\dir{-};
(9,6)*{};
(0,0)*{\circ }
**\dir{-};
(7.5,6)*{};
(0,0)*{\circ }
**\dir{-};
(6,6)*{};
(0,0)*{\circ }
**\dir{-};
(-3,5)*{...};
(3,5)*{...};
(-8.2,-7.9)*+\hbox{${{1}}$}*\frm{o};
(8.2,-7.9)*+\hbox{${{\, k\, }}$}*\frm{o};
(0,-7.9)*+\hbox{${{\, i\, }}$}*\frm{o};
(-8.2,7.9)*+\hbox{${{1}}$}*\frm{o};
(8.2,7.9)*+\hbox{${{\, l\, }}$}*\frm{o};
(0,7.9)*+\hbox{${{\, _j\, }}$}*\frm{o}
\endxy}\Ea
$$
In this notation elements (\ref{2: examples of DP}) gets represented, respectively, as
$$
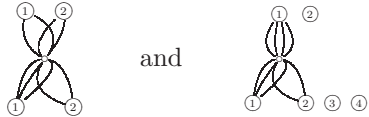
\Ba{c}\resizebox{10mm}{!}{\xy
(-2.5,5)*{}="1",
(-0.8,5)*{}="2",
(0.8,5)*{}="3",
(2.5,5)*{}="4",
 (0,0)*{\circ}="A";
 (-6,-10)*+{_1}*\frm{o}="B";
  (6,-10)*+{_2}*\frm{o}="C";
  (-4,10)*+{_1}*\frm{o}="X";
  (4,10)*+{_2}*\frm{o}="Y";
 "A"; "B" **\crv{(-5,-0)}; 
 "A"; "B" **\crv{(-5,-6)};
  "A"; "C" **\crv{(5,-0.5)};
   "A"; "B" **\crv{(5,-1)};
  "A"; "C" **\crv{(-5,-7)};
  "A"; "X" **\crv{(-6,3)}; 
 "A"; "X" **\crv{(5,6)};
  "A"; "Y" **\crv{(-4,3)}; 
 "A"; "Y" **\crv{(5,6)};
 \endxy}
 \Ea
 \ \ \ \ \ \mbox{and} \ \ \ \ \
\Ba{c}\resizebox{17mm}{!}{ \xy
(-2.5,5)*{}="1",
(-0.8,5)*{}="2",
(0.8,5)*{}="3",
(2.5,5)*{}="4",
 (0,0)*{\circ}="A";
 (-6,-10)*+{_1}*\frm{o}="B";
  (6,-10)*+{_2}*\frm{o}="C";
  (0,10)*+{_1}*\frm{o}="X";
  (7,10)*+{_2}*\frm{o};
   (12,-10)*+{_3}*\frm{o};
   (18,-10)*+{_4}*\frm{o};
 "A"; "B" **\crv{(-5,-0)}; 
 "A"; "B" **\crv{(-5,-6)};
  "A"; "C" **\crv{(5,-0.5)};
   "A"; "B" **\crv{(5,-1)};
  "A"; "C" **\crv{(-5,-7)};
"A"; "X" **\crv{(-6,3)}; 
 "A"; "X" **\crv{(-2,3)};
  "A"; "X" **\crv{(2,3)};
   "A"; "X" **\crv{(6,3)};
 \endxy}
 \Ea
$$
while the unique generator of $\CB(k,l)\subset \caD\cP(k,l)$ as a two leveled collection of solely white vertices
\Beq\label{2: graphs for Comb(l,k)}
\Ba{c}\resizebox{15mm}{!}{
\xy
(0,5)*+{_1}*\frm{o};
(6,5)*+{_2}*\frm{o};
(18,5)*+{_l}*\frm{o};
(13,5)*+{\cdots};
(0,-5)*+{_1}*\frm{o};
(6,-5)*+{_2}*\frm{o};
(18,-5)*+{_k}*\frm{o};
(13,-5)*+{\cdots};
\endxy}\Ea.
\Eeq
In particular the multiplication generator in $\CB$ is given by the graph
$\Ba{c}\resizebox{6mm}{!}{
\xy
(0,3)*+{_1}*\frm{o};
(-3,-3)*+{_1}*\frm{o};
(3,-3)*+{_2}*\frm{o};
\endxy}\Ea$ while the comultiplication generator by $\Ba{c}\resizebox{6mm}{!}{
\xy
(0,-3)*+{_1}*\frm{o};
(-3,3)*+{_1}*\frm{o};
(3,3)*+{_2}*\frm{o};
\endxy}\Ea$. These new big white vertices come in two types --- the incoming ones and outgoing ones;
we call them from now on {\em in-vertices} and, respectively, {\em out-vertices} of al element $p\in \caD\cP(l,k)$.

The $\bS$-bimodule $\caD\cP=\{\ \caD\cP(k,l)\}$ has a natural structure of a prop with respect to the
following composition operations:
\sip

(I) The {\em horizontal composition}
$$
\Ba{rccc}
\Box_h: &  \caD\cP(k,l)  \ot\caD\cP(k',l')
 &\lon & \caD\cP(k+k, l+l')\\
 & p\ot q & \lon &p \Box_h q
\Ea
$$
is given  as follows:
\Bi
\item[(i)] if $p\in \caD\cP^{J_1,...,J_l}_{I_1,...,I_k}$ and $q\in \caD\cP^{J_1',...,J'_{l'}}_{I'_1,...,I'_{k'}}$ the $\bS$-equivariance of the horizontal composition $\circ_h$ in $\cP$ assures that
$$
p\circ_h q\in \caD\cP^{J_1,...,J_l, J_1',...,J'_{l'}}_{I_1,...,I_k,I'_1,...,I'_{k'} }
$$
and we set $p\Box_h q:=p\circ_h q$.

\item[(ii)] if $p\in \caD\cP^{J_1,...,J_l}_{I_1,...,I_k}\subset \bar{P}(m,n)$ and $q$ is the generator $\Ba{c}\resizebox{11mm}{!}{
\xy
(0,5)*+{_1}*\frm{o};
(6,5)*+{_2}*\frm{o};
(18,5)*+{_{l'}}*\frm{o};
(13,5)*+{\cdots};
(0,-5)*+{_1}*\frm{o};
(6,-5)*+{_2}*\frm{o};
(18,-5)*+{_{k'}}*\frm{o};
(13,-5)*+{\cdots};
\endxy}\Ea$ of $\CB(l',k')$, then we define
$$
p\Box_h q:= \pi^{J_1,...,J_l, J_1',...,J'_{l'}}_{I_1,...,I_k,I'_1,...,I'_{k'} }(p) \in \caD\cP^{J_1,...,J_l, J_1',...,J'_{l'}}_{I_1,...,I_k,I'_1,...,I'_{k'} }
$$
where $J_1'=...=J'_{l'}=I'_1=...=I'_{k'}=\emptyset$. Similarly one defines $q\Box_h p$.
\item[(iii)] if $p\in \CB(l,k)$ and $q\in \CB(l',k')$ then $p\Box_h q$ is defined to be the horizontal composition of these elements in the prop $\CB$.

\Ei

\sip

(II) The {\em vertical composition}
$$
\Ba{rccc}
\Box_v: & \caD\cP(l,n) \ot \caD\cP(n,k) &\lon & \caD\cP(l,k)\\
 & \Ga\ot \Ga' & \lon & \Ga\Box_v\Ga',
\Ea
$$
by the following three step procedure:
 \Bi
 \item[(1)] erase all $n$ out-vertices of $\Ga'$ and all $n$ in-vertices of $\Ga$; this procedure
 leave  output legs of $\Ga'$ (which are labelled by elements of the set $[n]$) and input legs
 of $\Ga$ (which are also labelled by elements of the same set $[n]$) ``hanging in the air";

 \item[(2)] take the sum over all possible ways of
 \Bi
 \item[(2i)]
 attaching a number of the hanging out-legs of $\Ga'$ to the same number of hanging in-legs of $\Ga$ {\em with the same numerical label},
 \item[(2ii)] attaching the remaining hanging edges of $\Ga'$ to the out-vertices of $\Ga$,
 \item[(2iii)] attaching the remaining (after step (2i)) hanging in-legs of $\Ga$ to the in-vertices of $\Ga'$;
\Ei
 \item[(3)] perform the vertical composition (along new internal edges in step (2i)) and horizontal composition (of all remaining disjoint internal vertices) a in $\cP$ to get a unique internal vertex in $\Ga\Box_v\Ga'$.
\Ei

Often props $\cP$ are defined as quotients of  free props $\cF ree \langle E\rangle$ generated by some $\bS$-bimodules $E$ modulo  ideals $\cR$ of relations; in this representation elements of $\cP$ (and hence internal vertices of elements from $\caD\cP$) are themselves graphs (modulo some equivalence relation). Therefore it is enough to understand the above compositions $\Box_h$ and $\Box_v$ for props of the form $\cP=\cF ree\langle E\rangle$ in which case one has, for example,
$$
\Ba{c}\resizebox{10mm}{!}{\xy
(-2.5,5)*{}="1",
(-0.8,5)*{}="2",
(0.8,5)*{}="3",
(2.5,5)*{}="4",
 (0,0)*{\circ}="A";
 (-6,-10)*+{_1}*\frm{o}="B";
  (6,-10)*+{_2}*\frm{o}="C";
  (-4,10)*+{_1}*\frm{o}="X";
  (4,10)*+{_2}*\frm{o}="Y";
 "A"; "B" **\crv{(-5,-0)}; 
 "A"; "B" **\crv{(-5,-6)};
  "A"; "C" **\crv{(5,-0.5)};
   "A"; "B" **\crv{(5,-1)};
  "A"; "C" **\crv{(-5,-7)};
  "A"; "X" **\crv{(-6,3)}; 
 "A"; "X" **\crv{(5,6)};
  "A"; "Y" **\crv{(-4,3)}; 
 "A"; "Y" **\crv{(5,6)};
 \endxy}
 \Ea
 \ \ \ \Box_h \
\Ba{c}\resizebox{17mm}{!}{ \xy
(-2.5,5)*{}="1",
(-0.8,5)*{}="2",
(0.8,5)*{}="3",
(2.5,5)*{}="4",
 (0,0)*{\circ}="A";
 (-6,-10)*+{_1}*\frm{o}="B";
  (6,-10)*+{_2}*\frm{o}="C";
  (0,10)*+{_1}*\frm{o}="X";
  (7,10)*+{_2}*\frm{o};
   (12,-10)*+{_3}*\frm{o};
   (18,-10)*+{_4}*\frm{o};
 "A"; "B" **\crv{(-5,-0)}; 
 "A"; "B" **\crv{(-5,-6)};
  "A"; "C" **\crv{(5,-0.5)};
   "A"; "B" **\crv{(5,-1)};
  "A"; "C" **\crv{(-5,-7)};
"A"; "X" **\crv{(-6,3)}; 
 "A"; "X" **\crv{(-2,3)};
  "A"; "X" **\crv{(2,3)};
   "A"; "X" **\crv{(6,3)};
 \endxy}
 \Ea
 = \Ba{c}\resizebox{10mm}{!}{\xy
(-2.5,5)*{}="1",
(-0.8,5)*{}="2",
(0.8,5)*{}="3",
(2.5,5)*{}="4",
 (0,0)*{\circ}="A";
 (-6,-10)*+{_1}*\frm{o}="B";
  (6,-10)*+{_2}*\frm{o}="C";
  (-4,10)*+{_1}*\frm{o}="X";
  (4,10)*+{_2}*\frm{o}="Y";
 "A"; "B" **\crv{(-5,-0)}; 
 "A"; "B" **\crv{(-5,-6)};
  "A"; "C" **\crv{(5,-0.5)};
   "A"; "B" **\crv{(5,-1)};
  "A"; "C" **\crv{(-5,-7)};
  "A"; "X" **\crv{(-6,3)}; 
 "A"; "X" **\crv{(5,6)};
  "A"; "Y" **\crv{(-4,3)}; 
 "A"; "Y" **\crv{(5,6)};
 \endxy}
 \Ea
\Ba{c}\resizebox{17mm}{!}{ \xy
(-2.5,5)*{}="1",
(-0.8,5)*{}="2",
(0.8,5)*{}="3",
(2.5,5)*{}="4",
 (0,0)*{\circ}="A";
 (-6,-10)*+{_3}*\frm{o}="B";
  (6,-10)*+{_4}*\frm{o}="C";
  (0,10)*+{_3}*\frm{o}="X";
  (7,10)*+{_4}*\frm{o};
   (12,-10)*+{_5}*\frm{o};
   (18,-10)*+{_5}*\frm{o};
 "A"; "B" **\crv{(-5,-0)}; 
 "A"; "B" **\crv{(-5,-6)};
  "A"; "C" **\crv{(5,-0.5)};
   "A"; "B" **\crv{(5,-1)};
  "A"; "C" **\crv{(-5,-7)};
"A"; "X" **\crv{(-6,3)}; 
 "A"; "X" **\crv{(-2,3)};
  "A"; "X" **\crv{(2,3)};
   "A"; "X" **\crv{(6,3)};
 \endxy}
 \Ea
$$
and
$$
\Ba{c}
\resizebox{9mm}{!}{\xy
 (0,8)*{\circ}="a",
 (-5,13)*+{_1}*\frm{o}="b_1";
(5,13)*+{_2}*\frm{o}="b_2",
(0,2)*+{_1}*\frm{o}="0",
\ar @{-} "a";"0" <0pt>
\ar @{-} "a";"b_1" <0pt>
\ar @{-} "a";"b_2" <0pt>
\endxy} \Ea
 \Box_v
\Ba{c}\resizebox{9mm}{!}{ \xy
 (0,7)*{\circ}="a",
  (-5,2)*+{_1}*\frm{o}="b_1";
(5,2)*+{_2}*\frm{o}="b_2",
(0,13)*+{_1}*\frm{o}="0",
\ar @{-} "a";"0" <0pt>
\ar @{-} "a";"b_1" <0pt>
\ar @{-} "a";"b_2" <0pt>
\endxy}\Ea
\ \ = \ \
\Ba{c}
\resizebox{8.8mm}{!}{\xy
(0,13)*{\circ}="0",
 (0,7)*{\circ}="a",
 (-5,2)*+{_1}*\frm{o}="b_1";
(5,2)*+{_2}*\frm{o}="b_2",
(-5,18)*+{_1}*\frm{o}="u_1",
(5,18)*+{_2}*\frm{o}="u_2",
\ar @{-} "a";"0" <0pt>
\ar @{-} "a";"b_1" <0pt>
\ar @{-} "a";"b_2" <0pt>
\ar @{-} "0";"u_1" <0pt>
\ar @{-} "0";"u_2" <0pt>
\endxy}
\Ea\ \ \ +\ \ \
\Ba{c}
\resizebox{8.4mm}{!}{\xy
(4,10)*{\circ}="0",
 (-4,8)*{\circ}="a",
(-5,2)*+{_1}*\frm{o}="b_1";
(5,2)*+{_2}*\frm{o}="b_2",
(-5,18)*+{_1}*\frm{o}="u_1",
(5,18)*+{_2}*\frm{o}="u_2",
\ar @{-} "b_1";"0" <0pt>
\ar @{-} "u_1";"a" <0pt>
\ar @{-} "a";"b_1" <0pt>
\ar @{-} "a";"b_2" <0pt>
\ar @{-} "0";"u_1" <0pt>
\ar @{-} "0";"u_2" <0pt>
\endxy}
\Ea\ \ \ +\ \ \
\Ba{c}
\resizebox{8.4mm}{!}{\xy
(4,10)*{\circ}="0",
 (-4,8)*{\circ}="a",
  (-5,2)*+{_1}*\frm{o}="b_1";
(5,2)*+{_2}*\frm{o}="b_2",
(-5,18)*+{_1}*\frm{o}="u_1",
(5,18)*+{_2}*\frm{o}="u_2",
\ar @{-} "b_2";"0" <0pt>
\ar @{-} "u_1";"a" <0pt>
\ar @{-} "a";"b_1" <0pt>
\ar @{-} "a";"b_2" <0pt>
\ar @{-} "0";"u_1" <0pt>
\ar @{-} "0";"u_2" <0pt>
\endxy}
\Ea\ \ \ +\ \ \
\Ba{c}
\resizebox{8.4mm}{!}{\xy
(4,10)*{\circ}="0",
 (-4,8)*{\circ}="a",
 (-5,2)*+{_1}*\frm{o}="b_1";
(5,2)*+{_2}*\frm{o}="b_2",
(-5,18)*+{_1}*\frm{o}="u_1",
(5,18)*+{_2}*\frm{o}="u_2",
\ar @{-} "b_2";"0" <0pt>
\ar @{-} "u_2";"a" <0pt>
\ar @{-} "a";"b_1" <0pt>
\ar @{-} "a";"b_2" <0pt>
\ar @{-} "0";"u_1" <0pt>
\ar @{-} "0";"u_2" <0pt>
\endxy}
\Ea\ \ \ +\ \ \
\Ba{c}
\resizebox{8.4mm}{!}{\xy
(4,10)*{\circ}="0",
 (-4,8)*{\circ}="a",
 (-5,2)*+{_1}*\frm{o}="b_1";
(5,2)*+{_2}*\frm{o}="b_2",
(-5,18)*+{_1}*\frm{o}="u_1",
(5,18)*+{_2}*\frm{o}="u_2",
\ar @{-} "b_1";"0" <0pt>
\ar @{-} "u_2";"a" <0pt>
\ar @{-} "a";"b_1" <0pt>
\ar @{-} "a";"b_2" <0pt>
\ar @{-} "0";"u_1" <0pt>
\ar @{-} "0";"u_2" <0pt>
\endxy}\Ea
$$
It is worth emphasizing that all graphs above are {\em oriented}\, in the sense that every edge
is directed (but there are now directed paths of directed edge forming a closed path). We do not show directions in our pictures assuming that the flow goes by default from bottom to the top. For example,
the graph $\Ba{c}\resizebox{8.4mm}{!}{\xy
(4,10)*{\circ}="0",
 (-4,8)*{\circ}="a",
 (-5,2)*+{_1}*\frm{o}="b_1";
(5,2)*+{_2}*\frm{o}="b_2",
(-5,18)*+{_1}*\frm{o}="u_1",
(5,18)*+{_2}*\frm{o}="u_2",
\ar @{-} "b_1";"0" <0pt>
\ar @{-} "u_2";"a" <0pt>
\ar @{-} "a";"b_1" <0pt>
\ar @{-} "a";"b_2" <0pt>
\ar @{-} "0";"u_1" <0pt>
\ar @{-} "0";"u_2" <0pt>
\endxy}\Ea$ stands in fact for the directed graph $\Ba{c}\resizebox{8.4mm}{!}{\xy
(4,10)*{\circ}="0",
 (-4,8)*{\circ}="a",
 (-5,2)*+{_1}*\frm{o}="b_1";
(5,2)*+{_2}*\frm{o}="b_2",
(-5,18)*+{_1}*\frm{o}="u_1",
(5,18)*+{_2}*\frm{o}="u_2",
\ar @{->} "b_1";"0" <0pt>
\ar @{<-} "u_2";"a" <0pt>
\ar @{<-} "a";"b_1" <0pt>
\ar @{<-} "a";"b_2" <0pt>
\ar @{->} "0";"u_1" <0pt>
\ar @{->} "0";"u_2" <0pt>
\endxy}\Ea$.

Let us summarize the above discussion in the following proposition-definition.

\bip

\subsubsection{\bf Proposition-definition}\label{2: Proposition on DP and its representations} {\em (i) There is an endofunctor $\caD$ in the category of augmented props. The value $\caD\cP$ of the endofunctor $\caD$
on an augmented prop $\cP$
is called the  {\em polydifferential prop associated to $\cP$}.

\mip

(ii) Any representation of the augmented prop $\cP$ in a graded vector space $V$ induces canonically an associated
representation of the polydifferential $\caD\cP$ in $\odot^\bu V$ given by polydifferential operators.

\sip

(iii) For any augmented prop $\cP$ there is a canonical injection of props,
\Beq\label{2: from Commb to DP}
i:\CB \lon \caD\cP
\Eeq
sending the generator $1$ of {\em $\CB(l,k)=\id_l\ot \id_k=\K$} to the graph (\ref{2: graphs for Comb(l,k)}).
}

\subsubsection{\bf Remark} The suboperad $\{\caD\cP(1,n)\}_{n\geq 1}$ of the prop $\caD\cP$ has been studied earlier in \cite{MW} where it was denoted by $\f\cP$.

\subsubsection{\bf Proposition} {\em The endofunctor $\caD$ is exact}.

\begin{proof} Due to the classical Maschke's theorem, the functor of taking invariants and/or coinvariants in the category of representations
 of finite groups in dg vector spaces over a field of characteristic zero is exact. Hence the required claim.
\end{proof}

\subsection{Polydifferential props and hypergraphs} In the definition of the polydifferential prop $\caD\cP$ horizontal compositions in a prop $\cP$  play as important role as vertical ones. Therefore,
to apply apply the polydifferential functor to a {\em properad}\, (or operad) $\cP$
one has to take first its enveloping prop $\cU\cP$ and then apply the endofunctor $\caD$ to the latter; for an augmented properad $\cP$ we understand the prop $\cU\cP$ as $\K \oplus \cU\bar{\cP}$ and define
$$
\caD\cP:=\caD(\cU{\cP}).
$$
The case of properads is of special interest as elements of $\caD\cP$ can be understood now as {\em hypergraphs}, that is, generalizations of graphs in which edges can connect more than two vertices
(cf.\ \cite{MW}).
For example, elements of a properad $\cP$
$$
p=
\Ba{c}\resizebox{8mm}{!}{ \xy
(0,0)*{\circ}="o",
(-1.5,5)*{}="2",
(1.5,5)*{}="3",
(-1.5,-5)*{}="5",
(1.5,-5)*{}="6",
(4,-5)*{}="7",
(-4,-5)*{}="8",
(-1.5,7)*{_1},
(1.5,7)*{_2},
(-1.5,-7)*{_2},
(1.5,-7)*{_3},
(4,-7)*{_4},
(-4,-7)*{_1},
\ar @{-} "o";"2" <0pt>
\ar @{-} "o";"3" <0pt>
\ar @{-} "o";"5" <0pt>
\ar @{-} "o";"6" <0pt>
\ar @{-} "o";"7" <0pt>
\ar @{-} "o";"8" <0pt>
\endxy}\Ea\in \cP(2,4) \ \ \ \mbox{and} \ \ \ \
q=\Ba{c}\resizebox{7mm}{!}{  \xy
(0,0)*{\circ}="o",
(0,5)*{}="1",
(-3,-5)*{}="2",
(3,-5)*{}="3",
(0,-5)*{}="4",
(0,7)*{_1},
(3.5,-7)*{_3},
(-3.5,-7)*{_1},
(0,-7)*{_2},
\ar @{-} "o";"1" <0pt>
\ar @{-} "o";"2" <0pt>
\ar @{-} "o";"3" <0pt>
\ar @{-} "o";"4" <0pt>
\endxy}\Ea\in \cP(1,3)
$$
generate  an element
$$
\Ba{c}\resizebox{8mm}{!}{  \xy
(-2,0)*{^p};
(0,0)*{\circ}="o",
(-2,5)*{}="2",
(2,5)*{}="3",
(-1.5,-5)*{}="5",
(1.5,-5)*{}="6",
(4,-5)*{}="7",
(-4,-5)*{}="8",
(-2,7)*{_1},
(2,7)*{_2},
(-1.5,-7)*{_2},
(1.5,-7)*{_3},
(4,-7)*{_4},
(-4,-7)*{_1},
\ar @{-} "o";"2" <0pt>
\ar @{-} "o";"3" <0pt>
\ar @{-} "o";"5" <0pt>
\ar @{-} "o";"6" <0pt>
\ar @{-} "o";"7" <0pt>
\ar @{-} "o";"8" <0pt>
\endxy}\Ea
\Ba{c}\resizebox{7.3mm}{!}{ \xy
(2,0)*{^q};
(0,0)*{\circ}="o",
(0,5)*{}="1",
(-3,-5)*{}="2",
(3,-5)*{}="3",
(0,-5)*{}="4",
(0,7)*{_3},
(3.5,-7)*{_7},
(-3.5,-7)*{_5},
(0,-7)*{_6},
\ar @{-} "o";"1" <0pt>
\ar @{-} "o";"2" <0pt>
\ar @{-} "o";"3" <0pt>
\ar @{-} "o";"4" <0pt>
\endxy}\Ea\ \in \cU\cP(3,7)
$$
which in turn generates  an element
$$
\Ba{c}\resizebox{15mm}{!}{  \xy
(-2,0.8)*{^p};
(11,0.8)*{^q};
(-1.5,8)*+{_1}*\frm{o}="1",
(1.5,8)*{}="2",
(5,8)*+{_2}*\frm{o}="3",
 (0,0)*{\circ}="A";
  (9,0)*{\circ}="O";
 (-6,-10)*+{_1}*\frm{o}="B";
  (6,-10)*+{_2}*\frm{o}="C";
   (14,-10)*+{_3}*\frm{o}="D";
 "A"; "B" **\crv{(-5,-0)}; 
  "A"; "D" **\crv{(5,-0.5)};
   "A"; "B" **\crv{(5,-1)};
  "A"; "C" **\crv{(-5,-7)};
  \ar @{-} "A";"1" <0pt>
\ar @{-} "A";"3" <0pt>
\ar @{-} "O";"C" <0pt>
\ar @{-} "O";"D" <0pt>
\ar @{-} "O";"3" <0pt>
 \endxy}
 \Ea \in \caD\cP(2,3)
$$
which looks like a real graph with many vertices of two types --- the small ones which are decorated
by elements of the properad $\cP$ and big ones corresponding to the inputs and outputs of the properad $\caD\cP$.
Sometimes it is useful to understand such a graph as a {\em hypergraph} with  small vertices
playing the role of {\em hyperedges} (cf. \cite{MW}).

%

\sip

We consider below several concrete examples of polydifferential props and their applications.

\subsection{A note on deformation complexes and Lie algebras of derivations of prop(erad)s} Our definitions of these notions are slightly non-standard, but only with these non-standard definitions the $\caL ie_\infty$
claim (ii) in the Main theorem {\ref{1: Main Theorem}} holds true.

\sip

Let $(cP,\delta)$ be an arbitrary dg prop(erad), and $\cP^+$ be the prop(erad) freely generated by $\cP$ and one extra operation
$\begin{xy}
 <0mm,-0.55mm>*{};<0mm,-3mm>*{}**@{-},
 <0mm,0.5mm>*{};<0mm,3mm>*{}**@{-},
 <0mm,0mm>*{\bullet};<0mm,0mm>*{}**@{},
 \end{xy}$ of arity $(1,1)$ and of cohomological degree $+1$. We make $\cP^+$ into a {\em differential}\,
 prop(erad)   by setting the value of the differential $\delta^+$  on the new generator by
 $$
 \delta^+ \begin{xy}
 <0mm,-0.55mm>*{};<0mm,-3mm>*{}**@{-},
 <0mm,0.5mm>*{};<0mm,3mm>*{}**@{-},
 <0mm,0mm>*{\bullet};<0mm,0mm>*{}**@{},
 \end{xy} := \Ba{c} \begin{xy}
 <0mm,0mm>*{};<0mm,-3mm>*{}**@{-},
 <0mm,0mm>*{};<0mm,6mm>*{}**@{-},
 <0mm,0mm>*{\bullet};
 <0mm,3mm>*{\bullet};
 \end{xy}
 \Ea
 $$
 and on any other element $a\in \cP(m,n)$ (which we identify pictorially with the $(m,n)$-corolla
 whose vertex is decorated with $a$) by the formula
 $$
 \delta^+
 \begin{xy}
 <0mm,0mm>*{\bullet};<0mm,0mm>*{}**@{},
 <0mm,0mm>*{};<-8mm,5mm>*{}**@{-},
 <0mm,0mm>*{};<-4.5mm,5mm>*{}**@{-},
 <0mm,0mm>*{};<-1mm,5mm>*{\ldots}**@{},
 <0mm,0mm>*{};<4.5mm,5mm>*{}**@{-},
 <0mm,0mm>*{};<8mm,5mm>*{}**@{-},
   <0mm,0mm>*{};<-8.5mm,5.5mm>*{^1}**@{},
   <0mm,0mm>*{};<-5mm,5.5mm>*{^2}**@{},
   <0mm,0mm>*{};<4.5mm,5.5mm>*{^{m\hspace{-0.5mm}-\hspace{-0.5mm}1}}**@{},
   <0mm,0mm>*{};<9.0mm,5.5mm>*{^m}**@{},
 <0mm,0mm>*{};<-8mm,-5mm>*{}**@{-},
 <0mm,0mm>*{};<-4.5mm,-5mm>*{}**@{-},
 <0mm,0mm>*{};<-1mm,-5mm>*{\ldots}**@{},
 <0mm,0mm>*{};<4.5mm,-5mm>*{}**@{-},
 <0mm,0mm>*{};<8mm,-5mm>*{}**@{-},
   <0mm,0mm>*{};<-8.5mm,-6.9mm>*{^1}**@{},
   <0mm,0mm>*{};<-5mm,-6.9mm>*{^2}**@{},
   <0mm,0mm>*{};<4.5mm,-6.9mm>*{^{n\hspace{-0.5mm}-\hspace{-0.5mm}1}}**@{},
   <0mm,0mm>*{};<9.0mm,-6.9mm>*{^n}**@{},
 \end{xy}:= \delta
\begin{xy}
 <0mm,0mm>*{\bullet};<0mm,0mm>*{}**@{},
 <0mm,0mm>*{};<-8mm,5mm>*{}**@{-},
 <0mm,0mm>*{};<-4.5mm,5mm>*{}**@{-},
 <0mm,0mm>*{};<-1mm,5mm>*{\ldots}**@{},
 <0mm,0mm>*{};<4.5mm,5mm>*{}**@{-},
 <0mm,0mm>*{};<8mm,5mm>*{}**@{-},
   <0mm,0mm>*{};<-8.5mm,5.5mm>*{^1}**@{},
   <0mm,0mm>*{};<-5mm,5.5mm>*{^2}**@{},
   <0mm,0mm>*{};<4.5mm,5.5mm>*{^{m\hspace{-0.5mm}-\hspace{-0.5mm}1}}**@{},
   <0mm,0mm>*{};<9.0mm,5.5mm>*{^m}**@{},
 <0mm,0mm>*{};<-8mm,-5mm>*{}**@{-},
 <0mm,0mm>*{};<-4.5mm,-5mm>*{}**@{-},
 <0mm,0mm>*{};<-1mm,-5mm>*{\ldots}**@{},
 <0mm,0mm>*{};<4.5mm,-5mm>*{}**@{-},
 <0mm,0mm>*{};<8mm,-5mm>*{}**@{-},
   <0mm,0mm>*{};<-8.5mm,-6.9mm>*{^1}**@{},
   <0mm,0mm>*{};<-5mm,-6.9mm>*{^2}**@{},
   <0mm,0mm>*{};<4.5mm,-6.9mm>*{^{n\hspace{-0.5mm}-\hspace{-0.5mm}1}}**@{},
   <0mm,0mm>*{};<9.0mm,-6.9mm>*{^n}**@{},
 \end{xy}
+
\overset{m-1}{\underset{i=0}{\sum}}
\begin{xy}
 <0mm,0mm>*{\bullet};<0mm,0mm>*{}**@{},
 <0mm,0mm>*{};<-8mm,5mm>*{}**@{-},
 <0mm,0mm>*{};<-3.5mm,5mm>*{}**@{-},
 <0mm,0mm>*{};<-6mm,5mm>*{..}**@{},
 <0mm,0mm>*{};<0mm,5mm>*{}**@{-},
  <0mm,5mm>*{\bullet};
  <0mm,5mm>*{};<0mm,8mm>*{}**@{-},
  <0mm,5mm>*{};<0mm,9mm>*{^{i\hspace{-0.2mm}+\hspace{-0.5mm}1}}**@{},
<0mm,0mm>*{};<8mm,5mm>*{}**@{-},
<0mm,0mm>*{};<3.5mm,5mm>*{}**@{-},
 <0mm,0mm>*{};<6mm,5mm>*{..}**@{},
   <0mm,0mm>*{};<-8.5mm,5.5mm>*{^1}**@{},
   <0mm,0mm>*{};<-4mm,5.5mm>*{^i}**@{},
   <0mm,0mm>*{};<9.0mm,5.5mm>*{^m}**@{},
 <0mm,0mm>*{};<-8mm,-5mm>*{}**@{-},
 <0mm,0mm>*{};<-4.5mm,-5mm>*{}**@{-},
 <0mm,0mm>*{};<-1mm,-5mm>*{\ldots}**@{},
 <0mm,0mm>*{};<4.5mm,-5mm>*{}**@{-},
 <0mm,0mm>*{};<8mm,-5mm>*{}**@{-},
   <0mm,0mm>*{};<-8.5mm,-6.9mm>*{^1}**@{},
   <0mm,0mm>*{};<-5mm,-6.9mm>*{^2}**@{},
   <0mm,0mm>*{};<4.5mm,-6.9mm>*{^{n\hspace{-0.5mm}-\hspace{-0.5mm}1}}**@{},
   <0mm,0mm>*{};<9.0mm,-6.9mm>*{^n}**@{},
 \end{xy}
 - (-1)^{|a|}
\overset{n-1}{\underset{i=0}{\sum}}
 \begin{xy}
 <0mm,0mm>*{\bullet};<0mm,0mm>*{}**@{},
 <0mm,0mm>*{};<-8mm,-5mm>*{}**@{-},
 <0mm,0mm>*{};<-3.5mm,-5mm>*{}**@{-},
 <0mm,0mm>*{};<-6mm,-5mm>*{..}**@{},
 <0mm,0mm>*{};<0mm,-5mm>*{}**@{-},
  <0mm,-5mm>*{\bullet};
  <0mm,-5mm>*{};<0mm,-8mm>*{}**@{-},
  <0mm,-5mm>*{};<0mm,-10mm>*{^{i\hspace{-0.2mm}+\hspace{-0.5mm}1}}**@{},
<0mm,0mm>*{};<8mm,-5mm>*{}**@{-},
<0mm,0mm>*{};<3.5mm,-5mm>*{}**@{-},
 <0mm,0mm>*{};<6mm,-5mm>*{..}**@{},
   <0mm,0mm>*{};<-8.5mm,-6.9mm>*{^1}**@{},
   <0mm,0mm>*{};<-4mm,-6.9mm>*{^i}**@{},
   <0mm,0mm>*{};<9.0mm,-6.9mm>*{^n}**@{},
 <0mm,0mm>*{};<-8mm,5mm>*{}**@{-},
 <0mm,0mm>*{};<-4.5mm,5mm>*{}**@{-},
 <0mm,0mm>*{};<-1mm,5mm>*{\ldots}**@{},
 <0mm,0mm>*{};<4.5mm,5mm>*{}**@{-},
 <0mm,0mm>*{};<8mm,5mm>*{}**@{-},
   <0mm,0mm>*{};<-8.5mm,5.5mm>*{^1}**@{},
   <0mm,0mm>*{};<-5mm,5.5mm>*{^2}**@{},
   <0mm,0mm>*{};<4.5mm,5.5mm>*{^{m\hspace{-0.5mm}-\hspace{-0.5mm}1}}**@{},
   <0mm,0mm>*{};<9.0mm,5.5mm>*{^m}**@{},
 \end{xy}.
 $$
The dg prop(erad) $(\cP^+, \delta^+)$ is uniquely characterized by the fact that there is a 1-1 correspondence between representations
 $$
 \rho: \cP^+ \lon \cE nd_V
 $$
of $(\cP^+, \delta^+)$ in a dg vector space $(V,d)$, and representations of $\cP$ in the same space $V$
but equipped with a deformed differential $d+d'$, where $d':=\rho(\begin{xy}
 <0mm,-0.55mm>*{};<0mm,-3mm>*{}**@{-},
 <0mm,0.5mm>*{};<0mm,3mm>*{}**@{-},
 <0mm,0mm>*{\bullet};<0mm,0mm>*{}**@{},
 \end{xy})$.
Clearly any $\cP$-algebra is a $\cP^+$-algebra (with $d'=0$) so that there a canonical epimorphism
$\pi: \cP^+\to \cP$ of dg prop(erad)s; we also have a canonical monomorphism $\cP\rar \cP^+$ of {\em non-differential}\, prop(erad)s.
We define the dg Lie algebra, $\Der(\cP)$, of derivations of $(\cP,\delta)$ as the {\em complex of derivations of $(\cP^+,\delta^+)$ with values in
$\cP$}, that is, as the Lie algebra generated by linear maps
$$
D: \cP^+ \lon \cP
$$
which satisfy the condition
$$
D(a \boxtimes_{1,1} b)= D(a) \boxtimes_{1,1} \pi(b) + (-1)^{|a||D|} \pi(a)\boxtimes_{1,1} D(b), \ \ \
 \forall \ a,b\in \cP^+,
$$
where $\boxtimes_{1,1}$ stands \cite{Va} for the properadic composition along any connected oriented 2-vertex graph. The differential $d$ in $\Der(\cP)$ is given by the formula
$$
dD:= \delta\circ D  - (-1)^{|D|}  D\circ \delta^+.
$$

Let next $f: \cP \rar \cQ$ be a morphism of dg prop(erad)s. Its deformation complex $Def(\cP \stackrel{f}{\rar} \cQ)$
 is a $\caL ie _\infty$ algebra which is defined in full details (and in two different ways) in \cite{MV} with the help of a cofibrant resolution $\tilde{\cP}$ of the prop $\cP$. Most often
 $\tilde{\cP}$ is equal to (the prop enveloping of) the cobar construction  $\Omega(\cC)=\cF ree\langle\overline \cC[-1]\rangle$ of some coaugmented homotopy coproperad $\cC$ (with the cokernel denoted by $\overline \cC$), in which case one has an isomorphism of $\bS$-modules,
 $$
  Def(\cP \stackrel{f}{\rar} \cQ)\equiv  Def(\cP \stackrel{f}{\rar} \cQ) \simeq \prod_{m,n\geq 1}
  \Hom_{\bS_m^{op}\times \bS_n}\left(\overline{C}(m,n), Q(m,n)  \right).
 $$
 Again we need a slightly extended version of the deformation functor,
 $$
  \Def(\cP \stackrel{f}{\rar} \cQ):=Def(\cP \stackrel{f}{\rar} \cQ) \oplus Q(1,1)[1] \simeq  \prod_{m,n\geq 1}
  \Hom_{\bS_m^{op}\times \bS_n}\left({C}(m,n), Q(m,n)  \right).
 $$
which includes an extra degree of freedom for deformations of the differentials. With this definition we have an isomorphism of complexes
$$
\Der(P)=\Def(\cP\stackrel{\Id}{\rar}\cP)[1]
$$
relating our extended version of the complex derivations with the extended version of the deformations of the identity map.  Note however that this isomorphism does {\em not}\, respect Lie brackets ---
they are quite different (with the r.h.s. being in general a $\caL ie_\infty$ algebra, while the l.h.s.\ being in a general a dg Lie algebra).

\sip

Finally we warn the reader that the formality theorem for quantization of Lie bialgebras as formulated
in item (ii) of Theorem {\ref{1: Main Theorem}} holds true only for our extended version $\Def$ of the deformation functor $Def$.




\bip

\bip

{\Large
\section{\bf Associative bialgebras, Lie bialgebras and their deformation complexes}
}

\bip

\subsection{A properad of Lie bialgebras} A Lie bialgebra is a graded vector space $V$
equipped with compatible Lie brackets
$$
[\ ,\ ]: \wedge^2 V \rar V, \ \ \ \ \ \ \vartriangle: V \rar \wedge^2 V
$$
The compatibility relations are best described in terms of the properad $\LB$ governing Lie bialgebras via its representations $\rho_0: \LB \rar \cE nd_V$ which send one generator to the Lie bracket and another generator to the Lie cobracket,
$$
\rho_0\left( \begin{xy}
 <0mm,0.66mm>*{};<0mm,3mm>*{}**@{-},
 <0.39mm,-0.39mm>*{};<2.2mm,-2.2mm>*{}**@{-},
 <-0.35mm,-0.35mm>*{};<-2.2mm,-2.2mm>*{}**@{-},
 <0mm,0mm>*{\bu};<0mm,0mm>*{}**@{},
 \end{xy}\right) =[\, , \, ]
 \ \ \  , \ \ \
\rho_0\left(
 \begin{xy}
 <0mm,-0.55mm>*{};<0mm,-2.5mm>*{}**@{-},
 <0.5mm,0.5mm>*{};<2.2mm,2.2mm>*{}**@{-},
 <-0.48mm,0.48mm>*{};<-2.2mm,2.2mm>*{}**@{-},
 <0mm,0mm>*{\bu};<0mm,0mm>*{}**@{},
 \end{xy}\right) = \vartriangle \ .
$$
More precisely, the properad $\LB$ is defined \cite{D} as a quotient,
$$
\LB:= {\cF ree \langle E_0\rangle}/(R)
$$
of the free properad generated by an $\bS$-bimodule $E_0=\{E_0(m,n)\}$,
\[
E_0(m,n):=\left\{
\Ba{rr}
sgn_2\ot \id_1\equiv\mbox{span}\left\langle
\begin{xy}
 <0mm,-0.55mm>*{};<0mm,-2.5mm>*{}**@{-},
 <0.5mm,0.5mm>*{};<2.2mm,2.2mm>*{}**@{-},
 <-0.48mm,0.48mm>*{};<-2.2mm,2.2mm>*{}**@{-},
 <0mm,0mm>*{\bu};<0mm,0mm>*{}**@{},
 <0mm,-0.55mm>*{};<0mm,-3.8mm>*{_1}**@{},
 <0.5mm,0.5mm>*{};<2.7mm,2.8mm>*{^2}**@{},
 <-0.48mm,0.48mm>*{};<-2.7mm,2.8mm>*{^1}**@{},
 \end{xy}
=-
\begin{xy}
 <0mm,-0.55mm>*{};<0mm,-2.5mm>*{}**@{-},
 <0.5mm,0.5mm>*{};<2.2mm,2.2mm>*{}**@{-},
 <-0.48mm,0.48mm>*{};<-2.2mm,2.2mm>*{}**@{-},
 <0mm,0mm>*{\bu};<0mm,0mm>*{}**@{},
 <0mm,-0.55mm>*{};<0mm,-3.8mm>*{_1}**@{},
 <0.5mm,0.5mm>*{};<2.7mm,2.8mm>*{^1}**@{},
 <-0.48mm,0.48mm>*{};<-2.7mm,2.8mm>*{^2}**@{},
 \end{xy}
   \right\rangle  & \mbox{if}\ m=2, n=1,\vspace{3mm}\\
\id_1\ot sgn_2\equiv
\mbox{span}\left\langle
\begin{xy}
 <0mm,0.66mm>*{};<0mm,3mm>*{}**@{-},
 <0.39mm,-0.39mm>*{};<2.2mm,-2.2mm>*{}**@{-},
 <-0.35mm,-0.35mm>*{};<-2.2mm,-2.2mm>*{}**@{-},
 <0mm,0mm>*{\bu};<0mm,0mm>*{}**@{},
   <0mm,0.66mm>*{};<0mm,3.4mm>*{^1}**@{},
   <0.39mm,-0.39mm>*{};<2.9mm,-4mm>*{^2}**@{},
   <-0.35mm,-0.35mm>*{};<-2.8mm,-4mm>*{^1}**@{},
\end{xy}=-
\begin{xy}
 <0mm,0.66mm>*{};<0mm,3mm>*{}**@{-},
 <0.39mm,-0.39mm>*{};<2.2mm,-2.2mm>*{}**@{-},
 <-0.35mm,-0.35mm>*{};<-2.2mm,-2.2mm>*{}**@{-},
 <0mm,0mm>*{\bu};<0mm,0mm>*{}**@{},
   <0mm,0.66mm>*{};<0mm,3.4mm>*{^1}**@{},
   <0.39mm,-0.39mm>*{};<2.9mm,-4mm>*{^1}**@{},
   <-0.35mm,-0.35mm>*{};<-2.8mm,-4mm>*{^2}**@{},
\end{xy}
\right\rangle
\ & \mbox{if}\ m=1, n=2, \vspace{3mm}\\
0 & \mbox{otherwise}
\Ea
\right.
\]
modulo the ideal generated by the following relations\footnote{Here and everywhere all internal edges and legs in the graphical representations of elements of all the props considered are implicitly (but sometimes explicitly) oriented
by the flow which runs from the bottom of a picture to the top.}
\Beq\label{3: LieB relations}
R:\left\{
\Ba{l}
\Ba{c}\begin{xy}
 <0mm,0mm>*{\bu};<0mm,0mm>*{}**@{},
 <0mm,-0.49mm>*{};<0mm,-3.0mm>*{}**@{-},
 <0.49mm,0.49mm>*{};<1.9mm,1.9mm>*{}**@{-},
 <-0.5mm,0.5mm>*{};<-1.9mm,1.9mm>*{}**@{-},
 <-2.3mm,2.3mm>*{\bu};<-2.3mm,2.3mm>*{}**@{},
 <-1.8mm,2.8mm>*{};<0mm,4.9mm>*{}**@{-},
 <-2.8mm,2.9mm>*{};<-4.6mm,4.9mm>*{}**@{-},
   <0.49mm,0.49mm>*{};<2.7mm,2.3mm>*{^3}**@{},
   <-1.8mm,2.8mm>*{};<0.4mm,5.3mm>*{^2}**@{},
   <-2.8mm,2.9mm>*{};<-5.1mm,5.3mm>*{^1}**@{},
 \end{xy}
\ + \
\begin{xy}
 <0mm,0mm>*{\bu};<0mm,0mm>*{}**@{},
 <0mm,-0.49mm>*{};<0mm,-3.0mm>*{}**@{-},
 <0.49mm,0.49mm>*{};<1.9mm,1.9mm>*{}**@{-},
 <-0.5mm,0.5mm>*{};<-1.9mm,1.9mm>*{}**@{-},
 <-2.3mm,2.3mm>*{\bu};<-2.3mm,2.3mm>*{}**@{},
 <-1.8mm,2.8mm>*{};<0mm,4.9mm>*{}**@{-},
 <-2.8mm,2.9mm>*{};<-4.6mm,4.9mm>*{}**@{-},
   <0.49mm,0.49mm>*{};<2.7mm,2.3mm>*{^2}**@{},
   <-1.8mm,2.8mm>*{};<0.4mm,5.3mm>*{^1}**@{},
   <-2.8mm,2.9mm>*{};<-5.1mm,5.3mm>*{^3}**@{},
 \end{xy}
\ + \
\begin{xy}
 <0mm,0mm>*{\bu};<0mm,0mm>*{}**@{},
 <0mm,-0.49mm>*{};<0mm,-3.0mm>*{}**@{-},
 <0.49mm,0.49mm>*{};<1.9mm,1.9mm>*{}**@{-},
 <-0.5mm,0.5mm>*{};<-1.9mm,1.9mm>*{}**@{-},
 <-2.3mm,2.3mm>*{\bu};<-2.3mm,2.3mm>*{}**@{},
 <-1.8mm,2.8mm>*{};<0mm,4.9mm>*{}**@{-},
 <-2.8mm,2.9mm>*{};<-4.6mm,4.9mm>*{}**@{-},
   <0.49mm,0.49mm>*{};<2.7mm,2.3mm>*{^1}**@{},
   <-1.8mm,2.8mm>*{};<0.4mm,5.3mm>*{^3}**@{},
   <-2.8mm,2.9mm>*{};<-5.1mm,5.3mm>*{^2}**@{},
 \end{xy}\Ea =0
 \ \ \ \ , \ \ \
 \Ba{c}\begin{xy}
 <0mm,0mm>*{\bu};<0mm,0mm>*{}**@{},
 <0mm,0.69mm>*{};<0mm,3.0mm>*{}**@{-},
 <0.39mm,-0.39mm>*{};<2.4mm,-2.4mm>*{}**@{-},
 <-0.35mm,-0.35mm>*{};<-1.9mm,-1.9mm>*{}**@{-},
 <-2.4mm,-2.4mm>*{\bu};<-2.4mm,-2.4mm>*{}**@{},
 <-2.0mm,-2.8mm>*{};<0mm,-4.9mm>*{}**@{-},
 <-2.8mm,-2.9mm>*{};<-4.7mm,-4.9mm>*{}**@{-},
    <0.39mm,-0.39mm>*{};<3.3mm,-4.0mm>*{^3}**@{},
    <-2.0mm,-2.8mm>*{};<0.5mm,-6.7mm>*{^2}**@{},
    <-2.8mm,-2.9mm>*{};<-5.2mm,-6.7mm>*{^1}**@{},
 \end{xy}
\ + \
 \begin{xy}
 <0mm,0mm>*{\bu};<0mm,0mm>*{}**@{},
 <0mm,0.69mm>*{};<0mm,3.0mm>*{}**@{-},
 <0.39mm,-0.39mm>*{};<2.4mm,-2.4mm>*{}**@{-},
 <-0.35mm,-0.35mm>*{};<-1.9mm,-1.9mm>*{}**@{-},
 <-2.4mm,-2.4mm>*{\bu};<-2.4mm,-2.4mm>*{}**@{},
 <-2.0mm,-2.8mm>*{};<0mm,-4.9mm>*{}**@{-},
 <-2.8mm,-2.9mm>*{};<-4.7mm,-4.9mm>*{}**@{-},
    <0.39mm,-0.39mm>*{};<3.3mm,-4.0mm>*{^2}**@{},
    <-2.0mm,-2.8mm>*{};<0.5mm,-6.7mm>*{^1}**@{},
    <-2.8mm,-2.9mm>*{};<-5.2mm,-6.7mm>*{^3}**@{},
 \end{xy}
\ + \
 \begin{xy}
 <0mm,0mm>*{\bu};<0mm,0mm>*{}**@{},
 <0mm,0.69mm>*{};<0mm,3.0mm>*{}**@{-},
 <0.39mm,-0.39mm>*{};<2.4mm,-2.4mm>*{}**@{-},
 <-0.35mm,-0.35mm>*{};<-1.9mm,-1.9mm>*{}**@{-},
 <-2.4mm,-2.4mm>*{\bu};<-2.4mm,-2.4mm>*{}**@{},
 <-2.0mm,-2.8mm>*{};<0mm,-4.9mm>*{}**@{-},
 <-2.8mm,-2.9mm>*{};<-4.7mm,-4.9mm>*{}**@{-},
    <0.39mm,-0.39mm>*{};<3.3mm,-4.0mm>*{^1}**@{},
    <-2.0mm,-2.8mm>*{};<0.5mm,-6.7mm>*{^3}**@{},
    <-2.8mm,-2.9mm>*{};<-5.2mm,-6.7mm>*{^2}**@{},
 \end{xy}\Ea =0
 \\
 \begin{xy}
 <0mm,2.47mm>*{};<0mm,0.12mm>*{}**@{-},
 <0.5mm,3.5mm>*{};<2.2mm,5.2mm>*{}**@{-},
 <-0.48mm,3.48mm>*{};<-2.2mm,5.2mm>*{}**@{-},
 <0mm,3mm>*{\bu};<0mm,3mm>*{}**@{},
  <0mm,-0.8mm>*{\bu};<0mm,-0.8mm>*{}**@{},
<-0.39mm,-1.2mm>*{};<-2.2mm,-3.5mm>*{}**@{-},
 <0.39mm,-1.2mm>*{};<2.2mm,-3.5mm>*{}**@{-},
     <0.5mm,3.5mm>*{};<2.8mm,5.7mm>*{^2}**@{},
     <-0.48mm,3.48mm>*{};<-2.8mm,5.7mm>*{^1}**@{},
   <0mm,-0.8mm>*{};<-2.7mm,-5.2mm>*{^1}**@{},
   <0mm,-0.8mm>*{};<2.7mm,-5.2mm>*{^2}**@{},
\end{xy}
\  - \
\begin{xy}
 <0mm,-1.3mm>*{};<0mm,-3.5mm>*{}**@{-},
 <0.38mm,-0.2mm>*{};<2.0mm,2.0mm>*{}**@{-},
 <-0.38mm,-0.2mm>*{};<-2.2mm,2.2mm>*{}**@{-},
<0mm,-0.8mm>*{\bu};<0mm,0.8mm>*{}**@{},
 <2.4mm,2.4mm>*{\bu};<2.4mm,2.4mm>*{}**@{},
 <2.77mm,2.0mm>*{};<4.4mm,-0.8mm>*{}**@{-},
 <2.4mm,3mm>*{};<2.4mm,5.2mm>*{}**@{-},
     <0mm,-1.3mm>*{};<0mm,-5.3mm>*{^1}**@{},
     <2.5mm,2.3mm>*{};<5.1mm,-2.6mm>*{^2}**@{},
    <2.4mm,2.5mm>*{};<2.4mm,5.7mm>*{^2}**@{},
    <-0.38mm,-0.2mm>*{};<-2.8mm,2.5mm>*{^1}**@{},
    \end{xy}
\  + \
\begin{xy}
 <0mm,-1.3mm>*{};<0mm,-3.5mm>*{}**@{-},
 <0.38mm,-0.2mm>*{};<2.0mm,2.0mm>*{}**@{-},
 <-0.38mm,-0.2mm>*{};<-2.2mm,2.2mm>*{}**@{-},
<0mm,-0.8mm>*{\bu};<0mm,0.8mm>*{}**@{},
 <2.4mm,2.4mm>*{\bu};<2.4mm,2.4mm>*{}**@{},
 <2.77mm,2.0mm>*{};<4.4mm,-0.8mm>*{}**@{-},
 <2.4mm,3mm>*{};<2.4mm,5.2mm>*{}**@{-},
     <0mm,-1.3mm>*{};<0mm,-5.3mm>*{^2}**@{},
     <2.5mm,2.3mm>*{};<5.1mm,-2.6mm>*{^1}**@{},
    <2.4mm,2.5mm>*{};<2.4mm,5.7mm>*{^2}**@{},
    <-0.38mm,-0.2mm>*{};<-2.8mm,2.5mm>*{^1}**@{},
    \end{xy}
\  - \
\begin{xy}
 <0mm,-1.3mm>*{};<0mm,-3.5mm>*{}**@{-},
 <0.38mm,-0.2mm>*{};<2.0mm,2.0mm>*{}**@{-},
 <-0.38mm,-0.2mm>*{};<-2.2mm,2.2mm>*{}**@{-},
<0mm,-0.8mm>*{\bu};<0mm,0.8mm>*{}**@{},
 <2.4mm,2.4mm>*{\bu};<2.4mm,2.4mm>*{}**@{},
 <2.77mm,2.0mm>*{};<4.4mm,-0.8mm>*{}**@{-},
 <2.4mm,3mm>*{};<2.4mm,5.2mm>*{}**@{-},
     <0mm,-1.3mm>*{};<0mm,-5.3mm>*{^2}**@{},
     <2.5mm,2.3mm>*{};<5.1mm,-2.6mm>*{^1}**@{},
    <2.4mm,2.5mm>*{};<2.4mm,5.7mm>*{^1}**@{},
    <-0.38mm,-0.2mm>*{};<-2.8mm,2.5mm>*{^2}**@{},
    \end{xy}
\ + \
\begin{xy}
 <0mm,-1.3mm>*{};<0mm,-3.5mm>*{}**@{-},
 <0.38mm,-0.2mm>*{};<2.0mm,2.0mm>*{}**@{-},
 <-0.38mm,-0.2mm>*{};<-2.2mm,2.2mm>*{}**@{-},
<0mm,-0.8mm>*{\bu};<0mm,0.8mm>*{}**@{},
 <2.4mm,2.4mm>*{\bu};<2.4mm,2.4mm>*{}**@{},
 <2.77mm,2.0mm>*{};<4.4mm,-0.8mm>*{}**@{-},
 <2.4mm,3mm>*{};<2.4mm,5.2mm>*{}**@{-},
     <0mm,-1.3mm>*{};<0mm,-5.3mm>*{^1}**@{},
     <2.5mm,2.3mm>*{};<5.1mm,-2.6mm>*{^2}**@{},
    <2.4mm,2.5mm>*{};<2.4mm,5.7mm>*{^1}**@{},
    <-0.38mm,-0.2mm>*{};<-2.8mm,2.5mm>*{^2}**@{},
    \end{xy}=0
\Ea
\right.
\Eeq
Its minimal resolution,
 $\LB_\infty$,  is a dg free prop,
$$
\LB_\infty=\cF ree \langle E\rangle,
$$
generated by the $\bS$--bimodule $ E=\{ E(m,n)\}_{m,n\geq 1, m+n\geq 3}$,
$$
 E(m,n):= sgn_m\ot sgn_n[m+n-3]=\mbox{span}\left\langle
\Ba{c}\resizebox{14mm}{!}{\begin{xy}
 <0mm,0mm>*{\bu};<0mm,0mm>*{}**@{},
 <-0.6mm,0.44mm>*{};<-8mm,5mm>*{}**@{-},
 <-0.4mm,0.7mm>*{};<-4.5mm,5mm>*{}**@{-},
 <0mm,0mm>*{};<-1mm,5mm>*{\ldots}**@{},
 <0.4mm,0.7mm>*{};<4.5mm,5mm>*{}**@{-},
 <0.6mm,0.44mm>*{};<8mm,5mm>*{}**@{-},
   <0mm,0mm>*{};<-8.5mm,5.5mm>*{^1}**@{},
   <0mm,0mm>*{};<-5mm,5.5mm>*{^2}**@{},
   <0mm,0mm>*{};<4.5mm,5.5mm>*{^{m\hspace{-0.5mm}-\hspace{-0.5mm}1}}**@{},
   <0mm,0mm>*{};<9.0mm,5.5mm>*{^m}**@{},
 <-0.6mm,-0.44mm>*{};<-8mm,-5mm>*{}**@{-},
 <-0.4mm,-0.7mm>*{};<-4.5mm,-5mm>*{}**@{-},
 <0mm,0mm>*{};<-1mm,-5mm>*{\ldots}**@{},
 <0.4mm,-0.7mm>*{};<4.5mm,-5mm>*{}**@{-},
 <0.6mm,-0.44mm>*{};<8mm,-5mm>*{}**@{-},
   <0mm,0mm>*{};<-8.5mm,-6.9mm>*{^1}**@{},
   <0mm,0mm>*{};<-5mm,-6.9mm>*{^2}**@{},
   <0mm,0mm>*{};<4.5mm,-6.9mm>*{^{n\hspace{-0.5mm}-\hspace{-0.5mm}1}}**@{},
   <0mm,0mm>*{};<9.0mm,-6.9mm>*{^n}**@{},
 \end{xy}}\Ea
\right\rangle,
$$
and  with the differential given on generating corollas by \cite{MaVo,Va}
\Beq\label{3: differential in LieBinfty}
\delta
\Ba{c}\resizebox{14mm}{!}{\begin{xy}
 <0mm,0mm>*{\bu};<0mm,0mm>*{}**@{},
 <-0.6mm,0.44mm>*{};<-8mm,5mm>*{}**@{-},
 <-0.4mm,0.7mm>*{};<-4.5mm,5mm>*{}**@{-},
 <0mm,0mm>*{};<-1mm,5mm>*{\ldots}**@{},
 <0.4mm,0.7mm>*{};<4.5mm,5mm>*{}**@{-},
 <0.6mm,0.44mm>*{};<8mm,5mm>*{}**@{-},
   <0mm,0mm>*{};<-8.5mm,5.5mm>*{^1}**@{},
   <0mm,0mm>*{};<-5mm,5.5mm>*{^2}**@{},
   <0mm,0mm>*{};<4.5mm,5.5mm>*{^{m\hspace{-0.5mm}-\hspace{-0.5mm}1}}**@{},
   <0mm,0mm>*{};<9.0mm,5.5mm>*{^m}**@{},
 <-0.6mm,-0.44mm>*{};<-8mm,-5mm>*{}**@{-},
 <-0.4mm,-0.7mm>*{};<-4.5mm,-5mm>*{}**@{-},
 <0mm,0mm>*{};<-1mm,-5mm>*{\ldots}**@{},
 <0.4mm,-0.7mm>*{};<4.5mm,-5mm>*{}**@{-},
 <0.6mm,-0.44mm>*{};<8mm,-5mm>*{}**@{-},
   <0mm,0mm>*{};<-8.5mm,-6.9mm>*{^1}**@{},
   <0mm,0mm>*{};<-5mm,-6.9mm>*{^2}**@{},
   <0mm,0mm>*{};<4.5mm,-6.9mm>*{^{n\hspace{-0.5mm}-\hspace{-0.5mm}1}}**@{},
   <0mm,0mm>*{};<9.0mm,-6.9mm>*{^n}**@{},
 \end{xy}}\Ea
\ \ = \ \
 \sum_{[1,\ldots,m]=I_1\sqcup I_2\atop
 {|I_1|\geq 0, |I_2|\geq 1}}
 \sum_{[1,\ldots,n]=J_1\sqcup J_2\atop
 {|J_1|\geq 1, |J_2|\geq 1}
}\hspace{0mm}
(-1)^{\sigma(I_1\sqcup I_2)+ |I_1||I_2|+|J_1||J_2|}
\Ba{c}\resizebox{20mm}{!}{ \begin{xy}
 <0mm,0mm>*{\bu};<0mm,0mm>*{}**@{},
 <-0.6mm,0.44mm>*{};<-8mm,5mm>*{}**@{-},
 <-0.4mm,0.7mm>*{};<-4.5mm,5mm>*{}**@{-},
 <0mm,0mm>*{};<0mm,5mm>*{\ldots}**@{},
 <0.4mm,0.7mm>*{};<4.5mm,5mm>*{}**@{-},
 <0.6mm,0.44mm>*{};<12.4mm,4.8mm>*{}**@{-},
     <0mm,0mm>*{};<-2mm,7mm>*{\overbrace{\ \ \ \ \ \ \ \ \ \ \ \ }}**@{},
     <0mm,0mm>*{};<-2mm,9mm>*{^{I_1}}**@{},
 <-0.6mm,-0.44mm>*{};<-8mm,-5mm>*{}**@{-},
 <-0.4mm,-0.7mm>*{};<-4.5mm,-5mm>*{}**@{-},
 <0mm,0mm>*{};<-1mm,-5mm>*{\ldots}**@{},
 <0.4mm,-0.7mm>*{};<4.5mm,-5mm>*{}**@{-},
 <0.6mm,-0.44mm>*{};<8mm,-5mm>*{}**@{-},
      <0mm,0mm>*{};<0mm,-7mm>*{\underbrace{\ \ \ \ \ \ \ \ \ \ \ \ \ \ \
      }}**@{},
      <0mm,0mm>*{};<0mm,-10.6mm>*{_{J_1}}**@{},
 <13mm,5mm>*{};<13mm,5mm>*{\bu}**@{},
 <12.6mm,5.44mm>*{};<5mm,10mm>*{}**@{-},
 <12.6mm,5.7mm>*{};<8.5mm,10mm>*{}**@{-},
 <13mm,5mm>*{};<13mm,10mm>*{\ldots}**@{},
 <13.4mm,5.7mm>*{};<16.5mm,10mm>*{}**@{-},
 <13.6mm,5.44mm>*{};<20mm,10mm>*{}**@{-},
      <13mm,5mm>*{};<13mm,12mm>*{\overbrace{\ \ \ \ \ \ \ \ \ \ \ \ \ \ }}**@{},
      <13mm,5mm>*{};<13mm,14mm>*{^{I_2}}**@{},
 <12.4mm,4.3mm>*{};<8mm,0mm>*{}**@{-},
 <12.6mm,4.3mm>*{};<12mm,0mm>*{\ldots}**@{},
 <13.4mm,4.5mm>*{};<16.5mm,0mm>*{}**@{-},
 <13.6mm,4.8mm>*{};<20mm,0mm>*{}**@{-},
     <13mm,5mm>*{};<14.3mm,-2mm>*{\underbrace{\ \ \ \ \ \ \ \ \ \ \ }}**@{},
     <13mm,5mm>*{};<14.3mm,-4.5mm>*{_{J_2}}**@{},
 \end{xy}}\Ea
\Eeq
where $\sigma(I_1\sqcup I_2)$ and $\sigma(J_1\sqcup J_2)$ are the signs of the shuffles
$[1,\ldots,m]\rar I_1\sqcup I_2$ and, respectively, $[1,\ldots,n]\rar J_1\sqcup J_2$. For example,
\Beq\label{3: delta in LieBi_infty on (2,2) corolla}
\delta
\begin{xy}
 <0.5mm,0.5mm>*{};<2.2mm,2.2mm>*{}**@{-},
 <-0.48mm,0.48mm>*{};<-2.2mm,2.2mm>*{}**@{-},
 <0mm,0mm>*{\bu};<0mm,0mm>*{}**@{},
 <0.5mm,0.5mm>*{};<2.7mm,2.8mm>*{^2}**@{},
 <-0.48mm,0.48mm>*{};<-2.7mm,2.8mm>*{^1}**@{},
 <0.39mm,-0.39mm>*{};<2.2mm,-2.2mm>*{}**@{-},
 <-0.35mm,-0.35mm>*{};<-2.2mm,-2.2mm>*{}**@{-},
 <0mm,0mm>*{};<0mm,0mm>*{}**@{},
   <0.39mm,-0.39mm>*{};<2.9mm,-4mm>*{^2}**@{},
   <-0.35mm,-0.35mm>*{};<-2.8mm,-4mm>*{^1}**@{},
\end{xy}
=
 \begin{xy}
 <0mm,2.47mm>*{};<0mm,0.12mm>*{}**@{-},
 <0.5mm,3.5mm>*{};<2.2mm,5.2mm>*{}**@{-},
 <-0.48mm,3.48mm>*{};<-2.2mm,5.2mm>*{}**@{-},
 <0mm,3mm>*{\bu};<0mm,3mm>*{}**@{},
  <0mm,-0.8mm>*{\bu};<0mm,-0.8mm>*{}**@{},
<-0.39mm,-1.2mm>*{};<-2.2mm,-3.5mm>*{}**@{-},
 <0.39mm,-1.2mm>*{};<2.2mm,-3.5mm>*{}**@{-},
     <0.5mm,3.5mm>*{};<2.8mm,5.7mm>*{^2}**@{},
     <-0.48mm,3.48mm>*{};<-2.8mm,5.7mm>*{^1}**@{},
   <0mm,-0.8mm>*{};<-2.7mm,-5.2mm>*{^1}**@{},
   <0mm,-0.8mm>*{};<2.7mm,-5.2mm>*{^2}**@{},
\end{xy}
\  - \
\begin{xy}
 <0mm,-1.3mm>*{};<0mm,-3.5mm>*{}**@{-},
 <0.38mm,-0.2mm>*{};<2.0mm,2.0mm>*{}**@{-},
 <-0.38mm,-0.2mm>*{};<-2.2mm,2.2mm>*{}**@{-},
<0mm,-0.8mm>*{\bu};<0mm,0.8mm>*{}**@{},
 <2.4mm,2.4mm>*{\bu};<2.4mm,2.4mm>*{}**@{},
 <2.77mm,2.0mm>*{};<4.4mm,-0.8mm>*{}**@{-},
 <2.4mm,3mm>*{};<2.4mm,5.2mm>*{}**@{-},
     <0mm,-1.3mm>*{};<0mm,-5.3mm>*{^1}**@{},
     <2.5mm,2.3mm>*{};<5.1mm,-2.6mm>*{^2}**@{},
    <2.4mm,2.5mm>*{};<2.4mm,5.7mm>*{^2}**@{},
    <-0.38mm,-0.2mm>*{};<-2.8mm,2.5mm>*{^1}**@{},
    \end{xy}
\  + \
\begin{xy}
 <0mm,-1.3mm>*{};<0mm,-3.5mm>*{}**@{-},
 <0.38mm,-0.2mm>*{};<2.0mm,2.0mm>*{}**@{-},
 <-0.38mm,-0.2mm>*{};<-2.2mm,2.2mm>*{}**@{-},
<0mm,-0.8mm>*{\bu};<0mm,0.8mm>*{}**@{},
 <2.4mm,2.4mm>*{\bu};<2.4mm,2.4mm>*{}**@{},
 <2.77mm,2.0mm>*{};<4.4mm,-0.8mm>*{}**@{-},
 <2.4mm,3mm>*{};<2.4mm,5.2mm>*{}**@{-},
     <0mm,-1.3mm>*{};<0mm,-5.3mm>*{^2}**@{},
     <2.5mm,2.3mm>*{};<5.1mm,-2.6mm>*{^1}**@{},
    <2.4mm,2.5mm>*{};<2.4mm,5.7mm>*{^2}**@{},
    <-0.38mm,-0.2mm>*{};<-2.8mm,2.5mm>*{^1}**@{},
    \end{xy}
\  - \
\begin{xy}
 <0mm,-1.3mm>*{};<0mm,-3.5mm>*{}**@{-},
 <0.38mm,-0.2mm>*{};<2.0mm,2.0mm>*{}**@{-},
 <-0.38mm,-0.2mm>*{};<-2.2mm,2.2mm>*{}**@{-},
<0mm,-0.8mm>*{\bu};<0mm,0.8mm>*{}**@{},
 <2.4mm,2.4mm>*{\bu};<2.4mm,2.4mm>*{}**@{},
 <2.77mm,2.0mm>*{};<4.4mm,-0.8mm>*{}**@{-},
 <2.4mm,3mm>*{};<2.4mm,5.2mm>*{}**@{-},
     <0mm,-1.3mm>*{};<0mm,-5.3mm>*{^2}**@{},
     <2.5mm,2.3mm>*{};<5.1mm,-2.6mm>*{^1}**@{},
    <2.4mm,2.5mm>*{};<2.4mm,5.7mm>*{^1}**@{},
    <-0.38mm,-0.2mm>*{};<-2.8mm,2.5mm>*{^2}**@{},
    \end{xy}
\ + \
\begin{xy}
 <0mm,-1.3mm>*{};<0mm,-3.5mm>*{}**@{-},
 <0.38mm,-0.2mm>*{};<2.0mm,2.0mm>*{}**@{-},
 <-0.38mm,-0.2mm>*{};<-2.2mm,2.2mm>*{}**@{-},
<0mm,-0.8mm>*{\bu};<0mm,0.8mm>*{}**@{},
 <2.4mm,2.4mm>*{\bu};<2.4mm,2.4mm>*{}**@{},
 <2.77mm,2.0mm>*{};<4.4mm,-0.8mm>*{}**@{-},
 <2.4mm,3mm>*{};<2.4mm,5.2mm>*{}**@{-},
     <0mm,-1.3mm>*{};<0mm,-5.3mm>*{^1}**@{},
     <2.5mm,2.3mm>*{};<5.1mm,-2.6mm>*{^2}**@{},
    <2.4mm,2.5mm>*{};<2.4mm,5.7mm>*{^1}**@{},
    <-0.38mm,-0.2mm>*{};<-2.8mm,2.5mm>*{^2}**@{},
    \end{xy}.
\Eeq
Strongly homotopy Lie bialgebra structures,
$
\rho: \LB_\infty \lon \cE nd_V,
$
 in a dg vector space $V$ can be  identified with Maurer-Cartan (MC, for short) elements of a dg Lie algebra $\fg_V$
 defined next.

\subsubsection{\bf Strongly homotopy Lie bialgebra structures as Maurer-Cartan elements}\label{2: subsect on LieBi as MC} Let
$V$ be a dg vector space (with differential denoted by $d$). According to the general theory \cite{MV}, there is a one-to-one correspondence between the set of representations, $\{
\rho: \LB_\infty \rar \cE nd_V\}$, and the set of Maurer-Cartan elements in the dg Lie algebra
\Beq\label{2: fl_V'}
\Def(\LB_\infty \stackrel{0}{\rar} \cE nd_V)\simeq \prod_{m,n\geq 1}
\wedge^mV^*\ot \wedge^n V[2-m-n]= \prod_{m,n\geq 1} \odot^m(V^*[-1])\ot \odot^n(V[-1]) [2] =: \fg_V[2]
\Eeq
controlling deformations of the zero map $\LB_\infty \stackrel{0}{\rar} \cE nd_V$.
The differential in $\fg_V$ is induced by the differential in $V$ while the Lie bracket
can be described explicitly as follows. First one notices that the completed graded vector space
$$
\fg_V= \prod_{m,n\geq 1} \odot^m(V^*[-1])\ot \odot^n(V[-1])=\widehat{\odot^{\bu \geq 1}}\left( V^*[-1])\oplus V[-1]\right)
$$
is naturally a 3-algebra with  degree $-2$ Lie brackets, $\{\ ,\ \}$,
given on generators by
\[
\{sv, sw\}=0,\ \ \{s\al, s\be\}=0, \ \ \{s\al, sv\}=<\al,v>, \ \ \forall v,w\in V, \al,\be\in V^*.
\]
where $s: V\rar V[-1]$ and $s: V^*\rar V^*[-1]$ are natural isomorphisms.
Maurer-Cartan elements in $\fg_V$, that is degree 3 elements $\nu$ satisfying the equation
$$
\{\nu,\nu\}=0,
$$
 are in 1-1 correspondence with
representations $\nu: \LB_\infty \rar \cE nd_V$. Such elements satisfying the condition
$$
\nu \in \odot^2(V^*[-1)\ot V[-1] \ \oplus V^*[-1]\ot \odot^2(V[-1])
$$
are precisely Lie bialgebra structures in $V$.

\sip

Sometimes it is useful to use a coordinate representation of the above structure. Let $\{x^i\}_{i\in I}$ be a basis in $V$ and $\{x_i\}_{i\in I}$ the  dual basis in $V^*$. Set
$\{\psi:=sx_i \}_{i\in I}$ and $\{\psi^i:=sx^i\}$ for the associated
bases of $V[-1]$ and $V^*[-1]$ respectively. Note that $|\psi_i|+|\psi^i|=2$ and that
$
{\fg}_V\simeq \K[[\psi_i,\psi^i]].
$
Then the degree $-2$ Lie brackets in $\fg_V$ are given explicitly by
\[
\{f, g\}=\sum_{i\in I} (-1)^{|f||\psi^i|} \frac{\p f}{\p  \psi_i} \frac{\p g}{\p \psi^i} -
(-1)^{|f||g|} \frac{\p g}{\p \psi_i} \frac{\p f}{\p \psi^i}.
\]
It is an easy exercise to check that elements of the form
$$
\nu:=\sum_{i,j,k\in I} C_{ij}^k \psi_k\psi^i\psi^j + \Phi_k^{ij}\psi^k \psi_i\psi_j.
$$
satisfy $\{\nu,\nu\}=0$ if and only if the maps
$$
\left(\bigtriangleup:V\rar \wedge^2 V, [\ ,\ ]:\wedge^2V \rar V\right)
$$
with the structure constants
$C_{ij}^k$ and $\Phi_k^{ij}$ respectively,
$$
[x_i,x_j]=:\sum_{k\in I} C_{ij}^k x_k,\ \ \ \ \bigtriangleup(x_k)=:\sum_{i,j\in I} \Phi_k^{ij} x_i\wedge x_j.
$$
define a Lie bialgebra structure in $V$.
(Here we assumed for simplicity that $V$ is concentrated in homological degree zero to avoid standard Koszul signs in the formulae.)

\subsection{Associative bialgebras.} An {\em associative  bialgebra}\,
 is, by definition, a graded vector space $V$ equipped with two degree zero linear
maps,
$$
\Ba{rccc}
\mu: &  V\ot V& \lon & V \\
       & a\ot b    & \lon & ab
\Ea
\ \ \ \ \  , \ \ \ \
\Ba{rccc}
\Delta: & V& \lon & V\ot V \\
       & a    & \lon &  \sum a_{1}\ot a_{2}
\Ea
$$
satisfying,
\Bi
\item[(i)] the associativity identity:
$(ab)c=a(bc)$;
\item[(ii)] the coassociativity identity: $(\Delta\ot\Id)\Delta a=
(\Id\ot \Delta)\Delta a$;
\item[(iii)] the compatibility identity: $\Delta$ is a morphism of algebras,
i.e.\ $\Delta(ab)=\sum (-1)^{a_2b_1} a_1b_1\ot a_2b_2$,
\Ei
for any $a,b, c\in V$. We often abbreviate ``associative bialgebra" to simply ``bialgebra".

A prop of bialgebras, $\Assb$, is generated by two kinds of operations,
the multiplication $\begin{xy}
 <0mm,0.66mm>*{};<0mm,3mm>*{}**@{-},
 <0.39mm,-0.39mm>*{};<2.2mm,-2.2mm>*{}**@{-},
 <-0.35mm,-0.35mm>*{};<-2.2mm,-2.2mm>*{}**@{-},
 <0mm,0mm>*{\circ};<0mm,0mm>*{}**@{},
\end{xy}$ and the comultiplication $\begin{xy}
 <0mm,-0.55mm>*{};<0mm,-2.5mm>*{}**@{-},
 <0.5mm,0.5mm>*{};<2.2mm,2.2mm>*{}**@{-},
 <-0.48mm,0.48mm>*{};<-2.2mm,2.2mm>*{}**@{-},
 <0mm,0mm>*{\circ};<0mm,0mm>*{}**@{},
 \end{xy}$, subject to relations which assure that its representations,
$$
\rho: \Assb \lon \cE nd_A,
$$
 in one-to-one correspondence with  associative bialgebra  structures on a graded vector space $A$. More precisely, $\cA ss\cB$ is the quotient,
$$
\Assb:= {\cF ree\langle A_0 \rangle}/(R)
$$
of the free prop, $\cF ree\langle A_0 \rangle$, generated\footnote{Later we shall work with 2-coloured props so we reserve from now on the ``dashed colour"  to $\Assb_\infty$ operations.} by an $\bS$-bimodule $A_0=\{A_0(m,n)\}$,
\[
A_0(m,n):=\left\{
\Ba{rr}
\K[\bS_2]\ot \id_1\equiv\mbox{span}\left\langle
\begin{xy}
 <0mm,-0.55mm>*{};<0mm,-2.5mm>*{}**@{.},
 <0.5mm,0.5mm>*{};<2.2mm,2.2mm>*{}**@{.},
 <-0.48mm,0.48mm>*{};<-2.2mm,2.2mm>*{}**@{.},
 <0mm,0mm>*{\circ};<0mm,0mm>*{}**@{},
 <0mm,-0.55mm>*{};<0mm,-3.8mm>*{_1}**@{},
 <0.5mm,0.5mm>*{};<2.7mm,2.8mm>*{^2}**@{},
 <-0.48mm,0.48mm>*{};<-2.7mm,2.8mm>*{^1}**@{},
 \end{xy}
\,
,\,
\begin{xy}
 <0mm,-0.55mm>*{};<0mm,-2.5mm>*{}**@{.},
 <0.5mm,0.5mm>*{};<2.2mm,2.2mm>*{}**@{.},
 <-0.48mm,0.48mm>*{};<-2.2mm,2.2mm>*{}**@{.},
 <0mm,0mm>*{\circ};<0mm,0mm>*{}**@{},
 <0mm,-0.55mm>*{};<0mm,-3.8mm>*{_1}**@{},
 <0.5mm,0.5mm>*{};<2.7mm,2.8mm>*{^1}**@{},
 <-0.48mm,0.48mm>*{};<-2.7mm,2.8mm>*{^2}**@{},
 \end{xy}
   \right\rangle  & \mbox{if}\ m=2, n=1,\vspace{3mm}\\
\id_1\ot \K[\bS_2]\equiv
\mbox{span}\left\langle
\begin{xy}
 <0mm,0.66mm>*{};<0mm,3mm>*{}**@{.},
 <0.39mm,-0.39mm>*{};<2.2mm,-2.2mm>*{}**@{.},
 <-0.35mm,-0.35mm>*{};<-2.2mm,-2.2mm>*{}**@{.},
 <0mm,0mm>*{\circ};<0mm,0mm>*{}**@{},
   <0mm,0.66mm>*{};<0mm,3.4mm>*{^1}**@{},
   <0.39mm,-0.39mm>*{};<2.9mm,-4mm>*{^2}**@{},
   <-0.35mm,-0.35mm>*{};<-2.8mm,-4mm>*{^1}**@{},
\end{xy}
\,
,\,
\begin{xy}
 <0mm,0.66mm>*{};<0mm,3mm>*{}**@{.},
 <0.39mm,-0.39mm>*{};<2.2mm,-2.2mm>*{}**@{.},
 <-0.35mm,-0.35mm>*{};<-2.2mm,-2.2mm>*{}**@{.},
 <0mm,0mm>*{\circ};<0mm,0mm>*{}**@{},
   <0mm,0.66mm>*{};<0mm,3.4mm>*{^1}**@{},
   <0.39mm,-0.39mm>*{};<2.9mm,-4mm>*{^1}**@{},
   <-0.35mm,-0.35mm>*{};<-2.8mm,-4mm>*{^2}**@{},
\end{xy}
\right\rangle
\ & \mbox{if}\ m=1, n=2, \vspace{3mm}\\
0 & \mbox{otherwise}
\Ea
\right.
\]
modulo the ideal generated by relations
\Beq\label{2: bialgebra relations}
R:\left\{
\Ba{c}
\begin{xy}
 <0mm,0mm>*{\circ};<0mm,0mm>*{}**@{},
 <0mm,-0.49mm>*{};<0mm,-3.0mm>*{}**@{.},
 <0.49mm,0.49mm>*{};<1.9mm,1.9mm>*{}**@{.},
 <-0.5mm,0.5mm>*{};<-1.9mm,1.9mm>*{}**@{.},
 <-2.3mm,2.3mm>*{\circ};<-2.3mm,2.3mm>*{}**@{},
 <-1.8mm,2.8mm>*{};<0mm,4.9mm>*{}**@{.},
 <-2.8mm,2.9mm>*{};<-4.6mm,4.9mm>*{}**@{.},
   <0.49mm,0.49mm>*{};<2.7mm,2.3mm>*{^3}**@{},
   <-1.8mm,2.8mm>*{};<0.4mm,5.3mm>*{^2}**@{},
   <-2.8mm,2.9mm>*{};<-5.1mm,5.3mm>*{^1}**@{},
 \end{xy}\Ea
\ - \
\Ba{c}
\begin{xy}
 <0mm,0mm>*{\circ};<0mm,0mm>*{}**@{},
 <0mm,-0.49mm>*{};<0mm,-3.0mm>*{}**@{.},
 <0.49mm,0.49mm>*{};<1.9mm,1.9mm>*{}**@{.},
 <-0.5mm,0.5mm>*{};<-1.9mm,1.9mm>*{}**@{.},
 <2.3mm,2.3mm>*{\circ};<-2.3mm,2.3mm>*{}**@{},
 <1.8mm,2.8mm>*{};<0mm,4.9mm>*{}**@{.},
 <2.8mm,2.9mm>*{};<4.6mm,4.9mm>*{}**@{.},
   <0.49mm,0.49mm>*{};<-2.7mm,2.3mm>*{^1}**@{},
   <-1.8mm,2.8mm>*{};<0mm,5.3mm>*{^2}**@{},
   <-2.8mm,2.9mm>*{};<5.1mm,5.3mm>*{^3}**@{},
 \end{xy}\Ea=0, \ \ \ \ \
 \Ba{c}\begin{xy}
 <0mm,0mm>*{\circ};<0mm,0mm>*{}**@{},
 <0mm,0.69mm>*{};<0mm,3.0mm>*{}**@{.},
 <0.39mm,-0.39mm>*{};<2.4mm,-2.4mm>*{}**@{.},
 <-0.35mm,-0.35mm>*{};<-1.9mm,-1.9mm>*{}**@{.},
 <-2.4mm,-2.4mm>*{\circ};<-2.4mm,-2.4mm>*{}**@{},
 <-2.0mm,-2.8mm>*{};<0mm,-4.9mm>*{}**@{.},
 <-2.8mm,-2.9mm>*{};<-4.7mm,-4.9mm>*{}**@{.},
    <0.39mm,-0.39mm>*{};<3.3mm,-4.0mm>*{^3}**@{},
    <-2.0mm,-2.8mm>*{};<0.5mm,-6.7mm>*{^2}**@{},
    <-2.8mm,-2.9mm>*{};<-5.2mm,-6.7mm>*{^1}**@{},
 \end{xy}\Ea
\ - \
 \Ba{c}\begin{xy}
 <0mm,0mm>*{\circ};<0mm,0mm>*{}**@{},
 <0mm,0.69mm>*{};<0mm,3.0mm>*{}**@{.},
 <0.39mm,-0.39mm>*{};<2.4mm,-2.4mm>*{}**@{.},
 <-0.35mm,-0.35mm>*{};<-1.9mm,-1.9mm>*{}**@{.},
 <2.4mm,-2.4mm>*{\circ};<-2.4mm,-2.4mm>*{}**@{},
 <2.0mm,-2.8mm>*{};<0mm,-4.9mm>*{}**@{.},
 <2.8mm,-2.9mm>*{};<4.7mm,-4.9mm>*{}**@{.},
    <0.39mm,-0.39mm>*{};<-3mm,-4.0mm>*{^1}**@{},
    <-2.0mm,-2.8mm>*{};<0mm,-6.7mm>*{^2}**@{},
    <-2.8mm,-2.9mm>*{};<5.2mm,-6.7mm>*{^3}**@{},
 \end{xy}\Ea=0,\ \ \ \ \ \
\Ba{c} \begin{xy}
 <0mm,2.47mm>*{};<0mm,-0.5mm>*{}**@{.},
 <0.5mm,3.5mm>*{};<2.2mm,5.2mm>*{}**@{.},
 <-0.48mm,3.48mm>*{};<-2.2mm,5.2mm>*{}**@{.},
 <0mm,3mm>*{\circ};<0mm,3mm>*{}**@{},
  <0mm,-0.8mm>*{\circ};<0mm,-0.8mm>*{}**@{},
<0mm,-0.8mm>*{};<-2.2mm,-3.5mm>*{}**@{.},
 <0mm,-0.8mm>*{};<2.2mm,-3.5mm>*{}**@{.},
     <0.5mm,3.5mm>*{};<2.8mm,5.7mm>*{^2}**@{},
     <-0.48mm,3.48mm>*{};<-2.8mm,5.7mm>*{^1}**@{},
   <0mm,-0.8mm>*{};<-2.7mm,-5.2mm>*{^1}**@{},
   <0mm,-0.8mm>*{};<2.7mm,-5.2mm>*{^2}**@{},
\end{xy}\Ea
\ - \
\Ba{c}\begin{xy}
 <0mm,0mm>*{\circ};<0mm,0mm>*{}**@{},
 <0mm,-0.49mm>*{};<0mm,-3.0mm>*{}**@{.},
 <-0.5mm,0.5mm>*{};<-3mm,2mm>*{}**@{.},
 <-3mm,2mm>*{};<0mm,4mm>*{}**@{.},
 <0mm,4mm>*{\circ};<-2.3mm,2.3mm>*{}**@{},
 <0mm,4mm>*{};<0mm,7.4mm>*{}**@{.},
<0mm,0mm>*{};<2.2mm,1.5mm>*{}**@{.},
 <6mm,0mm>*{\circ};<0mm,0mm>*{}**@{},
 <6mm,4mm>*{};<3.8mm,2.5mm>*{}**@{.},
 <6mm,4mm>*{};<6mm,7.4mm>*{}**@{.},
 <6mm,4mm>*{\circ};<-2.3mm,2.3mm>*{}**@{},
 <0mm,4mm>*{};<6mm,0mm>*{}**@{.},
<6mm,4mm>*{};<9mm,2mm>*{}**@{.},
<6mm,0mm>*{};<9mm,2mm>*{}**@{.},
<6mm,0mm>*{};<6mm,-3mm>*{}**@{.},
   <-1.8mm,2.8mm>*{};<0mm,7.8mm>*{^1}**@{},
   <-2.8mm,2.9mm>*{};<0mm,-4.3mm>*{_1}**@{},
<-1.8mm,2.8mm>*{};<6mm,7.8mm>*{^2}**@{},
   <-2.8mm,2.9mm>*{};<6mm,-4.3mm>*{_2}**@{},
 \end{xy}
\Ea=0
\right.
\Eeq
which are not quadratic in the properadic sense (it is proven, however, in \cite{MV} that $\Assb$ is {\em homotopy Koszul}). A minimal resolution, $(\Assb_\infty,\delta)$ of $\Assb$ exists and is generated by a relatively ``small"
$\bS$-bimodule
 $ A=\{ A(m,n)\}_{m,n\geq 1, m+n\geq 3}$,
\[
 A(m,n):= \K[\bS_m]\ot \K[\bS_n][m+n-3]=\mbox{span}\left\langle
\Ba{c}
\resizebox{19mm}{!}{\begin{xy}
 <0mm,0mm>*{\circ};<0mm,0mm>*{}**@{},
 <0mm,0mm>*{};<-8mm,5mm>*{}**@{.},
 <0mm,0mm>*{};<-4.5mm,5mm>*{}**@{.},
 <0mm,0mm>*{};<-1mm,5mm>*{\ldots}**@{},
 <0mm,0mm>*{};<4.5mm,5mm>*{}**@{.},
 <0mm,0mm>*{};<8mm,5mm>*{}**@{.},
   <0mm,0mm>*{};<-10.5mm,5.9mm>*{^{\tau(1)}}**@{},
   <0mm,0mm>*{};<-4mm,5.9mm>*{^{\tau(2)}}**@{},
   <0mm,0mm>*{};<10.0mm,5.9mm>*{^{\tau(m)}}**@{},
 <0mm,0mm>*{};<-8mm,-5mm>*{}**@{.},
 <0mm,0mm>*{};<-4.5mm,-5mm>*{}**@{.},
 <0mm,0mm>*{};<-1mm,-5mm>*{\ldots}**@{},
 <0mm,0mm>*{};<4.5mm,-5mm>*{}**@{.},
 <0mm,0mm>*{};<8mm,-5mm>*{}**@{.},
   <0mm,0mm>*{};<-10.5mm,-6.9mm>*{^{\sigma(1)}}**@{},
   <0mm,0mm>*{};<-4mm,-6.9mm>*{^{\sigma(2)}}**@{},
   <0mm,0mm>*{};<10.0mm,-6.9mm>*{^{\sigma(n)}}**@{},
 \end{xy}}\Ea
\right\rangle_{\tau\in \bS_n\atop \sigma\in \bS_m},
\]
The differential $\delta$ in $\Assb_\infty$ is not
quadratic, and its explicit value on generic $(m,n)$-corolla is not
known at present, but we can  (and will) assume from now on that $\delta$
preserves
the {\em path grading}\, of  $\Assb_\infty$ (which associates to any graph, $G$, from
$\Assb_\infty$ the total number of directed paths connecting input legs of $G$ to the output ones, see \cite{MaVo} for more details).

\mip

 Let $V$ be a $\Z$-graded
vector space over a field $\K$ of characteristic zero. The associated symmetric tensor algebra
$\f_V:= {\odot^{\bullet}} V= \oplus_{n\geq 0} \odot^n V$
comes equipped  with the standard graded commutative and
co-commutative bialgebra structure, i.e.\ there is a non-trivial representation,
\Beq\label{2: rho_0}
\rho_0: \Assb \lon \cE nd_{\f_V}.
\Eeq
According to \cite{MV}, the deformation complex
$$
C_{GS}^\bu\left(\f_V,\f_V\right)=\Def\left(\Assb \stackrel{\rho_0}{\lon} \cE nd_{\f_V}\right)\simeq \prod_{m,n\geq 1}\Hom(\f_V^{\ot m}, \f_V^{\ot n})[2-m-n]
$$
and its polydifferential subcomplex $C_{poly}^\bu\left(\f_V,\f_V\right)$ come equipped with a $\caL ie_\infty$ algebra structure, $
\left\{\mu_n: \wedge^n C_{GS}^\bu(\f_V,\f_V)\lon C_{GS}^\bu(\f_V,\f_V)[2-n]\right\}_{n\geq 1},
$
such that $\mu_1$ coincides precisely with the Gerstenhaber-Shack differential.
According to \cite{GS}, the cohomology of the complex $(C_{GS}^\bu(\f_V,\f_V), \mu_1)$ is precisely
the vector space $\fg_V$; moreover, it is not hard to see that the operation $\mu_2$ induces
the Lie brackets $\{\ ,\ \}$ in $\fg_V$. As we show in this paper, the set  of $\caL ie_\infty$ quasi-isomorphisms,
\Beq\label{3 formaily maps}
\left\{F: \left(\fg_V[2],\{\ ,\ \}\right) \lon \left(C_{GS}^\bu(\f_V,\f_V), \mu_\bu\right)\right\},
\Eeq
is always non-empty.

\bip

\bip

{\Large
\section{\bf A prop governing formality maps for Lie bialgebras}
}

\bip

\subsection{An operad of directed graphs}
Let $G_{n,l}$ be a set of directed graphs $\Ga$ with $n$ vertices and $l$  edges such that
some bijections $V(\Ga)\rar [n]$ and $E(\Ga)\rar [l]$ are fixed, i.e.\ every edges and every vertex of $\Ga$ has a fixed numerical label. There is
a natural right action of the group $\bS_n \times  \bS_l$ on the set $G_{n,l}$ with $\bS_n$ acting by relabeling the vertices and  $\bS_l$ by relabeling the
edges. 

\sip

For each fixed integer $d$, a collection of $\bS_n$-modules,
$$
d\cG ra_{d}=\left\{d\cG ra_d(n)= \prod_{l\geq 0} \K \langle G_{n,l}\rangle \ot_{ \bS_l}  \sgn_l^{\ot |d-1|} [l(d-1)]   \right\}_{n\geq 1}
$$
 is an operad with respect to the following operadic composition,
$$
\Ba{rccc}
\circ_i: &  d\cG ra_d(n) \times d\cG ra_d(m) &\lon & d\cG ra_d(m+n-1),  \ \ \forall\ i\in [n]\\
         &       (\Ga_1, \Ga_2) &\lon &      \Ga_1\circ_i \Ga_2,
\Ea
$$
where  $\Ga_1\circ_i \Ga_2$ is defined by substituting the graph $\Ga_2$ into the $i$-labeled vertex $v_i$ of $\Ga_1$ and taking a sum over  re-attachments of dangling edges (attached before to $v_i$) to vertices of $\Ga_2$
in all possible ways.

\sip

The operad of directed graphs $d\cG ra_d$ contains a suboperad $\cG ra^{or}$ spanned by graphs
with no {\em closed}\, paths of directed edges (wheels); we call such graphs {\em oriented}.

\sip

For any operad $\cP=\{\cP(n)\}_{n\geq 1}$  in the category of graded vector spaces,
the linear the map
$$
\Ba{rccc}
[\ ,\ ]:&  \sP \ot \sP & \lon & \sP\\
& (a\in \cP(n), b\in \cP(m)) & \lon &
[a, b]:= \sum_{i=1}^n a\circ_i b - (-1)^{|a||b|}\sum_{i=1}^m b\circ_i a\ \in \cP(m+n-1)
\Ea
$$
makes a graded vector space
$
\sP:= \prod_{n\geq 1}\cP(n)$
into a Lie algebra \cite{KM}; moreover, these brackets induce a Lie algebra structure on the subspace
of invariants
$
\sP^\bS:=  \prod_{n\geq 1}\cP(n)^{\bS_n}$. In particular,
the graded vector spaces
$$
\mathsf{dFGC}_{d}:= \prod_{n\geq 1} d\cG ra_{d}(n)^{\bS_n}[d(1-n)], \ \ \ \ \
\mathsf{FGC}^{or}_{d}:= \prod_{n\geq 1} \cG ra^{or}_{d}(n)^{\bS_n}[d(1-n)]
$$
are Lie algebra with respect to the above Lie brackets, and as such they can be identified
with the deformation complexes $\Def(\caL ie_d\stackrel{0}{\rar} d\cG ra_{d})$ and, respectively,
 $\Def(\caL ie_d\stackrel{0}{\rar} \cG ra_{d}^{or})$ of athe zero morphism. Hence non-trivial Maurer-Cartan elements of $(\mathsf{dFGC}_{d}/\mathsf{FGC}_d^{or}, [\ ,\ ])$ give us non-trivial morphisms of operads
$$
f:\caL ie_d {\lon} d\cG ra_{d}\ \ \ \mathrm{and,\ respectively}\ \ \ \ f:\caL ie_d {\lon} \cG ra_{d}^{or}
$$
 One such non-trivial morphism $f$ is given explicitly on the generator of $\caL ie_{d}$ by \cite{Wi}
\Beq\label{2:  map from Lie to dgra}
f \left(\Ba{c}\begin{xy}
 <0mm,0.66mm>*{};<0mm,3mm>*{}**@{-},
 <0.39mm,-0.39mm>*{};<2.2mm,-2.2mm>*{}**@{-},
 <-0.35mm,-0.35mm>*{};<-2.2mm,-2.2mm>*{}**@{-},
 <0mm,0mm>*{\circ};<0mm,0mm>*{}**@{},
   <0.39mm,-0.39mm>*{};<2.9mm,-4mm>*{^{_2}}**@{},
   <-0.35mm,-0.35mm>*{};<-2.8mm,-4mm>*{^{_1}}**@{},
\end{xy}\Ea\right)=
\Ba{c}\resizebox{6.3mm}{!}{\xy
(0,1)*+{_1}*\cir{}="b",
(8,1)*+{_2}*\cir{}="c",
\ar @{->} "b";"c" <0pt>
\endxy}
\Ea  - (-1)^d
\Ba{c}\resizebox{7mm}{!}{\xy
(0,1)*+{_2}*\cir{}="b",
(8,1)*+{_1}*\cir{}="c",
\ar @{->} "b";"c" <0pt>
\endxy}
\Ea=:\xy
 (0,0)*{\bullet}="a",
(5,0)*{\bu}="b",
\ar @{->} "a";"b" <0pt>
\endxy
\Eeq
Note that elements of $\mathsf{dFGC}_{d}$ can be identified with graphs from $\caD\cG ra_d$ whose vertices' labels are symmetrized (for $d$ even) or skew-symmetrized (for $d$ odd) so that in pictures we can forget about labels of vertices  and denote them by unlabelled black bullets as in the formula above. Note also that graphs from  $\mathsf{dFGC}_{d}$ come equipped with an orientation, $or$, which is a choice of ordering of edges (for $d$ even) or a choice of ordering of vertices (for $d$ odd) up to an even permutation on both cases. Thus every graph $\Ga\in \mathsf{dFGC}_{d}$  has at most two different orientations, $pr$ and $or^{opp}$, and one has
the standard relation, $(\Ga, or)=-(\Ga, or^{opp})$; as usual, the data $(\Ga, or)$ is abbreviate by $\Ga$ (with some choice of orientation implicitly assumed).  Note that the homological degree of graph $\Ga$ from $\mathsf{dFGC}_{d}$ is given by
$
|\Ga|=d(\# V(\Ga) -1) + (1-d) \# E(\Ga).
$

\sip

The above morphism (\ref{2:  map from Lie to dgra}) makes
 $(\mathsf{dFGC}_{d}, [\ ,\ ])$ (resp., $(\mathsf{FGC}^{or}_{d}, [\ ,\ ])$ into a {\em differential}\, Lie algebra with the differential
 $$
 \delta:= [\xy
 (0,0)*{\bullet}="a",
(5,0)*{\bu}="b",
\ar @{->} "a";"b" <0pt>
\endxy ,\ ].
 $$
 The dg Lie algebra  $\mathsf{dFGC}_{d}$ (resp., $\mathsf{FGC}^{or}_{d}$) contains a  dg subalgebra $\mathsf{df{GC}}_{d}$ (resp.\ $\mathsf{fGC}^{or}_{d}$) spanned by graphs
with at least bivalent vertices and with no bivalent vertices of the form  $\xy
(0,0)*{}="a",
(4,0)*{\bu}="b",
(8,0)*{}="c",
\ar @{->} "a";"b" <0pt>\ar @{->} "b";"c" <0pt>
\endxy$.
 It was proven in \cite{Wi2} that
$$
H^\bu(\mathsf{dFGC}_{d})= \mathsf{dfGC}_{d}\ \ \ \text{and} \ \ \ H^\bu(\mathsf{FGC}^{or}_{d})= \mathsf{fGC}^{or}_{d}
$$
 Furthermore, $\mathsf{dfGC}_{d}$ and $\mathsf{fGC}^{or}_{d}$ contains dg Lie subalgebras $\mathsf{dGC}_{d}$ and respectively $\mathsf{GC}^{or}_d$  spanned by {\em connected}\,  graphs, and one has
$$
\mathsf{dfGC}_{d}= \odot^{\bu\geq 1}\left(\mathsf{dGC}_{d}[-d]\right)[d]
\ \ \ \text{and}\ \ \
\mathsf{fGC}^{or}_{d}= \odot^{\bu\geq 1}\left(\mathsf{GC}^{or}_{d}[-d]\right)[d].
$$
 Moreover, there is an isomorphism of graded vector spaces \cite{Wi2},
$$
H^\bu(\mathsf{GC}^{or}_{d+1})=H^\bu(\mathsf{dGC}_d),
$$
and an isomorphism of Lie algebras \cite{Wi},
$$
H^0(\mathsf{dGC}_{2})=\fg\fr\ft_1,
$$
where $\fg\fr\ft_1$ is the Lie algebra of the Grothendieck-teichm\"u ller group $GRT_1$ introduced by Drinfeld \cite{D2} in the context of deformation quantization of Lie bialgebras. In particular, one has a remarkable isomorphism of Lie algebras,
$
H^0(\mathsf{GC}_3^{or})=\fg\fr\ft_1$. Moreover \cite{Wi2}, $H^i(\mathsf{GC}_3^{or})=0$ for $i\leq 2$ and $H^{-1}(\mathsf{GC}_3^{or})$
is a 1-dimensional space generated by the graph $
\Ba{c}\resizebox{4mm}{!}{   \xy
   \ar@/^0.6pc/(0,-5)*{\bullet};(0,5)*{\bullet}
   \ar@/^{-0.6pc}/(0,-5)*{\bullet};(0,5)*{\bullet}
 \endxy}\Ea
$. Hence
\Beqrn
H^0(\mathsf{fGC}^{or}_{3}) &=& \sum_{k\geq 1}\sum_{3-3k=i_1+\ldots+i_k\atop i_1,\ldots,i_k\geq -1} H^{i_1}\left(\mathsf{GC}^{or}_{3}\right)\wedge \ldots \wedge H^{i_1}\left(\mathsf{GC}^{or}_{3}\right) \\
&=& H^0(\mathsf{GC}_3^{or})\\
&=&\fg\fr\ft_1.
\Eeqrn

Let $\wLB_\infty$ be the genus completion
of the properad $\LB_\infty$. Then, according to \cite{MW2}, there is a canonical morphism of dg Lie algebras
$$
f: \mathsf{GC}_3^{or}\to \Der(\wLB_\infty)
$$
which is a quasi-isomorphism up to one class corresponding to the standard rescaling derivation of $\wLB_\infty$, multiplying the generator of arity $(m,n)$ by $m+n-2$.
Here $\Der(\wLB_\infty)$ is the dg Lie algebra of derivations of $\wLB_\infty$ as a {\em properad}. It is useful to add to
$\mathsf{GC}_{d}^{or}$  one extra generator in degree zero corresponding to the rescaling derivation.
So {\em from now on we are working with the extension}
$$
\GC_{d}^\uparrow:= \mathsf{GC}_{d}^{or}\oplus \K[0]
$$
of this graph complex, where the Lie bracket is extended such that the bracket with the rescaling class multiplies a graph by twice the number of loops.
We also consider the full graph complex,
$$
\mathsf{fGC}_{d}^\uparrow:=\odot^{\bu\geq 1}\left(\mathsf{GC}_{d}^\uparrow[-d]\right)[d],
$$
Then the map $f$ above extends to a quasi-isomorphism,
$$
f: \mathsf{{GC}}_3^\uparrow\to \Der(\wLB_\infty)
$$
 and
$$
H^0(\mathsf{fGC}^\uparrow_{3})=H^0(\mathsf{GC}^\uparrow_{3})=H^0(\mathsf{GC}_{3}^{or})
\oplus \K=\fg\fr\ft_1\oplus \K=\fg\fr\ft,
$$
is the Lie algebra of the full Grothendieck-Teichm\"uller group $GRT=GRT_1\rtimes \K^*$.
This extension becomes especially useful when one works with props rather than with properads. The properad  $\wLB_\infty$ has an associated dg prop which is usually denoted by the same letter; we denote the dg Lie algebra of derivations of $\wLB_\infty$ as a {\em prop}\, by $\Der(\wLB_\infty)_{\mathrm{prop}}$. As the functor from the category of properads to the category of props is exact,  the above result immediately implies existence of quasi-isomorphism  of dg Lie algebras
 \Beq\label{4: morphism from fGC_3 to derLieb-prop}
F: \mathsf{fGC}_{3}^\uparrow\to \Der(\wLB_\infty)_{\mathrm{prop}}
\Eeq
Therefore we conclude
\Beq\label{4: Formula for H^0 DerLieb-prop}
H^0\left(\Der(\wLB_\infty)_{\mathrm{prop}}\right)=H^0(\mathsf{fGC}_{3}^\uparrow) =\fg\fr\ft.
\Eeq

\subsection{A canonical representation of $\cG ra_3^{or}$ in $\fg_V$}\label{4: subsec on canonical repr of Gra}
For {\em any}\, graded vector space the operad  $\cG ra_3^{or}$ has a canonical representation in the associated vector space $\fg_V$,
\Beq\label{3: Gra representation in g_V}
\Ba{rccc}
\rho: & \cG ra^{or}_3(n) & \lon & \cE  nd_{\fg_V}(n)=
\Hom( \fg_V^{\ot n},\fg_V)\\
      & \Ga &\lon & \Phi_\Ga
\Ea
\Eeq
given by the formula,
$$
\Phi_\Ga(\ga_1,\ldots, \ga_n) :=\mu\left(\prod_{e\in E(\Ga)}
\Delta_e \left(\ga_1(\psi)\ot \ga_2(\psi)\ot \ldots\ot
\ga_n(\psi) \right)\right)
$$
where, for an edge $e=\Ba{c}\xy
(0,2)*{_{a}},
(6,2)*{_{b}},
 (0,0)*{\bullet}="a",
(6,0)*{\bu}="b",
\ar @{->} "a";"b" <0pt>
\endxy \Ea$ connecting a vertex labeled by $a\in [n]$ and to a vertex labelled by $b\in [n]$, we set
$$
\Delta_e \left(\ga_1\ot \ga_2\ot \ldots\ot
\ga_n \right)=
\left\{\Ba{cc}\underset{i\in I}{\sum}(-1)^{|\psi^i|(|\ga_a| + |\ga_{a+1}|+\ldots+ |\ga_{b-1}|)} \ga_1\ot ...\ot \frac{\p\ga_a}{\p \psi_i}\ot ...
\ot \frac{\p\ga_b}{\p \psi^i}\ot ... \ot
\ga_n & \mbox{for} \ a< b   \\
\underset{i\in I}{\sum}(-1)^{|\psi_i|(|\ga_b| + |\ga_{b+1}|+\ldots+ |\ga_{a-1}| +1 )} \ga_1\ot ... \ot \frac{\p\ga_b}{\p \psi^i}\ot ...
\ot \frac{\p\ga_a}{\p \psi_i}\ot ... \ot
\ga_n & \mbox{for} \ b< a
\Ea\right.
$$
and where  $\mu$ is the multiplication map,
$$
\Ba{rccc}
\mu:&   \fg_V^{\ot n} & \lon & \fg_V\\
   & \ga_1\ot \ga_2\ot \ldots \ot \ga_n &\lon & \ga_1 \ga_2 \cdots
   \ga_n.
\Ea
$$
If $V$ is {\em finite}\, dimensional, then the above formulae can be used to define a  representation of the operad  $d\cG ra_3$ in $\fg_V$.


\subsection{A prop of directed graphs}
Let $k\geq 0$, $m\geq 1$ and $n\geq 1$ be  integers.
Let  $\sG_{k;m,n}$ be a set of directed graphs  consists, by definition,
of directed oriented graphs $\Ga$ with $k$ vertices called {\em internal}, $m$ vertices called
{\em in}-vertices and $n$ vertices called {\em out}-vertices (i.e.\
there is a partition)
satisfying the following conditions
\Bi
\item[(i)] the set of vertices is partitioned,
 $
V(\Ga):=V_{int}(\Ga)\sqcup V_{in}(\Ga)\sqcup V_{out}(\Ga),
$
into three subsets of the cardinalities,
$$
\# V_{int}(\Ga)=k,\ \ \ \# V_{in}(\Ga)=m,\ \ \ \# V_{out}(\Ga)=n;
$$
elements of $V_{int}(\Ga)$ are called  {\em internal}\, vertices,
elements of  $V_{in}(\Ga)$ are called {\em in}-vertices, and
elements of  $V_{out}(\Ga)$ are called {\em out}-vertices;

\item[(iii)] the vertices are labelled via isomorphisms,
$$
i_{int}: V_{int}(\Ga)\rar [k],\ \ \ i_{in}: V_{in}(\Ga)\rar [m],\ \ \
i_{out}: V_{int}(\Ga)\rar [n];
$$
\item[(iv)] every in-vertex can have only outgoing edges (called {\em in-legs}), while every out-vertex can have only ingoing edges (called {\em in-legs});
\item[(v)] there are no edges connecting in-vertices to out-vertices, i.e.\
the set of all edges is  partitioned into a disjoint union, $E(\Ga):=E_{int}(\Ga)\sqcup E_{in}(\Ga)\sqcup E_{out}(\Ga)$,
where $E_{in}(\Ga)$ is the set of in-legs, $E_{out}(\Ga)$  the set of out-legs,
and $E_{int}(\Ga)$ the set of {\em internal} \, edges which connect internal vertices to internal ones;
\item[(vi)] the set of legs $E_{in}(\Ga)\sqcup E_{out}(\Ga)$ is totally ordered up to an even permutation.
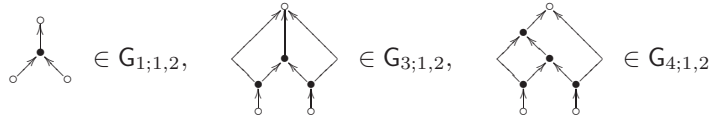
\Ei
$$
\Ba{c}\resizebox{10mm}{!}{\xy
(0,13)*{\circ}="0",
 (0,7)*{\bu}="a",
(-5,2)*{\circ}="b_1",
(5,2)*{\circ}="b_2",
(-8,-2)*{}="c_1",
(-2,-2)*{}="c_2",
(2,-2)*{}="c_3",
\ar @{->} "a";"0" <0pt>
\ar @{<-} "a";"b_1" <0pt>
\ar @{<-} "a";"b_2" <0pt>
\endxy}
\Ea\in \sG_{1;1,2},
\ \ \
\Ba{c}\resizebox{14mm}{!}{\xy
(0,17)*{\circ}="u",
 (0,7)*{\bu}="a",
(-10,7)*{}="L",
(10,7)*{}="R",
(-5,2)*{\bu}="b_1",
(5,2)*{\bu}="b_2",
(-5,-3)*{\circ}="c_1",
(5,-3)*{\circ}="c_3",
\ar @{<-} "a";"b_1" <0pt>
\ar @{<-} "a";"b_2" <0pt>
\ar @{<-} "b_1";"c_1" <0pt>
\ar @{<-} "b_2";"c_3" <0pt>
\ar @{-} "b_1";"L" <0pt>
\ar @{<-} "u";"L" <0pt>
\ar @{-} "b_2";"R" <0pt>
\ar @{->} "R";"u" <0pt>
\ar @{->} "a";"u" <0pt>
\endxy}
\Ea\in \sG_{3;1,2},
\ \ \
\Ba{c}\resizebox{14mm}{!}{\xy
(0,17)*{\circ}="u",
(-5,12)*{\bu}="0",
 (0,7)*{\bu}="a",
(-10,7)*{}="L",
(10,7)*{}="R",
(-5,2)*{\bu}="b_1",
(5,2)*{\bu}="b_2",
(-5,-3)*{\circ}="c_1",
(5,-3)*{\circ}="c_3",
\ar @{->} "a";"0" <0pt>
\ar @{<-} "a";"b_1" <0pt>
\ar @{<-} "a";"b_2" <0pt>
\ar @{<-} "b_1";"c_1" <0pt>
\ar @{<-} "b_2";"c_3" <0pt>
\ar @{-} "b_1";"L" <0pt>
\ar @{<-} "0";"L" <0pt>
\ar @{-} "b_2";"R" <0pt>
\ar @{->} "R";"u" <0pt>
\ar @{->} "0";"u" <0pt>
\endxy}
\Ea\in \sG_{4;1,2}
 $$
where  in- and out-vertices are
depicted by small white circles, while internal vertices by small black ones.
 (and we do not show numerical labels of the vertices).

\sip

Define a graded vector space
$$
\cB\cG ra^{or}(k;m,n):=\R\langle\sG_{k;m,n}\rangle
$$
by assigning to each graph $\Ga\in \sG_{k;m,n}$ homological degree,
$$
|\Ga|=-2\#E_{int}(\Ga) -\#E_{in}(\Ga)-\# E_{out}(\Ga),
$$
i.e.\ by assigning to each internal edge degree $-2$ and to each leg degree $-1$. Consider next an $\bS$-bimodule,
$$
\cB\cG ra^{or}=\left\{\cB\cG ra^{or}(m,n):=\sum_{k\geq 0} \cB\cG ra^{or}(k;m,n)\right\}.
$$
We claim that it has a natural structure of a prop with
\Bi
\item horizontal composition
$$
\Ba{rccc}
\boxtimes: & \cB\cG ra^{or}(k;m,n) \ot\cB\cG ra^{or}(k';m',n') &\lon & \cB\cG ra^{or}(k+k';m+m',n+n')\\
 & \Ga\ot \Ga' & \lon & \Ga\boxtimes \Ga'
\Ea
$$
by taking the disjoint union of the graphs $\Ga$ and $\Ga'$ and relabelling
in-, out- and internal vertices of $\Ga'$ accordingly;
\item vertical composition,
$$
\Ba{rccc}
\circ: & \cB\cG ra^{or}(k;m,n) \ot\cB\cG ra^{or}(k';n,l) &\lon & \cB\cG ra^{or}(k+k';m,l)\\
 & \Ga\ot \Ga' & \lon & \Ga\circ\Ga',
\Ea
$$
by the following three step procedure: (a) erase all $m$ out-vertices of $\Ga$ and all $m$ in-vertices of $\Ga'$, (b) take a sum over all possible ways of attaching the hanging out-legs of $\Ga$ to hanging in-legs of $\Ga'$ as well as to out-vertices of $\Ga'$, and also attaching the remaining in-legs of $\Ga'$
to in-vertices of  $\Ga$, (c) relabel internal vertices of $\Ga'$ by adding the value $k$ to each label.
\Ei
For example, a vertical composition of the following two graphs,
$$
\Ba{rccc}
\circ: & \cB\cG ra^{or}(1;2,1) \ot\cB\cG ra^{or}(1;1,2) &\lon & \cB\cG
ra^\uparrow(2;2,2)\\
&
\Ba{c}\resizebox{10mm}{!}
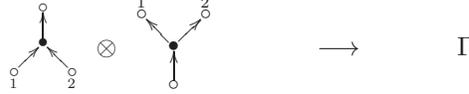
{ \xy
(-5,0)*{_1},
(5,0)*{_2},
(0,13)*{\circ}="0",
 (0,7)*{\bu}="a",
(-5,2)*{\circ}="b_1",
(5,2)*{\circ}="b_2",
\ar @{->} "a";"0" <0pt>
\ar @{<-} "a";"b_1" <0pt>
\ar @{<-} "a";"b_2" <0pt>
\endxy}\Ea
\ot
\Ba{c}
\resizebox{10mm}{!}{\xy
(-5,15)*{_1},
(5,15)*{_2},
(0,2)*{\circ}="0",
 (0,8)*{\bu}="a",
(-5,13)*{\circ}="b_1",
(5,13)*{\circ}="b_2",
\ar @{<-} "a";"0" <0pt>
\ar @{->} "a";"b_1" <0pt>
\ar @{->} "a";"b_2" <0pt>
\endxy} \Ea
 &\lon &  \Ga
\Ea
$$
is given by the following sum (cf.\ \S {\ref{2: Definition of DP prop}})
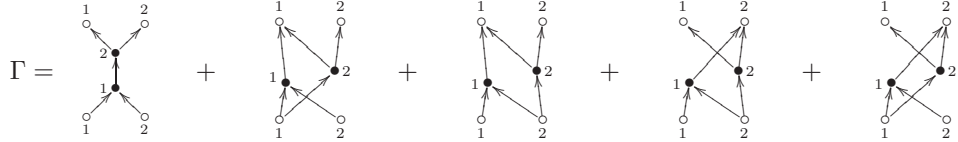
$$
\Ga=
\Ba{c}\resizebox{10mm}{!}{
\xy
(-5,0)*{_1},
(5,0)*{_2},
(-5,20)*{^1},
(5,20)*{^2},
(-2,13)*{_2},
(-2,7)*{_1},
(0,13)*{\bu}="0",
 (0,7)*{\bu}="a",
(-5,2)*{\circ}="b_1",
(5,2)*{\circ}="b_2",
(-5,18)*{\circ}="u_1",
(5,18)*{\circ}="u_2",
\ar @{->} "a";"0" <0pt>
\ar @{<-} "a";"b_1" <0pt>
\ar @{<-} "a";"b_2" <0pt>
\ar @{->} "0";"u_1" <0pt>
\ar @{->} "0";"u_2" <0pt>
\endxy}
\Ea\ \ \ +\ \ \
\Ba{c}\resizebox{12mm}{!}{
\xy
(-5,0)*{_1},
(5,0)*{_2},
(-5,20)*{^1},
(5,20)*{^2},
(6,10)*{_2},
(-6,9)*{_1},
(4,10)*{\bu}="0",
 (-4,8)*{\bu}="a",
(-5,2)*{\circ}="b_1",
(5,2)*{\circ}="b_2",
(-5,18)*{\circ}="u_1",
(5,18)*{\circ}="u_2",
\ar @{->} "b_1";"0" <0pt>
\ar @{<-} "u_1";"a" <0pt>
\ar @{<-} "a";"b_1" <0pt>
\ar @{<-} "a";"b_2" <0pt>
\ar @{->} "0";"u_1" <0pt>
\ar @{->} "0";"u_2" <0pt>
\endxy}
\Ea\ \ \ +\ \ \
\Ba{c}\resizebox{12mm}{!}{
\xy
(-5,0)*{_1},
(5,0)*{_2},
(-5,20)*{^1},
(5,20)*{^2},
(6,10)*{_2},
(-6,8)*{_1},
(4,10)*{\bu}="0",
 (-4,8)*{\bu}="a",
(-5,2)*{\circ}="b_1",
(5,2)*{\circ}="b_2",
(-5,18)*{\circ}="u_1",
(5,18)*{\circ}="u_2",
\ar @{->} "b_2";"0" <0pt>
\ar @{<-} "u_1";"a" <0pt>
\ar @{<-} "a";"b_1" <0pt>
\ar @{<-} "a";"b_2" <0pt>
\ar @{->} "0";"u_1" <0pt>
\ar @{->} "0";"u_2" <0pt>
\endxy}
\Ea\ \ \ +\ \ \
\Ba{c}\resizebox{12mm}{!}{
\xy
(-5,0)*{_1},
(5,0)*{_2},
(-5,20)*{^1},
(5,20)*{^2},
(6,10)*{_2},
(-6,8)*{_1},
(4,10)*{\bu}="0",
 (-4,8)*{\bu}="a",
(-5,2)*{\circ}="b_1",
(5,2)*{\circ}="b_2",
(-5,18)*{\circ}="u_1",
(5,18)*{\circ}="u_2",
\ar @{->} "b_2";"0" <0pt>
\ar @{<-} "u_2";"a" <0pt>
\ar @{<-} "a";"b_1" <0pt>
\ar @{<-} "a";"b_2" <0pt>
\ar @{->} "0";"u_1" <0pt>
\ar @{->} "0";"u_2" <0pt>
\endxy}
\Ea\ \ \ +\ \ \
\Ba{c}\resizebox{12mm}{!}{
\xy
(-5,0)*{_1},
(5,0)*{_2},
(-5,20)*{^1},
(5,20)*{^2},
(6,10)*{_2},
(-6,8)*{_1},
(4,10)*{\bu}="0",
 (-4,8)*{\bu}="a",
(-5,2)*{\circ}="b_1",
(5,2)*{\circ}="b_2",
(-5,18)*{\circ}="u_1",
(5,18)*{\circ}="u_2",
\ar @{->} "b_1";"0" <0pt>
\ar @{<-} "u_2";"a" <0pt>
\ar @{<-} "a";"b_1" <0pt>
\ar @{<-} "a";"b_2" <0pt>
\ar @{->} "0";"u_1" <0pt>
\ar @{->} "0";"u_2" <0pt>
\endxy}\Ea
$$
Note that the $\bS$-bimodule $\{\cB \cG ra^{or}(0;m,n)\}$ is a subprop
of $\cB \cG ra^\uparrow$ isomorphic to $\CB$.

\sip

Note also that we can also define compositions
$$
\circ: \cG ra^{or}_3(p_1) \ot \cG ra^{or}_3(p_2)\ot\ldots \ot
 \cG ra^{or}_3(p_k)\ot
\cB\cG ra^{or}(k,m,n)\lon \cB\cG ra^{or}(p_1+p_2+\ldots+p_k,m,n)
$$
by substituting $k$ graphs from $\cG ra_3^{or}$ into internal vertices of a graph from $\cB\cG ra^{or}(k,m,n)$ and redistributing edges in exactly the same way as in the definition of the operadic composition in $\cG ra_3^{or}$ (and setting to zero all graphs which do not satisfy  the condition that every black vertex has at least one incoming edge and at least one outgoing edge). For example,
\Beq\label{3: Example of prop coomp}
\Ba{rccc}
\bu: &\cB\cG ra(1;1,1)\ot \cG ra^{or}(2)  &\lon & \cB\cG ra(2;1,1)
\\
& 
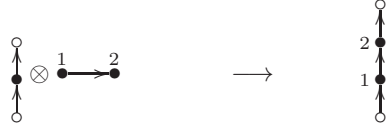
\xy
 (0,0)*{\bullet}="a",
(0,5)*{\circ}="u",
(0,-5)*{\circ}="d",
\ar @{->} "d";"a" <0pt>
\ar @{->} "a";"u" <0pt>
\endxy \ot
\xy
(0,2)*{_{1}},
(7,2)*{_{2}},
 (0,0)*{\bullet}="a",
(7,0)*{\bu}="b",
\ar @{->} "a";"b" <0pt>
\endxy &\lon & \xy
(-2,0)*{_{1}},
(-2,5)*{_{2}},
 (0,0)*{\bullet}="a",
(0,5)*{\bu}="b",
(0,10)*{\circ}="u",
(0,-5)*{\circ}="d",
\ar @{->} "d";"a" <0pt>
\ar @{->} "a";"b" <0pt>
\ar @{->} "b";"u" <0pt>
\endxy
\Ea
\Eeq

If we denote by the same symbol  $\cG ra_3^{or}$ the {\em prop}\, generated
by the operad $\cG ra_3^{or}$ then we conclude that the data
$$
\cB ra^{or}:= \left\{\cB\cG ra^{or}, \cG ra_3^{or}\right\}
$$
generates a {\em 2-coloured}\, prop in the category of graded vector spaces.

\mip

Abusing notations, let us denote by $\caL ie_3$ the prop generated by the operad
$\caL ie\{2\}$. Then formula (\ref{2:  map from Lie to dgra}) says that there is a canonical morphism of props,
\Beq\label{3: i_1: Lie to Bra}
i_1: \caL ie_3 \lon \cB ra^{or}.
\Eeq
{ There is also a morphism of props
\Beq\label{3: i_2 AssB to Bra}
i_2: \Assb\lon \cB ra^{or}
\Eeq
given on the generators as follows,}
$$
i_2\left(\begin{xy}
 <0mm,-0.55mm>*{};<0mm,-3.5mm>*{}**@{.},
 <0.5mm,0.5mm>*{};<2.2mm,2.2mm>*{}**@{.},
 <-0.48mm,0.48mm>*{};<-2.2mm,2.2mm>*{}**@{.},
 <0mm,0mm>*{\circ};<0mm,0mm>*{}**@{},
 <0.5mm,0.5mm>*{};<2.7mm,2.8mm>*{^2}**@{},
 <-0.48mm,0.48mm>*{};<-2.7mm,2.8mm>*{^1}**@{},
 \end{xy}\right)=
 \Ba{c}\resizebox{8mm}{!}{ \xy
(-3,9)*{^1},
(3,9)*{^2},
 (0,2)*{\circ}="a",
(-3,7)*{\circ}="b_1",
(3,7)*{\circ}="b_2",
 \endxy}\Ea  \ \ \ \ \ \ \ \ \ \ , \ \ \ \ \ \ \ \ \ \ \
i_2\left(\begin{xy}
 <0mm,0.66mm>*{};<0mm,4mm>*{}**@{.},
 <0.39mm,-0.39mm>*{};<2.2mm,-2.2mm>*{}**@{.},
 <-0.35mm,-0.35mm>*{};<-2.2mm,-2.2mm>*{}**@{.},
 <0mm,0mm>*{\circ};<0mm,0mm>*{}**@{},
   <0.39mm,-0.39mm>*{};<2.9mm,-4mm>*{^2}**@{},
   <-0.35mm,-0.35mm>*{};<-2.8mm,-4mm>*{^1}**@{},
\end{xy}\right)=\Ba{c}\resizebox{8mm}{!}{ \xy
(-3,0)*{_1},
(3,0)*{_2},
 (0,7)*{\circ}="a",
(-3,2)*{\circ}="b_1",
(3,2)*{\circ}="b_2",
 \endxy}\Ea
$$

\subsubsection{\bf A canonical representation of the prop $\cB ra^{or}$}\label{4: Subsection on repr of Bra}
Let $V$ be an arbitrary graded vector space with a basis $\{x^i\}$ so that
$\f_V=\K[x^i]$ and $\fg_V=\K[[\psi^i,\psi_i]]$ with $\psi^i=sx^i$. Let
$$
\cE nd_{\f_V,\ \fg_V}=\{\Hom( \fg_V^{\ot p}\ot \f_V^{\ot m},   \fg_V^{\ot q}\ot \f_V^{\ot n}\}
$$
 be the two coloured endomorphism prop of vector spaces $\f_V$ (in white colour) and $\fg_V$ (in black colour). A representation,
$$
\rho:  \cB ra^{or} \lon   \cE nd_{\f_V,\ \fg_V}
$$
is uniquely determined by its  values on the generators. The values of $\rho$ on the generators of $\cG ra_3^{ar}$ are given by (\ref{3: Gra representation in g_V}). The value,
$$
\rho_\Ga\in \Hom\left(\fg_V^{\ot k}\ot \f_V^{\ot m} ,  \f_V^{\ot n}\right),
$$
of $\rho$ on a graph $\Ga\in \sG_{k;m,n}$ of $\cB ra^{or}$ is defined
as a composition of maps,
\Beq\label{4: formular for rho_Ga for Bra}
\Ba{rccccccc}
\rho_{\Ga}: & \fg_V^{\ot k}\ot \f_V^{\ot m} \hspace{-3mm}&\lon & \hspace{-1mm} \f_V^{\ot n}\ot\fg_V^{\ot k}\ot \f_V^{\ot m} &\stackrel{\Psi_\Ga}{\lon} &
 \hspace{-1mm}
 \f_V^{\ot n}\ot \f_V^{\ot m}  \hspace{-2mm} & \stackrel{\mu}{\lon} &   \hspace{-2mm} \f_V^{\ot n}\\
& \hspace{-3mm}\ga_1\ot ...\ot \ga_k\ot f_{{1}}\ot ...\ot f_{{m}}\hspace{-1mm}
& \hspace{-3mm} \lon & 1^{\ot n}\ot \ga_1\ot ...\ot \ga_k\ot f_{{1}}\ot ...\ot f_{{m}} & & & &
\Ea
\Eeq
where

$$
\Psi_\Ga(\underbrace{1,...,1}_n, \ga_1,..., \ga_k,f_{{1}},..., f_{{m}})
= \hspace{-2mm}\left(\prod_{\overline{e}\in E_{out}(\Ga)} \hspace{-4mm}\Delta_{\overline{e}}\prod_{{e}\in E_{int}(\Ga)}\hspace{-4mm}\Delta_e \prod_{\underline{e}\in E_{in}(\Ga)}\hspace{-4mm}\Delta_{\underline{e}}
(\underbrace{1\ot ...\ot 1}_n\ot \ga_1\ot ...\ot \ga_k\ot f_{{1}}\ot ...\ot f_{{m}})\right)_{ \hspace{-1mm}\psi_i=0\atop \hspace{-1mm} \psi^i=0}
$$
with
\Bi
\item for an in-leg  $e=\Ba{c}\xy
(0,2)*{_{{\al}}},
(6,2)*{_{a}},
 (0,0)*{\circ}="a",
(6,0)*{\bu}="b",
\ar @{->} "a";"b" <0pt>
\endxy \Ea$ connecting ${\al}$-th in-vertex to $a$-th internal vertex, $\al\in [m]$, $a\in [k]$,
$$
\Delta_{\underline{e}}:=\sum_{i\in I}\underbrace{\frac{\p}{\p\psi^i}}_{\mathrm{acts\ on}\
 a-\mathrm{th}\atop \mathrm{tensor\ factor\ of}\ \fg_V^{\ot k}}\ot
\underbrace{\frac{\p}{\p x_i}}_{\mathrm{acts\ on}
 \  \al-\mathrm{th}\atop \mathrm{tensor\ factor\ of}\ \f_V^{\ot m}}
$$
\item for an internal edge  $e=\Ba{c}\xy
(0,2)*{_{{a}}},
(6,2.1)*{_{b}},
 (0,0)*{\bullet}="a",
(6,0)*{\bu}="b",
\ar @{->} "a";"b" <0pt>
\endxy \Ea$, $a,b\in [k]$,
$$
\Delta_{{e}}:=\sum_{i\in I}\underbrace{\frac{\p}{\p\psi^i}}_{\mathrm{acts\ on}\
 a\mathrm{-th}\atop \mathrm{tensor\ factor\ of}\ \fg_V^{\ot k}}\ot
\underbrace{\frac{\p}{\p \psi_i}}_{\mathrm{acts\ on}\
 b\mathrm{-th}\atop \mathrm{tensor\ factor\ of}\ \fg_V^{\ot k}}
$$
\item for an out-leg  $\overline{e}=\Ba{c}\xy
(0,2)*{_{a}},
(6,2.2)*{_{{\be}}},
 (0,0)*{\bu}="a",
(6,0)*{\circ}="b",
\ar @{->} "a";"b" <0pt>
\endxy \Ea$, $a\in [k]$, $\be\in [n]$,
$$
\Delta_{\overline{e}}:=\sum_{i\in I}\underbrace{x_i}_{\mathrm{acts\ as\ multiplication\ on}\
 \be-\mathrm{th}\atop \mathrm{tensor\ factor\ of}\ 1^{\ot n}}\ot
\underbrace{\frac{\p}{\p \psi_i}}_{\mathrm{acts\ on}
 \  a-\mathrm{th}\atop \mathrm{tensor\ factor\ of}\ \fg_V^{\ot k}}
$$
\Ei
and where
$$
\Ba{rccc}
\mu: & \f_V^{\ot n}\ot \f_V^{\ot m} & \lon & \f_V^{\ot n}\\
     &  g_1\ot ...\ot g_n \ot h_1\ot...\ot h_m & \lon &
( g_1\ot ...\ot g_n)\cdot (\Delta^{n-1} (h_1\cdots h_m))
\Ea
$$
with $\cdot$ being the standard multiplication in the algebra $\f_V^{\ot n}$.

\mip

The proof of the claim that the above formula for $\rho(\Ga)$ gives a  morphism of props (i.e.\ respects prop compositions) is a straightforward untwisting of the definitions; in fact, the prop compositions in $\cB \cG ra^{or}$ have been defined just precisely for the purpose to make this claim true, i.e.\ they have been read out from the compositions of the operators $\rho(\Ga)$.

\subsubsection{\bf Remarks} (i)
Note that operators $\Delta_{\underline{e}}$ and $\Delta_{\overline{e}}$ have degree $-1$ (in contrast to operators $\Delta_e$ which have degree $-2$) so that the ordering of legs of a graph $\Ga$ from $\cB\cG ra^{or}$ is required
to make the definition of $\rho(\Ga)$ unambiguous. However, we shall see below that this choice is irrelevant for the application of the morphism $\rho$ to deformation quantization.

\sip

(ii) As our formula for $\rho(\Ga)$ involves evaluation at $\psi_\bu=\psi^\bu=0$, we can say that $\rho(\Ga)(\ga_1,...,\ga_k,f_1,...,f_m)$ is non-zero only in the case when
 for any $a\in [k]$
one has
$$
\ga_a\in \odot^{\# In(v_a)} V^*[1] \ot   \odot^{\# Out(v_a)} V[1]
$$
where $v_a$ is the $a$-labelled internal vertex of $\Ga$ and  $In(v_a)$
(resp., $Out(v_a)$) is the set of input (resp., output) half-edges.

\subsubsection{{\bf Examples} {\em (cf.\ \cite{Sh})}}

1) $\rho\left( \Ba{c}\resizebox{6mm}{!}{ \xy
(-3,0)*{_1},
(3,0)*{_2},
 (0,7)*{\circ}="a",
(-3,2)*{\circ}="b_1",
(3,2)*{\circ}="b_2",
 \endxy}\Ea \right): \f_V^{\ot 2}\rar \f_V$ is just the multiplication
in $\f_V$.

\sip

2)   $\rho\left(   \Ba{c}\resizebox{6mm}{!}{ \xy
(-3,9)*{^1},
(3,9)*{^2},
 (0,2)*{\circ}="a",
(-3,7)*{\circ}="b_1",
(3,7)*{\circ}="b_2",
 \endxy}\Ea  \right): \f_V\rar \f_V^{\ot 2}$ is  the comultiplication
$\Delta$ in $\f_V$.

3) If $\ga=\sum_{i,j,k\in I}C_{ij}^k\psi_k \psi^i\psi^j$, $C_{ij}^k\in \R$, then, for any
$f_1(x),f_2(x)\in \f_V$, one has (modulo a sign coming from a choice of ordering of legs)
$$
\rho\left( \Ba{c}\resizebox{7mm}{!}{ \xy
(-5,0)*{_1},
(5,0)*{_2},
(0,13)*{\circ}="0",
 (0,7)*{\bu}="a",
(-5,2)*{\circ}="b_1",
(5,2)*{\circ}="b_2",
\ar @{->} "a";"0" <0pt>
\ar @{<-} "a";"b_1" <0pt>
\ar @{<-} "a";"b_2" <0pt>
\endxy}\Ea\right)(\ga, f_1,f_2)=\sum_{i,j,k\in I} \pm x_k C^k_{ij} \frac{\p f_1}{\p x_i}
\frac{\p f_2}{\p x_j}
$$

4) If $\ga=\sum_{i,j,k\in I}\Phi^{ij}_k\psi_i\psi_j\psi^k$, $\Phi^{ij}_k\in \R$, then, for any
$f(x)\in \f_V$, one has (modulo a sign coming from a choice of ordering of legs)
$$
\rho\left( \Ba{c}
\resizebox{6mm}{!}{\xy
(-5,15)*{_1},
(5,15)*{_2},
(0,2)*{\circ}="0",
 (0,8)*{\bu}="a",
(-5,13)*{\circ}="b_1",
(5,13)*{\circ}="b_2",
\ar @{<-} "a";"0" <0pt>
\ar @{->} "a";"b_1" <0pt>
\ar @{->} "a";"b_2" <0pt>
\endxy} \Ea \right)(\ga, f)=\sum_{i,j,k\in I} \pm (x_i\ot x_j)\cdot \Phi_k^{ij} \Delta (\frac{\p f}{\p x_k})
$$

5) If f $\ga_1=\sum_{i,j,k\in I}C_{ij}^k\psi_k \psi^i\psi^j$,$\ga_2=\sum_{i,j,k\in I}\Phi^{ij}_k\psi_i\psi_j\psi^k$, $C_{ij}^k,\Phi^{ij}_k\in \R$, then, for any
$f_1,f_2\in \f_V$, one has (modulo a sign coming from a choice of ordering of legs)
$$
\rho\left( \Ba{c}
\resizebox{7mm}{!}{
\xy
(-5,0)*{_1},
(5,0)*{_2},
(-5,20)*{^1},
(5,20)*{^2},
(-2,13)*{_2},
(-2,7)*{_1},
(0,13)*{\bu}="0",
 (0,7)*{\bu}="a",
(-5,2)*{\circ}="b_1",
(5,2)*{\circ}="b_2",
(-5,18)*{\circ}="u_1",
(5,18)*{\circ}="u_2",
\ar @{->} "a";"0" <0pt>
\ar @{<-} "a";"b_1" <0pt>
\ar @{<-} "a";"b_2" <0pt>
\ar @{->} "0";"u_1" <0pt>
\ar @{->} "0";"u_2" <0pt>
\endxy}
\Ea
\right)(\ga, f)=\sum_{i,j,k,m,n\in I} \pm (x_m\ot x_n)\cdot \Phi_k^{mn} C^k_{ij} \Delta( \frac{\p f_1}{\p x_i}
\frac{\p f_2}{\p x_j})
$$

\subsection{A 2-coloured prop of $\caL ie_3$ actions on bialgebras} Consider an auxiliary 2-coloured prop (in straight and dashed colours), $\cQ ua$, generated by the following binary operations
\Beq\label{3: generators of Gac}
\underbrace{\begin{xy}
 <0mm,0.66mm>*{};<0mm,3mm>*{}**@{-},
 <0.39mm,-0.39mm>*{};<2.2mm,-2.2mm>*{}**@{-},
 <-0.35mm,-0.35mm>*{};<-2.2mm,-2.2mm>*{}**@{-},
 <0mm,0mm>*{\bu};<0mm,0mm>*{}**@{},
   <0.39mm,-0.39mm>*{};<2.9mm,-4mm>*{^2}**@{},
   <-0.35mm,-0.35mm>*{};<-2.8mm,-4mm>*{^1}**@{},
\end{xy}=-
\begin{xy}
 <0mm,0.66mm>*{};<0mm,3mm>*{}**@{-},
 <0.39mm,-0.39mm>*{};<2.2mm,-2.2mm>*{}**@{-},
 <-0.35mm,-0.35mm>*{};<-2.2mm,-2.2mm>*{}**@{-},
 <0mm,0mm>*{\bu};<0mm,0mm>*{}**@{},
   <0.39mm,-0.39mm>*{};<2.9mm,-4mm>*{^1}**@{},
   <-0.35mm,-0.35mm>*{};<-2.8mm,-4mm>*{^2}**@{},
\end{xy}}_{\mathrm{degree}\ -2},\ \ \ \
\underbrace{\begin{xy}
 <0mm,0.66mm>*{};<0mm,3mm>*{}**@{.},
 <0.39mm,-0.39mm>*{};<2.2mm,-2.2mm>*{}**@{.},
 <-0.35mm,-0.35mm>*{};<-2.2mm,-2.2mm>*{}**@{-},
 <0mm,0mm>*{\bu};<0mm,0mm>*{}**@{},
   <0.39mm,-0.39mm>*{};<2.9mm,-4mm>*{^2}**@{},
   <-0.35mm,-0.35mm>*{};<-2.8mm,-4mm>*{^1}**@{},
\end{xy}
}_{\mathrm{degree}\ -2},\ \ \ \
\underbrace{\begin{xy}
 <0mm,-0.55mm>*{};<0mm,-2.5mm>*{}**@{.},
 <0.5mm,0.5mm>*{};<2.2mm,2.2mm>*{}**@{.},
 <-0.48mm,0.48mm>*{};<-2.2mm,2.2mm>*{}**@{.},
 <0mm,0mm>*{\circ};<0mm,0mm>*{}**@{},
 <0mm,-0.55mm>*{};<0mm,-3.8mm>*{_1}**@{},
 <0.5mm,0.5mm>*{};<2.7mm,2.8mm>*{^2}**@{},
 <-0.48mm,0.48mm>*{};<-2.7mm,2.8mm>*{^1}**@{},
 \end{xy}
\ \
,\ \
\begin{xy}
 <0mm,-0.55mm>*{};<0mm,-2.5mm>*{}**@{.},
 <0.5mm,0.5mm>*{};<2.2mm,2.2mm>*{}**@{.},
 <-0.48mm,0.48mm>*{};<-2.2mm,2.2mm>*{}**@{.},
 <0mm,0mm>*{\circ};<0mm,0mm>*{}**@{},
 <0mm,-0.55mm>*{};<0mm,-3.8mm>*{_1}**@{},
 <0.5mm,0.5mm>*{};<2.7mm,2.8mm>*{^1}**@{},
 <-0.48mm,0.48mm>*{};<-2.7mm,2.8mm>*{^2}**@{},
 \end{xy}\ \ ,
 \begin{xy}
 <0mm,0.66mm>*{};<0mm,3mm>*{}**@{.},
 <0.39mm,-0.39mm>*{};<2.2mm,-2.2mm>*{}**@{.},
 <-0.35mm,-0.35mm>*{};<-2.2mm,-2.2mm>*{}**@{.},
 <0mm,0mm>*{\circ};<0mm,0mm>*{}**@{},
   <0mm,0.66mm>*{};<0mm,3.4mm>*{^1}**@{},
   <0.39mm,-0.39mm>*{};<2.9mm,-4mm>*{^2}**@{},
   <-0.35mm,-0.35mm>*{};<-2.8mm,-4mm>*{^1}**@{},
\end{xy}
\ \
,\ \
\begin{xy}
 <0mm,0.66mm>*{};<0mm,3mm>*{}**@{.},
 <0.39mm,-0.39mm>*{};<2.2mm,-2.2mm>*{}**@{.},
 <-0.35mm,-0.35mm>*{};<-2.2mm,-2.2mm>*{}**@{.},
 <0mm,0mm>*{\circ};<0mm,0mm>*{}**@{},
   <0mm,0.66mm>*{};<0mm,3.4mm>*{^1}**@{},
   <0.39mm,-0.39mm>*{};<2.9mm,-4mm>*{^1}**@{},
   <-0.35mm,-0.35mm>*{};<-2.8mm,-4mm>*{^2}**@{},
\end{xy}}_{\mathrm{degree}\ 0}
\Eeq
modulo Jacobi relations
\Beq\label{4: Jacobi relation for Lie}
\Ba{c}
 \begin{xy}
 <0mm,0mm>*{\bu};<0mm,0mm>*{}**@{},
 <0mm,0.69mm>*{};<0mm,3.0mm>*{}**@{-},
 <0.39mm,-0.39mm>*{};<2.4mm,-2.4mm>*{}**@{-},
 <-0.35mm,-0.35mm>*{};<-1.9mm,-1.9mm>*{}**@{-},
 <-2.4mm,-2.4mm>*{\bu};<-2.4mm,-2.4mm>*{}**@{},
 <-2.0mm,-2.8mm>*{};<0mm,-4.9mm>*{}**@{-},
 <-2.8mm,-2.9mm>*{};<-4.7mm,-4.9mm>*{}**@{-},
    <0.39mm,-0.39mm>*{};<3.3mm,-4.0mm>*{^3}**@{},
    <-2.0mm,-2.8mm>*{};<0.5mm,-6.7mm>*{^2}**@{},
    <-2.8mm,-2.9mm>*{};<-5.2mm,-6.7mm>*{^1}**@{},
 \end{xy}
\ + \
 \begin{xy}
 <0mm,0mm>*{\bu};<0mm,0mm>*{}**@{},
 <0mm,0.69mm>*{};<0mm,3.0mm>*{}**@{-},
 <0.39mm,-0.39mm>*{};<2.4mm,-2.4mm>*{}**@{-},
 <-0.35mm,-0.35mm>*{};<-1.9mm,-1.9mm>*{}**@{-},
 <-2.4mm,-2.4mm>*{\bu};<-2.4mm,-2.4mm>*{}**@{},
 <-2.0mm,-2.8mm>*{};<0mm,-4.9mm>*{}**@{-},
 <-2.8mm,-2.9mm>*{};<-4.7mm,-4.9mm>*{}**@{-},
    <0.39mm,-0.39mm>*{};<3.3mm,-4.0mm>*{^2}**@{},
    <-2.0mm,-2.8mm>*{};<0.5mm,-6.7mm>*{^1}**@{},
    <-2.8mm,-2.9mm>*{};<-5.2mm,-6.7mm>*{^3}**@{},
 \end{xy}
\ + \
 \begin{xy}
 <0mm,0mm>*{\bu};<0mm,0mm>*{}**@{},
 <0mm,0.69mm>*{};<0mm,3.0mm>*{}**@{-},
 <0.39mm,-0.39mm>*{};<2.4mm,-2.4mm>*{}**@{-},
 <-0.35mm,-0.35mm>*{};<-1.9mm,-1.9mm>*{}**@{-},
 <-2.4mm,-2.4mm>*{\bu};<-2.4mm,-2.4mm>*{}**@{},
 <-2.0mm,-2.8mm>*{};<0mm,-4.9mm>*{}**@{-},
 <-2.8mm,-2.9mm>*{};<-4.7mm,-4.9mm>*{}**@{-},
    <0.39mm,-0.39mm>*{};<3.3mm,-4.0mm>*{^1}**@{},
    <-2.0mm,-2.8mm>*{};<0.5mm,-6.7mm>*{^3}**@{},
    <-2.8mm,-2.9mm>*{};<-5.2mm,-6.7mm>*{^2}**@{},
 \end{xy}\Ea=0
\Eeq
the associative bialgebra relations (\ref{2: bialgebra relations}), and also the following ones
\Beq\label{3: Gac prop rel 1}
\Ba{c}
\xy
(-6,-8)*{_{1}},
(0,-8)*{_{2}},
(4,-4.5)*{_{3}},
 (0,0)*{\circ}="a",
(-3,-3)*{\bu}="b1",
(3,-3)*{}="b2",
 (0,3)*{}="u",
 (-6,-6)*{}="d1",
 (0,-6)*{}="d2",
\ar @{.} "a";"u" <0pt>
\ar @{-} "a";"b1" <0pt>
\ar @{.} "a";"b2" <0pt>
\ar @{-} "b1";"d1" <0pt>
\ar @{-} "b1";"d2" <0pt>
\endxy\Ea  -
\Ba{c}\xy
(-4,-4.5)*{_{1}},
(0,-8)*{_{2}},
(7,-8)*{_{3}},
 (0,0)*{\bu}="a",
(-3,-3)*{}="b2",
(3,-3)*{\bu}="b1",
 (0,3)*{}="u",
 (0,-6)*{}="d1",
 (6,-6)*{}="d2",
\ar @{.} "a";"u" <0pt>
\ar @{.} "a";"b1" <0pt>
\ar @{-} "a";"b2" <0pt>
\ar @{-} "b1";"d1" <0pt>
\ar @{.} "b1";"d2" <0pt>
\endxy\Ea
 +
\Ba{c}\xy
(-4,-4.5)*{_{2}},
(0,-8)*{_{1}},
(7,-8)*{_{3}},
 (0,0)*{\bu}="a",
(-3,-3)*{}="b2",
(3,-3)*{\bu}="b1",
 (0,3)*{}="u",
 (0,-6)*{}="d1",
 (6,-6)*{}="d2",
\ar @{.} "a";"u" <0pt>
\ar @{.} "a";"b1" <0pt>
\ar @{-} "a";"b2" <0pt>
\ar @{-} "b1";"d1" <0pt>
\ar @{.} "b1";"d2" <0pt>
\endxy\Ea=0,
\Eeq
\Beq\label{3: Gac prop rel 2}
\Ba{c}\xy
(-4,-4.5)*{_{1}},
(0,-8)*{_{2}},
(7,-8)*{_{3}},
 (0,0)*{\bu}="a",
(-3,-3)*{}="b2",
(3,-3)*{\circ}="b1",
 (0,3)*{}="u",
 (0,-6)*{}="d1",
 (6,-6)*{}="d2",
\ar @{.} "a";"u" <0pt>
\ar @{.} "a";"b1" <0pt>
\ar @{-} "a";"b2" <0pt>
\ar @{.} "b1";"d1" <0pt>
\ar @{.} "b1";"d2" <0pt>
\endxy\Ea   -
\Ba{c}\xy
(-6,-8)*{_{1}},
(0,-8)*{_{2}},
(4,-4.5)*{_{3}},
 (0,0)*{\circ}="a",
(-3,-3)*{\bu}="b1",
(3,-3)*{}="b2",
 (0,3)*{}="u",
 (-6,-6)*{}="d1",
 (0,-6)*{}="d2",
\ar @{.} "a";"u" <0pt>
\ar @{.} "a";"b1" <0pt>
\ar @{.} "a";"b2" <0pt>
\ar @{-} "b1";"d1" <0pt>
\ar @{.} "b1";"d2" <0pt>
\endxy\Ea
  -
\Ba{c}\xy
(-4,-4.5)*{_{2}},
(0,-8)*{_{1}},
(7,-8)*{_{3}},
 (0,0)*{\circ}="a",
(-3,-3)*{}="b2",
(3,-3)*{\bu}="b1",
 (0,3)*{}="u",
 (0,-6)*{}="d1",
 (6,-6)*{}="d2",
\ar @{.} "a";"u" <0pt>
\ar @{.} "a";"b1" <0pt>
\ar @{.} "a";"b2" <0pt>
\ar @{-} "b1";"d1" <0pt>
\ar @{.} "b1";"d2" <0pt>
\endxy\Ea=0, \ \  \ \ \
\Ba{c}
\xy
(-3,9)*{_{1}},
(3,9)*{_{2}},
 (0,0)*{\bu}="a",
(-3,-3)*{}="b2",
(3,-3)*{}="b1",
 (0,4)*{\circ}="u",
 (-3,7)*{}="u1",
 (3,7)*{}="u2",
\ar @{.} "a";"u" <0pt>
\ar @{.} "a";"b1" <0pt>
\ar @{-} "a";"b2" <0pt>
\ar @{.} "u";"u1" <0pt>
\ar @{.} "u";"u2" <0pt>
\endxy\Ea
  -
\Ba{c}\xy
(-3,9)*{_{1}},
(3,6)*{_{2}},
 (0,0)*{\circ}="a",
(-3,4)*{\bu}="b1",
(3,4)*{}="b2",
 (0,-3)*{}="d",
 (-3,7)*{}="u",
 (-6,0)*{}="l",
\ar @{.} "a";"d" <0pt>
\ar @{.} "a";"b1" <0pt>
\ar @{.} "a";"b2" <0pt>
\ar @{-} "b1";"l" <0pt>
\ar @{.} "b1";"u" <0pt>
\endxy\Ea
 -
\Ba{c}\xy
(-3,6)*{_{1}},
(3,9)*{_{2}},
 (0,0)*{\circ}="a",
(3,4)*{\bu}="b1",
(-3,4)*{}="b2",
 (0,-3)*{}="d",
 (3,7)*{}="u",
 (6,0)*{}="l",
\ar @{.} "a";"d" <0pt>
\ar @{.} "a";"b1" <0pt>
\ar @{.} "a";"b2" <0pt>
\ar @{-} "b1";"l" <0pt>
\ar @{.} "b1";"u" <0pt>
\endxy\Ea=0.
\Eeq
Rather surprisingly, this prop arises naturally in the context of formality maps for universal quantizations of Lie bialgebras.

\subsection{A prop  governing formality maps}
Sometimes it is useful to define a prop in terms of generators and relations, sometimes it is easier to define it by specifying its generic representation. We shall define a useful for our purposes prop $\cQ ua_\infty$ using both approaches (but start with the second one which explains its relation to formality maps)) and then prove
that $\cQ ua_\infty$ is a minimal resolution of $\cQ ua$.

\sip

Let $\cQ ua_\infty$ be a 2-coloured prop whose arbitrary representation,  $\rho: \cQ ua\rar\cE nd_{\fg,A}$, in a pair of graded vector spaces $\fg$ and $A$ provides these spaces with the following list of algebraic structures,
\Bi
\item[(i)] the structure of a $\caH olie_{3}$-algebra in $\fg$, i.e.\ a representation $\rho_1: \caH olie_{3}\rar \cE nd_{\fg}$ (note that $\fg[2]$ is then an ordinary $\caL ie_\infty$ algebra);
\item[(ii)] the structure of an $\Assb_\infty$-algebra in $A$,
 i.e.\ a representation $\rho_2: \Assb_\infty \rar \cE nd_{A}$
\item[(iii)] a morphism of $\caL ie_\infty$-algebras,
$$
F: \fg[2]\lon \Def(\Assb_\infty \stackrel{\rho}{\lon} \cE nd_A)
$$
\Ei
The prop $\cQ ua_\infty$ is  free with the space of generators in the colour responsible for datum (i)  given by
\Beq\label{4: generators of Lie_infty2}
E(n):=sgn_n[3n-4]=\left\langle\Ba{c}
 \resizebox{16mm}{!}{\xy
(1,-5)*{\ldots},
(-13,-7)*{_1},
(-8,-7)*{_2},
(-3,-7)*{_3},
(7,-7)*{_{n-1}},
(13,-7)*{_n},
 (0,0)*{\bu}="a",
(0,5)*{}="0",
(-12,-5)*{}="b_1",
(-8,-5)*{}="b_2",
(-3,-5)*{}="b_3",
(8,-5)*{}="b_4",
(12,-5)*{}="b_5",
\ar @{-} "a";"0" <0pt>
\ar @{-} "a";"b_2" <0pt>
\ar @{-} "a";"b_3" <0pt>
\ar @{-} "a";"b_1" <0pt>
\ar @{-} "a";"b_4" <0pt>
\ar @{-} "a";"b_5" <0pt>
\endxy}=(-1)^{\sigma}
\resizebox{18mm}{!}{\xy
(1,-6)*{\ldots},
(-13,-7)*{_{\sigma(1)}},
(-6.7,-7)*{_{\sigma(2)}},
(13,-7)*{_{\sigma(n)}},
 (0,0)*{\bu}="a",
(0,5)*{}="0",
(-12,-5)*{}="b_1",
(-8,-5)*{}="b_2",
(-3,-5)*{}="b_3",
(8,-5)*{}="b_4",
(12,-5)*{}="b_5",
\ar @{-} "a";"0" <0pt>
\ar @{-} "a";"b_2" <0pt>
\ar @{-} "a";"b_3" <0pt>
\ar @{-} "a";"b_1" <0pt>
\ar @{-} "a";"b_4" <0pt>
\ar @{-} "a";"b_5" <0pt>
\endxy}\Ea
\right\rangle_{\sigma\in \bS_n}
\Eeq
and the space of generators (of mixed colours) responsible for data (ii) and (iii)  given by,
\Beq\label{4: generators of Qua}
E(n,k+m):=\R[\bS_n]\ot \sgn_k \ot  \R[\bS_m][m+n+3k-3]=\mathrm{span}\left\langle
\Ba{c}
\resizebox{16mm}{!}{\xy
 (0,7)*{\overbrace{\ \ \ \  \ \ \ \ \ \ \ \ \ \ }},
 (0,9)*{^n},
 (0,3)*{^{...}},
 (5.6,-5.5)*{^{...}},
 (6,-7)*{\underbrace{  \ \ \ \ \ \ }},
 (6,-9)*{_m},
 (-5.6,-5.5)*{^{...}},
 (-6,-7)*{\underbrace{  \ \ \ \ \ \ }},
 (-6,-9)*{_k},
 (0,0)*{\circ}="0",
(-7,5)*{}="u_1",
(-4,5)*{}="u_2",
(4,5)*{}="u_3",
(7,5)*{}="u_4",
(2,-5)*{}="d_1",
(3.6,-5)*{}="d_2",
(9,-5)*{}="d_3",
(-2,-5)*{}="s_1",
(-3.6,-5)*{}="s_2",
(-9,-5)*{}="s_3",
\ar @{.} "0";"u_1" <0pt>
\ar @{.} "0";"u_2" <0pt>
\ar @{.} "0";"u_3" <0pt>
\ar @{.} "0";"u_4" <0pt>
\ar @{.} "0";"d_1" <0pt>
\ar @{.} "0";"d_2" <0pt>
\ar @{.} "0";"d_3" <0pt>
\ar @{-} "0";"s_1" <0pt>
\ar @{-} "0";"s_2" <0pt>
\ar @{-} "0";"s_3" <0pt>
\endxy}\Ea
\right\rangle_{n\geq 1,m\geq1, k\geq 0\atop
n+k+m\geq 3}
\Eeq
Corollas with $k=0$ take care about datum (ii), i.e.\ they are the standard generators of $\Assb_\infty$, while corollas with $k\geq 1$ take care about
datum (ii). The prop $\cQ ua_\infty$ is differential; the values of its differential
$d$ on
\Bi
\item the generators (\ref{4: generators of Lie_infty2})
of $\caH olie_3$ is given by the standard formula,
$$
d\hspace{-3mm}
{ \xy
(1,-5)*{\ldots},
(-13,-7)*{_1},
(-8,-7)*{_2},
(-3,-7)*{_3},
(7,-7)*{_{n-1}},
(13,-7)*{_n},
 (0,0)*{\bu}="a",
(0,5)*{}="0",
(-12,-5)*{}="b_1",
(-8,-5)*{}="b_2",
(-3,-5)*{}="b_3",
(8,-5)*{}="b_4",
(12,-5)*{}="b_5",
\ar @{-} "a";"0" <0pt>
\ar @{-} "a";"b_2" <0pt>
\ar @{-} "a";"b_3" <0pt>
\ar @{-} "a";"b_1" <0pt>
\ar @{-} "a";"b_4" <0pt>
\ar @{-} "a";"b_5" <0pt>
\endxy}
=
\sum_{ [n]=I_1\sqcup I_2\atop
\# I_1\geq 1, \# I_2\geq 1}\pm
\Ba{c}
\begin{xy}
<10mm,0mm>*{\bu},
<10mm,0.8mm>*{};<10mm,5mm>*{}**@{-},
<0mm,-10mm>*{...},
<14mm,-5mm>*{\ldots},
<13mm,-7mm>*{\underbrace{\ \ \ \ \ \ \ \ \ \ \ \ \  }},
<14mm,-10mm>*{_{I_2}};
<10.3mm,0.1mm>*{};<20mm,-5mm>*{}**@{-},
<9.7mm,-0.5mm>*{};<6mm,-5mm>*{}**@{-},
<9.9mm,-0.5mm>*{};<10mm,-5mm>*{}**@{-},
<9.6mm,0.1mm>*{};<0mm,-4.4mm>*{}**@{-},
<0mm,-5mm>*{\bu};
<-5mm,-10mm>*{}**@{-},
<-2.7mm,-10mm>*{}**@{-},
<2.7mm,-10mm>*{}**@{-},
<5mm,-10mm>*{}**@{-},
<0mm,-12mm>*{\underbrace{\ \ \ \ \ \ \ \ \ \ }},
<0mm,-15mm>*{_{I_1}},
\end{xy}
\Ea
$$
\item the generators (\ref{4: generators of Qua}) with $k=0$
are determined by a choice of a minimal resolution, $\cA ss\cB_\infty$, i.e.\ $d$
restricted to these generators equals precisely the differential $\delta$ in $\Assb_\infty$;

\item  generators (\ref{4: generators of Qua}) with $k\geq 1,n\geq 1$
  are determined uniquely by the choice of a minimal resolution $\Assb_\infty$  and the standard  formulae for a $\caL ie_\infty$ morphism.
 \Ei
 Roughly speaking, the value of the differential on a generator (\ref{4: generators of Qua}) with $k\geq 1,n\geq 1$ a sum over all possible attachments of $k$ solid legs to the vertices in the graph
  $\delta \Ba{c}
\resizebox{9mm}{!}{\begin{xy}
 <0mm,0mm>*{\circ};<0mm,0mm>*{}**@{},
 <0mm,0mm>*{};<-8mm,5mm>*{}**@{.},
 <0mm,0mm>*{};<-4.5mm,5mm>*{}**@{.},
 <0mm,0mm>*{};<-1mm,5mm>*{\ldots}**@{},
 <0mm,0mm>*{};<4.5mm,5mm>*{}**@{.},
 <0mm,0mm>*{};<8mm,5mm>*{}**@{.},
   <0mm,0mm>*{};<-10.5mm,5.9mm>*{^{1}}**@{},
   <0mm,0mm>*{};<-4mm,5.9mm>*{^{2}}**@{},
   <0mm,0mm>*{};<10.0mm,5.9mm>*{^{n}}**@{},
 <0mm,0mm>*{};<-8mm,-5mm>*{}**@{.},
 <0mm,0mm>*{};<-4.5mm,-5mm>*{}**@{.},
 <0mm,0mm>*{};<-1mm,-5mm>*{\ldots}**@{},
 <0mm,0mm>*{};<4.5mm,-5mm>*{}**@{.},
 <0mm,0mm>*{};<8mm,-5mm>*{}**@{.},
   <0mm,0mm>*{};<-10.5mm,-6.9mm>*{^{1}}**@{},
   <0mm,0mm>*{};<-4mm,-6.9mm>*{^{2}}**@{},
   <0mm,0mm>*{};<10.0mm,-6.9mm>*{^{m}}**@{},
 \end{xy}}\Ea$ plus standard terms of ``$\caL ie_\infty$" type.
For example, one has in the simplest cases $m=1$, $k=0$, $n\geq 2$,
$$
d
\resizebox{20mm}{!}{\xy
(1,-5)*{\ldots},
(-13,-7)*{_{1}},
(-8,-7)*{_{2}},
(-3,-7)*{_{3}},
(8,-7)*{_{{n-1}}},
(14,-7)*{_{n}},
 (0,0)*{\circ}="a",
(0,5)*{}="0",
(-12,-5)*{}="b_1",
(-8,-5)*{}="b_2",
(-3,-5)*{}="b_3",
(8,-5)*{}="b_4",
(12,-5)*{}="b_5",
\ar @{.} "a";"0" <0pt>
\ar @{.} "a";"b_2" <0pt>
\ar @{.} "a";"b_3" <0pt>
\ar @{.} "a";"b_1" <0pt>
\ar @{.} "a";"b_4" <0pt>
\ar @{.} "a";"b_5" <0pt>
\endxy}
=\sum_{k=0}^{n-2}\sum_{l=2}^{n-k}
(-1)^{k+l(n-k-l)}
\resizebox{25mm}{!}{\begin{xy}
<0mm,0mm>*{\circ},
<0mm,0.8mm>*{};<0mm,5mm>*{}**@{-},
<-9mm,-5mm>*{\ldots},
<14mm,-5mm>*{\ldots},
<-0.7mm,-0.3mm>*{};<-13mm,-5mm>*{}**@{.},
<-0.6mm,-0.5mm>*{};<-6mm,-5mm>*{}**@{.},
<0.6mm,-0.3mm>*{};<20mm,-5mm>*{}**@{.},
<0.3mm,-0.5mm>*{};<8mm,-5mm>*{}**@{.},
<0mm,-0.5mm>*{};<0mm,-4.3mm>*{}**@{.},
<0mm,-5mm>*{\circ};
<-5mm,-10mm>*{}**@{.},
<-2.7mm,-10mm>*{}**@{.},
<2.7mm,-10mm>*{}**@{.},
<5mm,-10mm>*{}**@{.},
<4mm,-7mm>*{^{1\ \  \dots\ \   k\ \ \qquad \ \   k+l+1\dots  \ \ n}},
<2mm,-12mm>*{_{{k+1} \ \dots\ \  k+l}},
\end{xy}}
$$
and  $n=1$, $k=0$, $m\geq 2$,
$$
d\
\resizebox{20mm}{!}{\xy
(1,5)*{\ldots},
(-13,7)*{_{1}},
(-8,7)*{_{2}},
(-3,7)*{_{3}},
(8,7)*{_{{m-1}}},
(14,7)*{_{m}},
 (0,0)*{\circ}="a",
(0,-5)*{}="0",
(-12,5)*{}="b_1",
(-8,5)*{}="b_2",
(-3,5)*{}="b_3",
(8,5)*{}="b_4",
(12,5)*{}="b_5",
\ar @{.} "a";"0" <0pt>
\ar @{.} "a";"b_2" <0pt>
\ar @{.} "a";"b_3" <0pt>
\ar @{.} "a";"b_1" <0pt>
\ar @{.} "a";"b_4" <0pt>
\ar @{.} "a";"b_5" <0pt>
\endxy}
\ =\ \sum_{k=0}^{m-2}\sum_{l=2}^{m-k}
(-1)^{k+l(m-k-l)} \ \
\resizebox{25mm}{!}{\begin{xy}
<0mm,0mm>*{\circ},
<0mm,-0.8mm>*{};<0mm,-5mm>*{}**@{-},
<-9mm,5mm>*{\ldots},
<14mm,5mm>*{\ldots},
<-0.7mm,0.3mm>*{};<-13mm,5mm>*{}**@{.},
<-0.6mm,0.5mm>*{};<-6mm,5mm>*{}**@{.},
<0.6mm,0.3mm>*{};<20mm,5mm>*{}**@{.},
<0.3mm,0.5mm>*{};<8mm,5mm>*{}**@{.},
<0mm,0.5mm>*{};<0mm,4.3mm>*{}**@{.},
<0mm,5mm>*{\circ};
<-5mm,10mm>*{}**@{.},
<-2.7mm,10mm>*{}**@{.},
<2.7mm,10mm>*{}**@{.},
<5mm,10mm>*{}**@{.},
<4mm,7mm>*{^{1\ \  \dots\ \   k\ \ \qquad \ \   k+l+1\dots  \ \ m}},
<2mm,12mm>*{_{{k+1} \ \dots\ \  k+l}},
\end{xy}}
$$
so that the formulae for the values of $d$ on corollas with $m=1$, $k\geq 1$, $n\geq 1$ and,
respectively, with $m\geq 1$, $k\geq 1$, $n= 1$, are given (modulo signs) by

\Beq\label{3: differential on Konts corollas}
d
\Ba{c}
\resizebox{20mm}{!}{\begin{xy}
 <0mm,0mm>*{\circ};
 <0mm,0mm>*{};<0mm,5mm>*{}**@{.},
 <0mm,0mm>*{};<-16mm,-5mm>*{}**@{-},
 <0mm,0mm>*{};<-11mm,-5mm>*{}**@{-},
 <0mm,0mm>*{};<-3.5mm,-5mm>*{}**@{-},
 <0mm,0mm>*{};<-6mm,-5mm>*{...}**@{},
   <0mm,0mm>*{};<-16mm,-8mm>*{^{1}}**@{},
   <0mm,0mm>*{};<-11mm,-8mm>*{^{2}}**@{},
   <0mm,0mm>*{};<-3mm,-8mm>*{^{k}}**@{},
 <0mm,0mm>*{};<16mm,-5mm>*{}**@{.},
 <0mm,0mm>*{};<8mm,-5mm>*{}**@{.},
 <0mm,0mm>*{};<3.5mm,-5mm>*{}**@{.},
 <0mm,0mm>*{};<11.6mm,-5mm>*{...}**@{},
   <0mm,0mm>*{};<17mm,-8mm>*{^{\bar{m}}}**@{},
<0mm,0mm>*{};<10mm,-8mm>*{^{\bar{2}}}**@{},
   <0mm,0mm>*{};<5mm,-8mm>*{^{\bar{1}}}**@{},
 \end{xy}}
\Ea
=
\sum_{B\varsubsetneq [k]\atop
\# B\geq 2}\ \pm
\Ba{c}
\resizebox{25mm}{!}{\begin{xy}
 <0mm,0mm>*{\circ};
 <0mm,0mm>*{};<0mm,5mm>*{}**@{.},
 <0mm,0mm>*{};<-16mm,-5mm>*{}**@{-},
 <0mm,0mm>*{};<-11mm,-5mm>*{}**@{-},
 <0mm,0mm>*{};<-3.5mm,-5mm>*{}**@{-},
 <0mm,0mm>*{};<-6mm,-5mm>*{...}**@{},
 <0mm,0mm>*{};<16mm,-5mm>*{}**@{.},
 <0mm,0mm>*{};<8mm,-5mm>*{}**@{.},
 <0mm,0mm>*{};<3.5mm,-5mm>*{}**@{.},
 <0mm,0mm>*{};<11.6mm,-5mm>*{...}**@{},
   <0mm,0mm>*{};<17mm,-8mm>*{^{\bar{m}}}**@{},
<0mm,0mm>*{};<10mm,-8mm>*{^{\bar{2}}}**@{},
   <0mm,0mm>*{};<5mm,-8mm>*{^{\bar{1}}}**@{},
<-17mm,-12mm>*{\underbrace{\ \ \ \ \ \ \ \ \ \   }},
<-17mm,-14.9mm>*{_B};
<-6mm,-7mm>*{\underbrace{\ \ \ \ \ \ \  }},
<-6mm,-10mm>*{_{[k]\setminus B}};
 (-16.5,-5.5)*{\bu}="a",
(-23,-10)*{}="b_1",
(-20,-10)*{}="b_2",
(-16,-10)*{...}="b_3",
(-12,-10)*{}="b_4",
\ar @{-} "a";"b_2" <0pt>
\ar @{-} "a";"b_1" <0pt>
\ar @{-} "a";"b_4" <0pt>
 \end{xy}}
\Ea
\ \ + \ \
\sum_{k, l, [k]=I_1\sqcup I_2\atop
{|I_1| \geq 0 \atop
|I_2| + l\geq 2}}
\pm
\Ba{c}
\resizebox{27mm}{!}{\begin{xy}
 <0mm,0mm>*{\circ};
 <0mm,0mm>*{};<0mm,6mm>*{}**@{.},
 <0mm,0mm>*{};<-16mm,-6mm>*{}**@{-},
 <0mm,0mm>*{};<-11mm,-6mm>*{}**@{-},
 <0mm,0mm>*{};<-3.5mm,-6mm>*{}**@{-},
 <0mm,0mm>*{};<-6mm,-6mm>*{...}**@{},
<0mm,0mm>*{};<2mm,-9mm>*{^{\bar{1}}}**@{},
<0mm,0mm>*{};<6mm,-9mm>*{^{\bar{k}}}**@{},
<0mm,0mm>*{};<19mm,-9mm>*{^{\overline{k+l+1}}}**@{},
<0mm,0mm>*{};<28mm,-9mm>*{^{\overline{m}}}**@{},
<0mm,0mm>*{};<13mm,-16.6mm>*{^{\overline{k+1}}}**@{},
<0mm,0mm>*{};<20mm,-16.6mm>*{^{\overline{k+l}}}**@{},
 <0mm,0mm>*{};<11mm,-6mm>*{}**@{.},
 <0mm,0mm>*{};<6mm,-6mm>*{}**@{.},
 <0mm,0mm>*{};<2mm,-6mm>*{}**@{.},
 <0mm,0mm>*{};<17mm,-6mm>*{}**@{.},
 <0mm,0mm>*{};<25mm,-6mm>*{}**@{.},
 <0mm,0mm>*{};<4mm,-6mm>*{...}**@{},
<0mm,0mm>*{};<20mm,-6mm>*{...}**@{},
<6.5mm,-16mm>*{\underbrace{\ \ \ \ \   }_{I_2}},
<-10mm,-9mm>*{\underbrace{\ \ \ \ \ \ \ \ \ \ \ \   }_{I_1}},
 %
 (11,-7)*{\circ}="a",
(4,-13)*{}="b_1",
(9,-13)*{}="b_2",
(16,-13)*{...},
(7,-13)*{...},
(13,-13)*{}="b_3",
(19,-13)*{}="b_4",
\ar @{-} "a";"b_2" <0pt>
\ar @{.} "a";"b_3" <0pt>
\ar @{-} "a";"b_1" <0pt>
\ar @{.} "a";"b_4" <0pt>
 \end{xy}}
\Ea
\Eeq

\mip


\Beqr
d
\Ba{c}
\resizebox{16mm}{!}{\begin{xy}
 <0mm,0mm>*{\circ};
 <0mm,0mm>*{};<-4mm,6mm>*{}**@{.},
 <0mm,0mm>*{};<-7mm,6mm>*{}**@{.},
 <0mm,0mm>*{};<4mm,6mm>*{}**@{.},
 <0mm,0mm>*{};<7mm,6mm>*{}**@{.},
 <0mm,0mm>*{};<-16mm,-5mm>*{}**@{-},
 <0mm,0mm>*{};<-11mm,-5mm>*{}**@{-},
 <0mm,0mm>*{};<-3.5mm,-5mm>*{}**@{-},
 <0mm,0mm>*{};<-6mm,-5mm>*{...}**@{},
   <0mm,0mm>*{};<-16mm,-8mm>*{^{1}}**@{},
   <0mm,0mm>*{};<-11mm,-8mm>*{^{2}}**@{},
   <0mm,0mm>*{};<-3mm,-8mm>*{^{k}}**@{},
 <0mm,0mm>*{};<3.5mm,-5mm>*{}**@{.},
 <0mm,5mm>*{...}**@{},
<5mm,-8mm>*{^{\bar{1}}}**@{},
 <-8mm,7mm>*{^{1}}**@{},
 <-4mm,7mm>*{^{2}}**@{},
 <8mm,7mm>*{^{n}}**@{},
 \end{xy}}
\Ea
&=&
\sum_{B\varsubsetneq [k]\atop
\# B\geq 2}\ \pm
\Ba{c}
\resizebox{20mm}{!}{\begin{xy}
 <0mm,0mm>*{\circ};
 <0mm,0mm>*{};<-4mm,6mm>*{}**@{.},
 <0mm,0mm>*{};<-7mm,6mm>*{}**@{.},
 <0mm,0mm>*{};<4mm,6mm>*{}**@{.},
 <0mm,0mm>*{};<7mm,6mm>*{}**@{.},
  <0mm,5mm>*{...}**@{},
 <0mm,0mm>*{};<-16mm,-5mm>*{}**@{-},
 <0mm,0mm>*{};<-11mm,-5mm>*{}**@{-},
 <0mm,0mm>*{};<-3.5mm,-5mm>*{}**@{-},
 <0mm,0mm>*{};<-6mm,-5mm>*{...}**@{},
 <0mm,0mm>*{};<3.5mm,-5mm>*{}**@{.},
   <0mm,0mm>*{};<5mm,-8mm>*{^{\bar{1}}}**@{},
<-17mm,-12mm>*{\underbrace{\ \ \ \ \ \ \ \ \ \   }},
<-17mm,-14.9mm>*{_B};
<-6mm,-7mm>*{\underbrace{\ \ \ \ \ \ \  }},
<-6mm,-10mm>*{_{[k]\setminus B}};
 (-16.5,-5.5)*{\bu}="a",
(-23,-10)*{}="b_1",
(-20,-10)*{}="b_2",
(-16,-10)*{...}="b_3",
(-12,-10)*{}="b_4",
\ar @{-} "a";"b_2" <0pt>
\ar @{-} "a";"b_1" <0pt>
\ar @{-} "a";"b_4" <0pt>
 \end{xy}}
\Ea 
\ \ + \ \ \sum_{k, l, [k]=I_1\sqcup I_2\atop
{|I_1|, |I_2| \geq 1 }}
\pm
\Ba{c}
\resizebox{21mm}{!}{\begin{xy}
 <0mm,0mm>*{\circ};
 <0mm,0mm>*{};<-3mm,6mm>*{}**@{.},
 <0mm,0mm>*{};<-6mm,6mm>*{}**@{.},
 <0mm,0mm>*{};<3mm,6mm>*{}**@{.},
 <0mm,0mm>*{};<6mm,6mm>*{}**@{.},
  <0mm,5mm>*{...}**@{},
 <0mm,0mm>*{};<-16mm,-6mm>*{}**@{-},
 <0mm,0mm>*{};<-11mm,-6mm>*{}**@{-},
 <0mm,0mm>*{};<-3.5mm,-6mm>*{}**@{-},
 <0mm,0mm>*{};<-6mm,-6mm>*{...}**@{},
 <0mm,0mm>*{};<6.4mm,-6mm>*{}**@{.},
<1mm,-16mm>*{\underbrace{\ \ \ \ \  \ \ \ \ \   }_{I_2}},
<-10mm,-9mm>*{\underbrace{\ \ \ \ \ \ \ \ \ \ \ \   }_{I_1}},
<0mm,9.7mm>*{\overbrace{\ \ \ \ \  \ \ \ \ \   }^{J_1}},
<8.6mm,2.7mm>*{\overbrace{\ \ \ \ \  \ \ \   }^{J_2}},
 (7,-7)*{\circ}="a",
(-3,-13)*{}="b_1",
(2,-13)*{}="b_2",
(7.9,-2)*{...},
(0,-13)*{...},
(6,-13)*{}="b_3",
(12,-13)*{}="b_4",
(5,-1)*{}="u1",
(12,-1)*{}="u2",
\ar @{-} "a";"b_2" <0pt>
\ar @{-} "a";"b_3" <0pt>
\ar @{-} "a";"b_1" <0pt>
\ar @{.} "a";"b_4" <0pt>
\ar @{.} "a";"u1" <0pt>
\ar @{.} "a";"u2" <0pt>
 \end{xy}}
\Ea  \nonumber
\Eeqr
We are interested in a few special cases of these formulae for which we give precise expressions,
$$
d\Ba{c}
\xy
(-5,-6)*{_{1}},
(-2,-6)*{_{2}},
(4,-6)*{_{3}},
 (0,0)*{\circ}="a",
(-2,-4)*{}="b1",
(-5,-4)*{}="b2",
 (0,4)*{}="u",
 (3,-4)*{}="b3",

\ar @{.} "a";"u" <0pt>
\ar @{-} "a";"b1" <0pt>
\ar @{-} "a";"b2" <0pt>
\ar @{.} "a";"b3" <0pt>
\endxy\Ea
=
-\Ba{c}
\xy
(-6,-8)*{_{1}},
(0,-8)*{_{2}},
(4,-4.5)*{_{3}},
 (0,0)*{\circ}="a",
(-3,-3)*{\bu}="b1",
(3,-3)*{}="b2",
 (0,3)*{}="u",
 (-6,-6)*{}="d1",
 (0,-6)*{}="d2",
\ar @{.} "a";"u" <0pt>
\ar @{-} "a";"b1" <0pt>
\ar @{.} "a";"b2" <0pt>
\ar @{-} "b1";"d1" <0pt>
\ar @{-} "b1";"d2" <0pt>
\endxy\Ea  +
\Ba{c}\xy
(-4,-4.5)*{_{1}},
(0,-8)*{_{2}},
(7,-8)*{_{3}},
 (0,0)*{\bu}="a",
(-3,-3)*{}="b2",
(3,-3)*{\bu}="b1",
 (0,3)*{}="u",
 (0,-6)*{}="d1",
 (6,-6)*{}="d2",
\ar @{.} "a";"u" <0pt>
\ar @{.} "a";"b1" <0pt>
\ar @{-} "a";"b2" <0pt>
\ar @{-} "b1";"d1" <0pt>
\ar @{.} "b1";"d2" <0pt>
\endxy\Ea
 -
\Ba{c}\xy
(-4,-4.5)*{_{2}},
(0,-8)*{_{1}},
(7,-8)*{_{3}},
 (0,0)*{\bu}="a",
(-3,-3)*{}="b2",
(3,-3)*{\bu}="b1",
 (0,3)*{}="u",
 (0,-6)*{}="d1",
 (6,-6)*{}="d2",
\ar @{.} "a";"u" <0pt>
\ar @{.} "a";"b1" <0pt>
\ar @{-} "a";"b2" <0pt>
\ar @{-} "b1";"d1" <0pt>
\ar @{.} "b1";"d2" <0pt>
\endxy\Ea,
$$
$$
d\Ba{c}
\xy
(5,-6)*{_{3}},
(2,-6)*{_{2}},
(-4,-6)*{_{1}},
 (0,0)*{\circ}="a",
(2,-4)*{}="b1",
(5,-4)*{}="b2",
 (0,4)*{}="u",
 (-3,-4)*{}="b3",

\ar @{.} "a";"u" <0pt>
\ar @{.} "a";"b1" <0pt>
\ar @{.} "a";"b2" <0pt>
\ar @{-} "a";"b3" <0pt>
\endxy\Ea
=
\Ba{c}\xy
(-4,-4.5)*{_{1}},
(0,-8)*{_{2}},
(7,-8)*{_{3}},
 (0,0)*{\bu}="a",
(-3,-3)*{}="b2",
(3,-3)*{\circ}="b1",
 (0,3)*{}="u",
 (0,-6)*{}="d1",
 (6,-6)*{}="d2",
\ar @{.} "a";"u" <0pt>
\ar @{.} "a";"b1" <0pt>
\ar @{-} "a";"b2" <0pt>
\ar @{.} "b1";"d1" <0pt>
\ar @{.} "b1";"d2" <0pt>
\endxy\Ea   -
\Ba{c}\xy
(-6,-8)*{_{1}},
(0,-8)*{_{2}},
(4,-4.5)*{_{3}},
 (0,0)*{\circ}="a",
(-3,-3)*{\bu}="b1",
(3,-3)*{}="b2",
 (0,3)*{}="u",
 (-6,-6)*{}="d1",
 (0,-6)*{}="d2",
\ar @{.} "a";"u" <0pt>
\ar @{.} "a";"b1" <0pt>
\ar @{.} "a";"b2" <0pt>
\ar @{-} "b1";"d1" <0pt>
\ar @{.} "b1";"d2" <0pt>
\endxy\Ea
  -
\Ba{c}\xy
(-4,-4.5)*{_{2}},
(0,-8)*{_{1}},
(7,-8)*{_{3}},
 (0,0)*{\circ}="a",
(-3,-3)*{}="b2",
(3,-3)*{\bu}="b1",
 (0,3)*{}="u",
 (0,-6)*{}="d1",
 (6,-6)*{}="d2",
\ar @{.} "a";"u" <0pt>
\ar @{.} "a";"b1" <0pt>
\ar @{.} "a";"b2" <0pt>
\ar @{-} "b1";"d1" <0pt>
\ar @{.} "b1";"d2" <0pt>
\endxy\Ea,
$$
$$
d\Ba{c}
\xy
(-3,6)*{_{1}},
(3,6)*{_{2}},
 (0,0)*{\circ}="a",
(-3,-4)*{}="b1",
(3,-4)*{}="b2",
 (-3,4)*{}="u",
 (3,4)*{}="b3",

\ar @{.} "a";"u" <0pt>
\ar @{-} "a";"b1" <0pt>
\ar @{.} "a";"b2" <0pt>
\ar @{.} "a";"b3" <0pt>
\endxy\Ea
=
\Ba{c}
\xy
(-3,9)*{_{1}},
(3,9)*{_{2}},
 (0,0)*{\bu}="a",
(-3,-3)*{}="b2",
(3,-3)*{}="b1",
 (0,4)*{\circ}="u",
 (-3,7)*{}="u1",
 (3,7)*{}="u2",
\ar @{.} "a";"u" <0pt>
\ar @{.} "a";"b1" <0pt>
\ar @{-} "a";"b2" <0pt>
\ar @{.} "u";"u1" <0pt>
\ar @{.} "u";"u2" <0pt>
\endxy\Ea
  -
\Ba{c}\xy
(-3,9)*{_{1}},
(3,6)*{_{2}},
 (0,0)*{\circ}="a",
(-3,4)*{\bu}="b1",
(3,4)*{}="b2",
 (0,-3)*{}="d",
 (-3,7)*{}="u",
 (-6,0)*{}="l",
\ar @{.} "a";"d" <0pt>
\ar @{.} "a";"b1" <0pt>
\ar @{.} "a";"b2" <0pt>
\ar @{-} "b1";"l" <0pt>
\ar @{.} "b1";"u" <0pt>
\endxy\Ea
 -
\Ba{c}\xy
(-3,6)*{_{1}},
(3,9)*{_{2}},
 (0,0)*{\circ}="a",
(3,4)*{\bu}="b1",
(-3,4)*{}="b2",
 (0,-3)*{}="d",
 (3,7)*{}="u",
 (6,0)*{}="l",
\ar @{.} "a";"d" <0pt>
\ar @{.} "a";"b1" <0pt>
\ar @{.} "a";"b2" <0pt>
\ar @{-} "b1";"l" <0pt>
\ar @{.} "b1";"u" <0pt>
\endxy\Ea.
$$
These equations say that there exists a canonical morphism of dg props,
\Beq\label{4: p from Qua indty to Qua}
p: \cQ ua_\infty \lon \cQ ua
\Eeq
which sends to zero all generators of $\cQ ua_\infty$ except  $\Ba{c}\xy
 <0mm,0.55mm>*{};<0mm,3.5mm>*{}**@{-},
 <0.5mm,-0.5mm>*{};<2.2mm,-2.2mm>*{}**@{-},
 <-0.48mm,-0.48mm>*{};<-2.2mm,-2.2mm>*{}**@{-},
 <0mm,0mm>*{\bu};<0mm,0mm>*{}**@{},
 <0.5mm,-0.5mm>*{};<2.7mm,-3.2mm>*{_2}**@{},
 <-0.48mm,-0.48mm>*{};<-2.7mm,-3.2mm>*{_1}**@{},
 \endxy\Ea$,
 $\Ba{c}\begin{xy}
 <0mm,0.66mm>*{};<0mm,3mm>*{}**@{.},
 <0.39mm,-0.39mm>*{};<2.2mm,-2.2mm>*{}**@{.},
 <-0.35mm,-0.35mm>*{};<-2.2mm,-2.2mm>*{}**@{-},
 <0mm,0mm>*{\bu};<0mm,0mm>*{}**@{},
   <0.39mm,-0.39mm>*{};<2.9mm,-4mm>*{^2}**@{},
   <-0.35mm,-0.35mm>*{};<-2.8mm,-4mm>*{^1}**@{},
\end{xy}\Ea$,
 $\Ba{c}\begin{xy}
 <0mm,-0.55mm>*{};<0mm,-3.5mm>*{}**@{.},
 <0.5mm,0.5mm>*{};<2.2mm,2.2mm>*{}**@{.},
 <-0.48mm,0.48mm>*{};<-2.2mm,2.2mm>*{}**@{.},
 <0mm,0mm>*{\circ};<0mm,0mm>*{}**@{},
 <0.5mm,0.5mm>*{};<2.7mm,2.8mm>*{^2}**@{},
 <-0.48mm,0.48mm>*{};<-2.7mm,2.8mm>*{^1}**@{},
 \end{xy}\Ea$ and $\Ba{c}\begin{xy}
 <0mm,0.66mm>*{};<0mm,4mm>*{}**@{.},
 <0.39mm,-0.39mm>*{};<2.2mm,-2.2mm>*{}**@{.},
 <-0.35mm,-0.35mm>*{};<-2.2mm,-2.2mm>*{}**@{.},
 <0mm,0mm>*{\circ};<0mm,0mm>*{}**@{},
   <0.39mm,-0.39mm>*{};<2.9mm,-4mm>*{^2}**@{},
   <-0.35mm,-0.35mm>*{};<-2.8mm,-4mm>*{^1}**@{},
\end{xy}\Ea$.

\begin{theorem}
The morphism $p$ is a quasi-isomorphism.
\end{theorem}
\begin{proof} The differential in the complex $(\cQ ua_\infty,d)$ is connected (that is, its value on any generator is given by a sum of connected graphs) and preserves the path grading\footnote{To any graph  $\Ga\in \cP(m,n)$ of a free prop generated by by an $\bS$-bimodule $E=\{E(m,n)\}_{m+n\geq3, m,n\geq 1}$ we can associate the number $|\Ga|_{path}\leq mn$ of directed paths connecting input legs of $\Ga$
with output ones.}. Hence by Theorem 27 in \cite{MaVo}, the genus filtration
$$
F_p^{genus}:=\left\{ \mathrm{span}\langle \Ga\in \cQ ua_\infty(m,n) \rangle\ :\ \text{the genus of}\ \Ga>-p\right\}
$$
induces a {\em converging}\, spectral sequence whose first term we denote by $Gr^{gen}(\cQ ua_\infty)$. The projection map $p$ respects genus filtrations of both sides, and hence induces a morphism of complexes
$$
p': Gr^{gen}(\cQ ua_\infty) \lon Gr^{gen}(\cQ ua).
$$
The complex $Gr^{gen}(\cQ ua_\infty)$ admits a path filtration
$$
F_p^{path}:=\left\{ \mathrm{span}\langle \Ga\in Gr^{gen}(\cQ ua_\infty)(m,n) \rangle\ : \ \ |\Ga|_{path}<p\right\}
$$
with the converging spectral sequence whose first term we denote by $Gr^{gen, path}(\cQ ua_\infty)$.
The map $p'$ above preserves path filtrations of both sides and hence induces a morphism
of the associated graded complexes,
$$
p'': Gr^{gen,path}(\cQ ua_\infty) \lon Gr^{gen,path}(\cQ ua)
$$
The r.h.s.\ equals the prop enveloping of a quadratic  $\frac{1}{2}$-prop $\cP_0$ generated by
 \Bi
 \item[(i)]
 the {\em Koszul}\, operad $\cA ss$ with generators  $\begin{xy}
 <0mm,-0.55mm>*{};<0mm,-3.5mm>*{}**@{.},
 <0.5mm,0.5mm>*{};<2.2mm,2.2mm>*{}**@{.},
 <-0.48mm,0.48mm>*{};<-2.2mm,2.2mm>*{}**@{.},
 <0mm,0mm>*{\circ};<0mm,0mm>*{}**@{},
 <0.5mm,0.5mm>*{};<2.7mm,2.8mm>*{^2}**@{},
 <-0.48mm,0.48mm>*{};<-2.7mm,2.8mm>*{^1}**@{},
 \end{xy}$,
 \item[(ii)] the 2-coloured {\em Koszul}\, \cite{A,HL} operad of non-commutative Gerstenhaber algebras (also known as operad
 of Leibniz pairs) with generators  $\Ba{c}\xy
 <0mm,0.55mm>*{};<0mm,3.5mm>*{}**@{-},
 <0.5mm,-0.5mm>*{};<2.2mm,-2.2mm>*{}**@{-},
 <-0.48mm,-0.48mm>*{};<-2.2mm,-2.2mm>*{}**@{-},
 <0mm,0mm>*{\bu};<0mm,0mm>*{}**@{},
 <0.5mm,-0.5mm>*{};<2.7mm,-3.2mm>*{_2}**@{},
 <-0.48mm,-0.48mm>*{};<-2.7mm,-3.2mm>*{_1}**@{},
 \endxy\Ea$, $\Ba{c}\begin{xy}
 <0mm,0.66mm>*{};<0mm,3mm>*{}**@{.},
 <0.39mm,-0.39mm>*{};<2.2mm,-2.2mm>*{}**@{.},
 <-0.35mm,-0.35mm>*{};<-2.2mm,-2.2mm>*{}**@{-},
 <0mm,0mm>*{\bu};<0mm,0mm>*{}**@{},
   <0.39mm,-0.39mm>*{};<2.9mm,-4mm>*{^2}**@{},
   <-0.35mm,-0.35mm>*{};<-2.8mm,-4mm>*{^1}**@{},
\end{xy}\Ea$,
 and $\Ba{c}\begin{xy}
 <0mm,0.66mm>*{};<0mm,4mm>*{}**@{.},
 <0.39mm,-0.39mm>*{};<2.2mm,-2.2mm>*{}**@{.},
 <-0.35mm,-0.35mm>*{};<-2.2mm,-2.2mm>*{}**@{.},
 <0mm,0mm>*{\circ};<0mm,0mm>*{}**@{},
   <0.39mm,-0.39mm>*{};<2.9mm,-4mm>*{^2}**@{},
   <-0.35mm,-0.35mm>*{};<-2.8mm,-4mm>*{^1}**@{},
\end{xy}\Ea$
\Ei
modulo $\frac{1}{2}$-prop relations
$$
\Ba{c} \begin{xy}
 <0mm,2.47mm>*{};<0mm,-0.5mm>*{}**@{.},
 <0.5mm,3.5mm>*{};<2.2mm,5.2mm>*{}**@{.},
 <-0.48mm,3.48mm>*{};<-2.2mm,5.2mm>*{}**@{.},
 <0mm,3mm>*{\circ};<0mm,3mm>*{}**@{},
  <0mm,-0.8mm>*{\circ};<0mm,-0.8mm>*{}**@{},
<0mm,-0.8mm>*{};<-2.2mm,-3.5mm>*{}**@{.},
 <0mm,-0.8mm>*{};<2.2mm,-3.5mm>*{}**@{.},
     <0.5mm,3.5mm>*{};<2.8mm,5.7mm>*{^2}**@{},
     <-0.48mm,3.48mm>*{};<-2.8mm,5.7mm>*{^1}**@{},
   <0mm,-0.8mm>*{};<-2.7mm,-5.2mm>*{^1}**@{},
   <0mm,-0.8mm>*{};<2.7mm,-5.2mm>*{^2}**@{},
\end{xy}\Ea=0 \ \ \ \ \ \mbox{and} \ \ \ \ \  \Ba{c}
\xy
(-3,9)*{_{1}},
(3,9)*{_{2}},
 (0,0)*{\bu}="a",
(-3,-3)*{}="b2",
(3,-3)*{}="b1",
 (0,4)*{\circ}="u",
 (-3,7)*{}="u1",
 (3,7)*{}="u2",
\ar @{.} "a";"u" <0pt>
\ar @{.} "a";"b1" <0pt>
\ar @{-} "a";"b2" <0pt>
\ar @{.} "u";"u1" <0pt>
\ar @{.} "u";"u2" <0pt>
\endxy\Ea=0.
$$
It is easy to compute the minimal resolution of the $\frac{1}{2}$-prop $\cP_0$ and notice that its prop
enveloping equals precisely  $Gr^{gen,path}(\cQ ua_\infty)$. Hence the morphism $p''$ is a quasi-isomorphism which implies that the morphism $p'$ is a quasi-isomorphism which in tern implies
that the morphism $p$ is a quasi-isomorphism.
\end{proof}

\subsection{Formality morphisms as Maurer-Cartan elements}\label{4.6 subsect on formalities as MC elements} The morphism (\ref{3: i_1: Lie to Bra})
factors through the inclusion $\cG ra^{or}\rar \cB\cG ra^{or}$, and extends  to the following one.

\subsubsection{\bf Proposition}\label{3: Prop on map from AssB to Bra}
{\em There is a morphism of props
\Beq\label{3: i_Gac to Bra}
i: \cQ ua \lon \cB ra^{or}
\Eeq
given on the generators as follows,}
$$
i\left(\begin{xy}
 <0mm,-0.55mm>*{};<0mm,-3.5mm>*{}**@{.},
 <0.5mm,0.5mm>*{};<2.2mm,2.2mm>*{}**@{.},
 <-0.48mm,0.48mm>*{};<-2.2mm,2.2mm>*{}**@{.},
 <0mm,0mm>*{\circ};<0mm,0mm>*{}**@{},
 <0.5mm,0.5mm>*{};<2.7mm,2.8mm>*{^2}**@{},
 <-0.48mm,0.48mm>*{};<-2.7mm,2.8mm>*{^1}**@{},
 \end{xy}\right)=
 \Ba{c}\resizebox{8mm}{!}{ \xy
(-3,9)*{^1},
(3,9)*{^2},
 (0,2)*{\circ}="a",
(-3,7)*{\circ}="b_1",
(3,7)*{\circ}="b_2",
 \endxy}\Ea  \ \ \ \ \ \ \ \ \ \ , \ \ \ \ \ \ \ \ \ \ \
i\left(\begin{xy}
 <0mm,0.66mm>*{};<0mm,4mm>*{}**@{.},
 <0.39mm,-0.39mm>*{};<2.2mm,-2.2mm>*{}**@{.},
 <-0.35mm,-0.35mm>*{};<-2.2mm,-2.2mm>*{}**@{.},
 <0mm,0mm>*{\circ};<0mm,0mm>*{}**@{},
   <0.39mm,-0.39mm>*{};<2.9mm,-4mm>*{^2}**@{},
   <-0.35mm,-0.35mm>*{};<-2.8mm,-4mm>*{^1}**@{},
\end{xy}\right)=\Ba{c}\resizebox{8mm}{!}{ \xy
(-3,0)*{_1},
(3,0)*{_2},
 (0,7)*{\circ}="a",
(-3,2)*{\circ}="b_1",
(3,2)*{\circ}="b_2",
 \endxy}\Ea
$$
$$
i\left(\Ba{c}
\xy
 <0mm,0.55mm>*{};<0mm,3.5mm>*{}**@{-},
 <0.5mm,-0.5mm>*{};<2.2mm,-2.2mm>*{}**@{-},
 <-0.48mm,-0.48mm>*{};<-2.2mm,-2.2mm>*{}**@{-},
 <0mm,0mm>*{\bu};<0mm,0mm>*{}**@{},
 <0.5mm,-0.5mm>*{};<2.7mm,-3.2mm>*{_2}**@{},
 <-0.48mm,-0.48mm>*{};<-2.7mm,-3.2mm>*{_1}**@{},
 \endxy\Ea\right)\ =\ \  \xy
(0,2)*{_{1}},
(7,2)*{_{2}},
 (0,0)*{\bullet}="a",
(7,0)*{\bu}="b",
\ar @{->} "a";"b" <0pt>
\endxy\
- \
\xy
(0,2)*{_{2}},
(7,2)*{_{1}},
 (0,0)*{\bullet}="a",
(7,0)*{\bu}="b",
\ar @{->} "a";"b" <0pt>
\endxy\ \ \ \ \ \  ,\ \ \ \ \ \ \
i\left(\begin{xy}
 <0mm,0.66mm>*{};<0mm,3mm>*{}**@{.},
 <0.39mm,-0.39mm>*{};<2.2mm,-2.2mm>*{}**@{.},
 <-0.35mm,-0.35mm>*{};<-2.2mm,-2.2mm>*{}**@{-},
 <0mm,0mm>*{\bu};<0mm,0mm>*{}**@{},
\end{xy}\right)\ =\ \ \xy
 (0,0)*{\bullet}="a",
(0,5)*{\circ}="u",
(0,-5)*{\circ}="d",
\ar @{->} "d";"a" <0pt>
\ar @{->} "a";"u" <0pt>
\endxy \ .
$$

\begin{proof} Relation (\ref{3: Gac prop rel 1}) follows from example (\ref{3: Example of prop coomp}) and the fact that
$$
\Ba{rccc}
\circ: &\cP\cG ra(1;1,1)\ot \cP\cG ra(1;1,1)  &\lon & \cP\cG ra(2;1,1)
\\
& \Ba{c}\xy
 (0,0)*{\bullet}="a",
(0,5)*{\circ}="u",
(0,-5)*{\circ}="d",
\ar @{->} "d";"a" <0pt>
\ar @{->} "a";"u" <0pt>
\endxy\Ea \ot
\Ba{c}\xy
 (0,0)*{\bullet}="a",
(0,5)*{\circ}="u",
(0,-5)*{\circ}="d",
\ar @{->} "d";"a" <0pt>
\ar @{->} "a";"u" <0pt>
\endxy\Ea
 &\lon &  \Ba{c}\xy
(-2,0)*{_{2}},
(-2,5)*{_{1}},
 (0,0)*{\bullet}="a",
(0,5)*{\bu}="b",
(0,10)*{\circ}="u",
(0,-5)*{\circ}="d",
\ar @{->} "d";"a" <0pt>
\ar @{->} "a";"b" <0pt>
\ar @{->} "b";"u" <0pt>
\endxy
\Ea  + \Ba{c}\xy
(-4.8,1)*{_{1}},
(4.8,1)*{_{2}},
 (-3,0)*{\bullet}="a",
 (3,0)*{\bullet}="b",
(0,5)*{\circ}="u",
(0,-5)*{\circ}="d",
\ar @{->} "d";"a" <0pt>
\ar @{->} "a";"u" <0pt>
\ar @{->} "d";"b" <0pt>
\ar @{->} "b";"u" <0pt>
\endxy\Ea
\Ea
$$

Analogously one checks that all other relations (\ref{2: bialgebra relations}) and (\ref{3: Gac prop rel 2}) are also respected by the map $i$.
\end{proof}

The composition of maps $p: \cQ ua_\infty \rar \cQ ua$ and $i$ from  (\ref{3: i_Gac to Bra}) gives us a canonical morphism of props,
$$
q_0: \cQ ua_\infty\lon \cB ra^{or},
$$
so that we can consider a graph complex (in fact, a $\caL ie_\infty$ algebra)\footnote{
The  symbol $\oplus$ means here a direct sum of graded vector spaces, {\em not}\,
 a direct sum  of $\caL ie_\infty$-algebras. It is worth noting, however, that each summand is a $\caL ie_\infty$
 subalgebra (with $\mathsf{fGC_3^{or}}$ being an ordinary dg Lie algebra). Moreover, we can (and will) assume without loss of generality that $\mathsf{BGC}^{or}$ is spanned by graphs with at least trivalent vertices (cf.\ \cite{Wi}).}
$$
\Def(\cQ ua_\infty \stackrel{q_0}{\rar} \cB ra)=: \mathsf{fGC}_3^{or} \oplus \mathsf{BGC}^{or}
$$
controlling deformations of the morphism $q_0$. Elements of $\mathsf{BGC}$ can be understood as (linear combinations of) oriented directed graphs $\Ga$ from $\sG_{k;m,n}$ which have labels of {\em black}\, vertices skew-symmetrized (so that labelling of {\em black}\, vertices can be omitted in the pictures), and which are assigned the homological degree
\Beqr
|\Ga|&=&3|V_{int}(\Ga)| + |V_{in}(\Ga)|+ |V_{out}(\Ga)|  -2|E_{int}(\Ga)| - |E_{in}(\Ga)|- |E_{out}(\Ga)| -2\nonumber\\
&=& \sum_{i=1}^{|V_{int}(\Ga)|} (3-|v_i|) + |V_{in}(\Ga)|+ |V_{out}(\Ga)| -2.\label{3: degree in BC}
\Eeqr
where $|v_i|$ stands for the valency (the total number of input and output half-edges)
of the $i$-th internal vertex.

%

\subsection{Universal formality maps}
A generic Maurer-Cartan element in the $\caL ie_\infty$ algebra
$\Def(\cQ ua_\infty \stackrel{q_0}{\rar} \cB ra^{or})$ (or its oriented version)
is a direct
sum
$$
(q', q) \in  \mathsf{fGC}_3^{or} \oplus \mathsf{BGC}^{or},
$$
where $q'$ is responsible for the deformation of the map $i_1$ in (\ref{3: i_1: Lie to Bra}).
In this paper we are interested in  deformations which keeps $i_0$ fixed (speaking plainly,
this means that we want to preserve the Poisson brackets $\{\ ,\ \}$ in $\fg_V$ which
define strongly homotopy bialgebras).  This leads us to the following notion.

\subsubsection{\bf Definitions} A Maurer-Cartan element of the $\caL ie_\infty$
algebra $\Def(\cQ ua_\infty \stackrel{q_0}{\rar} \cB ra^{or})$  of the form $(0,q)$, that is,  a  generic Maurer-Cartan element $q$
 of its $\caL ie_\infty$ subalgebra $\mathsf{BGC}^{or}$  is called a
 {\em universal quantization}\, of arbitrary (possibly, infinite-dimensional) strongly homotopy Lie bialgebras. A universal quantization $q\in \cM\cC( \mathsf{BGC}^{or})$
 is a called
 a {\em universal formality map}\, if the associated
 morphism of dg props,
$$
F: \cQ ua_\infty \stackrel{q_0+q}{\lon} \cB ra^{or},
$$
 satisfies the following condition,
 \Beq\label{4: universal formality condition}
 F\left(\resizebox{13mm}{!}{ \xy
 (0,7)*{\overbrace{\ \ \ \  \ \ \ \ \ \ \ \ \ \ }},
 (0,9)*{^m},
 (0,3)*{^{...}},
 (5.6,-5.5)*{^{...}},
 (6,-7)*{\underbrace{  \ \ \ \ \ \ }},
 (6,-9)*{_n},
 (0,0)*{\circ}="0",
(-7,5)*{}="u_1",
(-4,5)*{}="u_2",
(4,5)*{}="u_3",
(7,5)*{}="u_4",
(2,-5)*{}="d_1",
(3.6,-5)*{}="d_2",
(9,-5)*{}="d_3",
(-4,-5)*{}="s_1",
\ar @{.} "0";"u_1" <0pt>
\ar @{.} "0";"u_2" <0pt>
\ar @{.} "0";"u_3" <0pt>
\ar @{.} "0";"u_4" <0pt>
\ar @{.} "0";"d_1" <0pt>
\ar @{.} "0";"d_2" <0pt>
\ar @{.} "0";"d_3" <0pt>
\ar @{-} "0";"s_1" <0pt>
\endxy}\right)=\frac{1}{m!n!}\Ba{c}\resizebox{16mm}{!}{\xy
(0,7.5)*{\overbrace{\ \ \ \ \ \ \ \ \  \ \ \ \ \ \ \ \ \ \ }},
 (0,9.5)*{^m},
 (0,-7.5)*{\underbrace{\ \ \ \ \ \ \ \ \ \ \ \  \ \ \ \ \ \ }},
 (0,-9.9)*{_n},
  (-6,5)*{...},
  (-6,-5)*{...},
 (-3,5)*{\circ}="u1",
  (-3,-5)*{\circ}="d1",
  (-6,5)*{...},
  (-6,-5)*{...},
  (-9,5)*{\circ}="u2",
  (-9,-5)*{\circ}="d2",
 (3,5)*{\circ}="u3",
  (3,-5)*{\circ}="d3",
  (6,5)*{...},
  (6,-5)*{...},
  (9,5)*{\circ}="u4",
  (9,-5)*{\circ}="d4",
 (0,0)*{\bullet}="a",
\ar @{->} "d1";"a" <0pt>
\ar @{->} "a";"u1" <0pt>
\ar @{->} "d2";"a" <0pt>
\ar @{->} "a";"u2" <0pt>
\ar @{->} "d3";"a" <0pt>
\ar @{->} "a";"u3" <0pt>
\ar @{->} "d4";"a" <0pt>
\ar @{->} "a";"u4" <0pt>
\endxy}\Ea
 \Eeq
where
the graph in the r.h.s.\ is assumed to be skew-symmetrized over the labels of in- and out-vertices.

\sip

The set of universal formality maps is denoted by $\cM\cC_{\mathrm{form}}$.

\subsection{On $GRT_1$ action}
The $\caL ie_\infty$ algebra structure
$$
\left\{\mu_n: \wedge^n \Def(\cQ ua_\infty \stackrel{q_0}{\rar} \cB ra^{or}) \rar \Def(\cQ ua_\infty \stackrel{q_0}{\rar} \cP\cG ra^{or})[2-n]\right\}_{n\geq 1}
$$
 satisfies the condition that the maps $\mu_n$ restricted to the subspace
 $\mathsf{fGC_3^{or}}\ot \wedge^{n-1} \Def(\cQ ua_\infty \stackrel{q_0}{\rar} \cP\cG ra^{or})$ vanish for $n\geq 3$. This means that the map
$$
\Ba{rccc}
Ad: &  \mathsf{fGC_3^{or}}\ot  \mathsf{BGC^{or}} &\lon & \mathsf{BGC^{or}}\\
    &        \ga\ot \Ga  & \lon & ad_\ga\Ga:= \mu_2(\ga, \Ga)
    \Ea
$$
gives us a injective morphism of Lie algebras,
\Beq\label{4: action of Lie alg GC_3}
\Ba{ccc}
Z\mathsf{fGC_3^{or}} & \lon & \mathrm{Der}(\mathsf{BGC^{or}})\\
      \ga           & \lon & ad_\ga
\Ea
\Eeq
where $Z\mathsf{fGC_3^{or}}\subset \mathsf{fGC_3^{or}}$ is the graded Lie subalgebra of co-cycles,
and  $\mathrm{Der}(\mathsf{BGC^{or}})$ is the Lie algebra of derivations
 of the
$\caL ie_\infty$  algebra $\mathsf{BGC^{or}}$.
Hence we also have a canonical map
$$
\fg\fr\ft_1 \lon H^0(\mathrm{Der}(\mathsf{BGC^{or}}))
$$
which implies the following

\subsubsection{\bf Proposition}\label{6: Claim on action of GRT on univ quant} {\em  The prounipotent Grothendieck-Teichm\"uller group  $GRT_1=\exp(\fg\fr\ft)$ acts  on the set of gauge
equivalence classes, $\cM\cC(\mathsf{BGC^{or}})/\sim$, of Maurer-Cartan elements in the $\caL ie_\infty$
algebra $\mathsf{BGC^{or}}$. Moreover, this action decsends to the action on
the subset $\cM\cC_{\mathrm{form}}\subset\cM\cC(\mathsf{BGC^{or}})$ of formality maps.
The action
 $$
 \Ba{ccc}
 GRT_1 \times \cM\cC_{\mathrm{form}}  &\lon & \cM\cC_{\mathrm{form}}\\
 (\exp(\ga), F=\sum \Ga)                & \lon & \sum \exp(\ga) \cdot \Ga
 \Ea
 $$
 is given by (repeated) substitutions of a graph representative of an element  $\ga\in \fg\fr
 \ft_1$ in $Z\mathsf{fGC_3^{or}}$ into black vertices of graphs $\Ga$ representing the formality map $F$ (cf.\ \cite{Wi}).}

\subsubsection{\bf From formality maps to quantizations of strongly homotopy algebras}\label{4: observation on formality and sh Lie alg}  Any element $q \in \cM\cC(\mathsf{BGC}^{or})$ is the same as a morphism of props,
which deforms the standard morphism $q_0$ in such a way that its component
$i_0$ stays invariant. As the prop $\cB ra^{or}$ admits  a canonical representation in $\cE nd_{\f_V,\fg_V}$ for any (possibly, infinite-dimensional) graded vector space $V$,
the universal quantization $q$ gives us a representation
$$
\cQ ua_\infty \stackrel{q_0+q}{\lon} \cB ra \lon \cE nd_{\fg_V,\f_V},
$$
that is,
a $\caL ie_\infty$ morphism,
$$
F=\left\{
F_k: \wedge^k (\fg_V[2]) \lon C^\bu_{GS}(\f_V,\f_V),\ \ \ |F_k|=1-k\right\}_{k\geq 1}
$$
If $q$ satisfies in addition the condition
(\ref{4: universal formality condition}), this $\caL ie_\infty$ morphism $F$ must be a quasi-isomorphism.

\mip

Let $\nu$ be an arbitrary Maurer-Cartan element  in $\fg_V$, that is,  an arbitrary strongly homotopy Lie bialgebra structure on $V$. Then the morphism $F$ gives us in turn an associated Maurer-Cartan element,
\Beq\label{3: qa-qua}
\rho^{\hbar}:=\rho_0+ \sum_{k=1}^\infty \frac{\hbar^{k}}{k!}F_k(\ga,\ldots, \ga),
\Eeq
in  $\fg\fs(\f_V,\f_V)[[\hbar]]$, that is, a continuous $\cA ss\cB_\infty$ structure in $\f_V[[\hbar]]$ which deforms the standard $\cA ss\cB$ structure in $\f_V$,
$$
\rho^\hbar|_{\hbar=0}= \rho_0,
$$
 and also satisfies,
$$
\frac{d \rho^{\hbar}}{d\hbar}|_{\hbar=0}=\ga.
$$
Thus any universal formality map gives a {\em universal quantization of strongly homotopy
algebras}.

\bip

\bip

{\Large
\section{\bf Existence and classification of universal formality maps}
}

\bip

\subsection{A one coloured prop of graph complexes} Consider an $\bS$-bimodule,
$$
\cB^{or}:=\left\{\cB^{or}(m,n):=\prod_{k\geq 0} \cB\cG ra^{or}(k;m,n)\ot_{\bS_k} \sgn_k[3k]\right\}.
$$
An element of $\cB^{or}$  can be understood as a graph $\Ga$ from $\cB ra^{or}$ whose  {\em internal}\, vertices have {\em no}\, labels and are totally ordered
(up to an even permutation)\footnote{A choice of such a  total ordering is called sometimes  an {\em orientation}\, on $\Ga$.  Every graph $\Ga$ from $\cB^{or}$ with more than one internal vertex has precisely two possible orientations, $or$ and $or^{opp}$, and one tacitly identifies $(\Ga,or)=-(\Ga,or^{opp})$. In our pictures  we show only $\Ga$,
a choice of orientation on $\Ga$ being tacitly assumed.}, and which is assigned the homological degree
\Beq\label{4: formula for degree in Bra circ}
|\Ga|=3|V_{int}(\Ga)| -2|E_{int}(\Ga)| -|E_{in}(\Ga)|-|E_{out}(\Ga)|.
\Eeq
One should, however, keep in mind that in reality such a graph $\Ga$ stands for a linear combination of graphs from  $\cP\cG ra^{or}$ whose internal vertices are skewsymmetrized, for example
$$
\Ba{c}\resizebox{1.3mm}{!}{\xy
 (0,0)*{\bullet}="a",
(0,5)*{\bu}="b",
(0,10)*{\circ}="u",
(0,-5)*{\circ}="d",
\ar @{->} "d";"a" <0pt>
\ar @{->} "a";"b" <0pt>
\ar @{->} "b";"u" <0pt>
\endxy}
\Ea =
 \Ba{c}\resizebox{2.8mm}{!}{\xy
(-2,0)*{_{2}},
(-2,5)*{_{1}},
 (0,0)*{\bullet}="a",
(0,5)*{\bu}="b",
(0,10)*{\circ}="u",
(0,-5)*{\circ}="d",
\ar @{->} "d";"a" <0pt>
\ar @{->} "a";"b" <0pt>
\ar @{->} "b";"u" <0pt>
\endxy}
\Ea  -\Ba{c}\resizebox{2.8mm}{!}{\xy
(-2,0)*{_{1}},
(-2,5)*{_{2}},
 (0,0)*{\bullet}="a",
(0,5)*{\bu}="b",
(0,10)*{\circ}="u",
(0,-5)*{\circ}="d",
\ar @{->} "d";"a" <0pt>
\ar @{->} "a";"b" <0pt>
\ar @{->} "b";"u" <0pt>
\endxy}
\Ea
$$
We use such an identification in the following definition of a degree 1 operation on $\cB^{or}$,
$$
\Ba{rccc}
\delta_{\bubu}: & \cB^{or} &\lon &  \cB^{or}\\
& \Ga &\lon & \delta_{\bubu}\Ga=\left\{\Ba{cc}
0 & \mbox{if}\ V_{int}(\Ga)=\emptyset\\
(-1)^{|\Ga|}\Ga\bu_1 \left(\xy
(0,2)*{_{1}},
(7,2)*{_{2}},
 (0,0)*{\bullet}="a",
(7,0)*{\bu}="b",
\ar @{->} "a";"b" <0pt>
\endxy\
- \
\xy
(0,2)*{_{2}},
(7,2)*{_{1}},
 (0,0)*{\bullet}="a",
(7,0)*{\bu}="b",
\ar @{->} "a";"b" <0pt>
\endxy\right) & \mbox{if}\ V_{int}(\Ga)\neq \emptyset\\
\Ea \right.
\Ea
$$
where $\bu_1$ stands for the substitution of the graph $\xy
(0,2)*{_{1}},
(7,2)*{_{2}},
 (0,0)*{\bullet}="a",
(7,0)*{\bu}="b",
\ar @{->} "a";"b" <0pt>
\endxy\
- \
\xy
(0,2)*{_{2}},
(7,2)*{_{1}},
 (0,0)*{\bullet}="a",
(7,0)*{\bu}="b",
\ar @{->} "a";"b" <0pt>
\endxy$ into the internal vertex labeled by $1$ in the skewsymmetrized linear combination of graphs corresponding to $\Ga$, and then the summation over all possible ways of reconnecting  the ``hanging" half-edges
to the two vertices of the substituted graph which respect the directed flow $\uparrow$ on edges and  the condition that each internal vertex of the resulting graph has at least one incoming half-edge and at least one outgoing half-edge (cf.\ (\ref{3: Example of prop coomp})). For example,
$$
\delta_\bubu\Ba{c}\resizebox{1.5mm}{!}{\xy
 (0,0)*{\bullet}="a",
(0,5)*{\circ}="u",
(0,-5)*{\circ}="d",
\ar @{->} "d";"a" <0pt>
\ar @{->} "a";"u" <0pt>
\endxy}\Ea=-\Ba{c}\resizebox{1.3mm}{!}{\xy
 (0,0)*{\bullet}="a",
(0,5)*{\bu}="b",
(0,10)*{\circ}="u",
(0,-5)*{\circ}="d",
\ar @{->} "d";"a" <0pt>
\ar @{->} "a";"b" <0pt>
\ar @{->} "b";"u" <0pt>
\endxy}
\Ea
$$

 Note that each element $\Ga$ from $\cB^{or}$ comes equipped with an {\em orientation}, i.e.\ an element (of unit length) in $\det \R^{|V_{int}(\Ga)|}\ot \det \R^{|E_{in}(\Ga)|} \ot \det  \R^{|E_{in}(\Ga)|}$, and that $\Ga$ is identified with $-\Ga^{opp}$, where $\Ga^{opp}$ is the unique graph which is identical to $\Ga$ in all data except orientation.

 \sip

\ \ \ \ It is obvious that the prop structure in $\cB ra^{or}$ induces
a prop structure in $\cB^{or}$.
What is less obvious is that the one-coloured prop $\cB^{or}$ has a natural differential $\delta_\bu$.

\sip

For any $l\geq 1$ and any $i\in [l]$ define a degree one element,
$$
\theta_i(l,l):=\Ba{c}\resizebox{13mm}{!}{\xy
 (0,7)*{^i},
  (0,-7)*{_i},
  (-6,5)*{...},
  (-6,-5)*{...},
  (-9,7)*{^1},
  (-9,-7)*{_1},
  (9,7)*{^l},
  (9,-7)*{_l},
 (-3,5)*{\circ},
  (-3,-5)*{\circ},
  (-6,5)*{...},
  (-6,-5)*{...},
  (-9,5)*{\circ},
  (-9,-5)*{\circ},
 (3,5)*{\circ},
  (3,-5)*{\circ},
  (6,5)*{...},
  (6,-5)*{...},
  (9,5)*{\circ},
  (9,-5)*{\circ},
 (0,0)*{\bullet}="a",
(0,5)*{\circ}="u",
(0,-5)*{\circ}="d",
\ar @{->} "d";"a" <0pt>
\ar @{->} "a";"u" <0pt>
\endxy}\Ea
$$
in $B^{or}(l,l)$, and then consider a linear map
$$
\Ba{rccc}
\delta_{\wibu}: & \cB^{or}(m,n) &\lon &  \cB^{or}(m,n)\\
& \Ga &\lon & \delta_{\wibu}\Ga:=\sum_{i=1}^m \theta_i(m,m)\circ_v \Ga
- (-1)^{|\Ga|}  \sum_{i=1}^n \Ga\circ_v \theta_i(n,n)
\Ea
$$
where $\circ_v$ stands for the vertical composition in the prop
$\cB ra^{or}$.

\begin{lemma}\label{4: Lemma for d}
The linear map $\delta_\bu:=\delta_\bubu+ \delta_\wibu$ is a differential
in the prop $\cB^{or}$ which makes it identical to the dg prop $\caD \wLB_\infty$.
\end{lemma}
\begin{proof} This claim can be proven in two different ways.

\sip

(i) By construction, every internal vertex of $\cB^{or}$ is at least trivalent.
A direct inspection of the formula for $\delta_\bu$ shows that it acts only
on black (i.e.\ internal) vertices,
$$
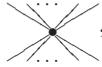
\resizebox{13mm}{!}{
\begin{xy}
 <0mm,0mm>*{\bu};<0mm,0mm>*{}**@{},
 <-0.6mm,0.44mm>*{};<-8mm,5mm>*{}**@{-},
 <-0.4mm,0.7mm>*{};<-4.5mm,5mm>*{}**@{-},
 <0mm,0mm>*{};<-1mm,5mm>*{\ldots}**@{},
 <0.4mm,0.7mm>*{};<4.5mm,5mm>*{}**@{-},
 <0.6mm,0.44mm>*{};<8mm,5mm>*{}**@{-},
 <-0.6mm,-0.44mm>*{};<-8mm,-5mm>*{}**@{-},
 <-0.4mm,-0.7mm>*{};<-4.5mm,-5mm>*{}**@{-},
 <0mm,0mm>*{};<-1mm,-5mm>*{\ldots}**@{},
 <0.4mm,-0.7mm>*{};<4.5mm,-5mm>*{}**@{-},
 <0.6mm,-0.44mm>*{};<8mm,-5mm>*{}**@{-},
 \end{xy}},
$$
 of graphs from $\cB^{or}$
by splitting them precisely as in (\ref{3: differential in LieBinfty}). The only non-obvious part of this claim is the formula for signs.  Recall that any graph $\Ga$ from $\cB \cG ra^{or}$ comes equipped with an orientation, that is, with a choice of an element (of unit length) in $\det \R^{|V_{int}(\Ga)|}\ot \det \R^{|E_{in}(\Ga)|}\ot \R^{|E_{out}(\Ga)|}$. It was proven in \cite{Th} that the latter space can be canonically identified with  $\det \R^{|e(\Ga)|}\ot \det \R^{\dim H_1(|Ga|)}$, where $e(\Ga)$ is the set of all edges of $\Ga$ and $H_1(|Ga|)$ is the first homology group of the geometric realization of $\Ga$. The map $\delta_\bu$ leaves $H_1(|Ga|)$ invariant so that, if we want to understand induced orientations on summands of the expression $d\Ga$ for any $\Ga\in \cB \cG ra^{or}$, it is enough
to take care only about the orientation of the space $\R^{|e(\Ga)|}$ generated by all edges. Then the translation of graph orientations  into signs becomes more or less straightforward: it was first done  in \cite{MaVo} (in a different context), and we refer to \S 7.1 of that paper for the proof that the formula for signs is exactly the one given by
(\ref{3: differential in LieBinfty}).

\sip

(ii) Consider a morphism of props,
$f: \cQ ua \lon \cB ra^{or}$,
given on the generators as follows (cf.\ (\ref{3: i_Gac to Bra})),
$$
f\left(\begin{xy}
 <0mm,-0.55mm>*{};<0mm,-3.5mm>*{}**@{.},
 <0.5mm,0.5mm>*{};<2.2mm,2.2mm>*{}**@{.},
 <-0.48mm,0.48mm>*{};<-2.2mm,2.2mm>*{}**@{.},
 <0mm,0mm>*{\circ};<0mm,0mm>*{}**@{},
 <0.5mm,0.5mm>*{};<2.7mm,2.8mm>*{^2}**@{},
 <-0.48mm,0.48mm>*{};<-2.7mm,2.8mm>*{^1}**@{},
 \end{xy}\right)=0,
 \ \ \ \ \ \ \ \ \ \ , \ \ \ \ \ \ \ \ \ \ \
f\left(\begin{xy}
 <0mm,0.66mm>*{};<0mm,4mm>*{}**@{.},
 <0.39mm,-0.39mm>*{};<2.2mm,-2.2mm>*{}**@{.},
 <-0.35mm,-0.35mm>*{};<-2.2mm,-2.2mm>*{}**@{.},
 <0mm,0mm>*{\circ};<0mm,0mm>*{}**@{},
   <0.39mm,-0.39mm>*{};<2.9mm,-4mm>*{^2}**@{},
   <-0.35mm,-0.35mm>*{};<-2.8mm,-4mm>*{^1}**@{},
\end{xy}\right)=0,
$$
$$
f\left(\Ba{c}
\xy
 <0mm,0.55mm>*{};<0mm,3.5mm>*{}**@{-},
 <0.5mm,-0.5mm>*{};<2.2mm,-2.2mm>*{}**@{-},
 <-0.48mm,-0.48mm>*{};<-2.2mm,-2.2mm>*{}**@{-},
 <0mm,0mm>*{\bu};<0mm,0mm>*{}**@{},
 <0.5mm,-0.5mm>*{};<2.7mm,-3.2mm>*{_2}**@{},
 <-0.48mm,-0.48mm>*{};<-2.7mm,-3.2mm>*{_1}**@{},
 \endxy\Ea\right)\ =\ \  \xy
(0,2)*{_{1}},
(7,2)*{_{2}},
 (0,0)*{\bullet}="a",
(7,0)*{\bu}="b",
\ar @{->} "a";"b" <0pt>
\endxy\
- \
\xy
(0,2)*{_{2}},
(7,2)*{_{1}},
 (0,0)*{\bullet}="a",
(7,0)*{\bu}="b",
\ar @{->} "a";"b" <0pt>
\endxy\ \ \ \ \ \  ,\ \ \ \ \ \ \
f\left(\begin{xy}
 <0mm,0.66mm>*{};<0mm,3mm>*{}**@{.},
 <0.39mm,-0.39mm>*{};<2.2mm,-2.2mm>*{}**@{.},
 <-0.35mm,-0.35mm>*{};<-2.2mm,-2.2mm>*{}**@{-},
 <0mm,0mm>*{\bu};<0mm,0mm>*{}**@{},
\end{xy}\right)\ =\ \ \xy
 (0,0)*{\bullet}="a",
(0,5)*{\circ}="u",
(0,-5)*{\circ}="d",
\ar @{->} "d";"a" <0pt>
\ar @{->} "a";"u" <0pt>
\endxy
$$
Then, as before, the dg vector space
$$
\Def(\cQ ua_\infty \stackrel{f}{\rar} \cB ra^{or})\simeq \mathsf{fGC}_3^{or} \oplus \mathsf{BGC^{or}_3}
$$
splits into a direct sum of complexes. Up to a degree shift graphs from  $\mathsf{BGC}_3^{or}$ can be identified with graphs from  $\cB^{or}$. It remains to notice that the differential  in $\mathsf{BGC}_3^{or}$ induced by the morphism $f$ is precisely
the sum $d_\bubu+d_\wibu$.
\end{proof}

As the functor $\caD$ exact, the above result computes immediately the cohomology of the dg prop $\cB^{or}$.

\subsubsection{\bf Corollary}\label{4: Cohomology of Bra_circ} {\em The canonical surjection
$
\pi: \cB^{or}\equiv \caD\wLB_\infty \lon \caD\wLB
$
 is a quasi-isomorphism.}

\subsection{The graph complex $\mathsf{BGC^{or}}$ reinterpreted} Note that elements
in $\mathsf{BBC}_3^{or}$ which have no internal vertices are cycles with respect to the differential $\delta_\bu$, that is,
$$
\delta_\bu \left(\Ba{c}\resizebox{13mm}{!}{\xy
  (-6,5)*{...},
  (-6,-5)*{...},
  (-9,7)*{^1},
  (-9,-7)*{_1},
  (9,7)*{^m},
  (9,-7)*{_n},
 (-3,5)*{\circ},
  (-3,-5)*{\circ},
  (-6,5)*{...},
  (-6,-5)*{...},
  (-9,5)*{\circ},
  (-9,-5)*{\circ},
 (3,5)*{\circ},
  (3,-5)*{\circ},
  (6,5)*{...},
  (6,-5)*{...},
  (9,5)*{\circ},
  (9,-5)*{\circ},
(0,5)*{\circ}="u",
(0,-5)*{\circ}="d",
\endxy}\Ea\right)=0.
$$
Therefore there is a morphism
of {\em dg}\, props,
$$
i_\circ: (\cA ss\cB,0) \lon (\caD\wLB_\infty, \delta_\bu)
$$
given on the generators as follows
\Beq\label{4: map g from AssB to Bra circ}
i_\circ\left(\begin{xy}
 <0mm,-0.55mm>*{};<0mm,-3.5mm>*{}**@{.},
 <0.5mm,0.5mm>*{};<2.2mm,2.2mm>*{}**@{.},
 <-0.48mm,0.48mm>*{};<-2.2mm,2.2mm>*{}**@{.},
 <0mm,0mm>*{\circ};<0mm,0mm>*{}**@{},
 <0.5mm,0.5mm>*{};<2.7mm,2.8mm>*{^2}**@{},
 <-0.48mm,0.48mm>*{};<-2.7mm,2.8mm>*{^1}**@{},
 \end{xy}\right)=
 \Ba{c}\resizebox{8mm}{!}{ \xy
(-3,9)*{^1},
(3,9)*{^2},
 (0,2)*{\circ}="a",
(-3,7)*{\circ}="b_1",
(3,7)*{\circ}="b_2",
 \endxy}\Ea  \ \ \ \ \ \ \ \ \ \ , \ \ \ \ \ \ \ \ \ \ \
i_\circ\left(\begin{xy}
 <0mm,0.66mm>*{};<0mm,4mm>*{}**@{.},
 <0.39mm,-0.39mm>*{};<2.2mm,-2.2mm>*{}**@{.},
 <-0.35mm,-0.35mm>*{};<-2.2mm,-2.2mm>*{}**@{.},
 <0mm,0mm>*{\circ};<0mm,0mm>*{}**@{},
   <0.39mm,-0.39mm>*{};<2.9mm,-4mm>*{^2}**@{},
   <-0.35mm,-0.35mm>*{};<-2.8mm,-4mm>*{^1}**@{},
\end{xy}\right)=\Ba{c}\resizebox{8mm}{!}{ \xy
(-3,0)*{_1},
(3,0)*{_2},
 (0,7)*{\circ}="a",
(-3,2)*{\circ}="b_1",
(3,2)*{\circ}="b_2",
 \endxy}\Ea
\Eeq
and hence a morphism of dg props $\cA ss\cB_\infty \rar \caD\wLB_\infty$ denoted by the same letter $i_\circ$.
Using the above identification of $\caD\LB_\infty$ with $\cB^{or}$ and the relation of the latter with $\cB ra^{or}$, it is now straightforward to see 
that the deformation complex
of the morphism $i_\circ$,
$$
\Def(\cA ss\cB \stackrel{i_\circ}{\lon} \caD\wLB_\infty   )
$$
can be identified as a $\caL ie_\infty$ algebra with the earlier defined graph complex $\mathsf{BGC}^{or}$.
Thus the Maurer-Cartan elements $\ga$ of, for example, the $\caL ie_\infty$ algebra  $\mathsf{BGC}^{or}$ can be understood not only as morphisms of two-coloured dg props,
$$
(\cQ ua_\infty, \delta) \stackrel{\ga}{\lon} (\cB ra^{or},0)
$$
but also as morphisms of one-coloured dg props,
$$
(\cA ss\cB_\infty, \delta) \stackrel{\ga}{\lon} (\caD\wLB_\infty, \delta_\bu).
$$
This double meaning of one and the same element $\ga\in \cM\cC(\mathsf{BGC^{or}})$ is not  surprising and  can be understood in terms of representations as follows.

\sip


Let $\nu:\LB_\infty \rar \cE nd_V$ be a strongly homotopy Lie bialgebra structure in a dg vector space $V$; we can understand it  as
a Maurer-Cartan element $\nu$ in the Lie algebra $\fg_V$. According to the general principle (see \S 2), there is an associated representation, $\rho_\nu: \caD\LB_\infty \lon \cE nd_{\f_V}$ which we would like to describe explicitly.
In fact, as we work with {\em completed}\, prop $\wLB_\infty$ one has to take care about convergence in the standard way: the associated representation will be a {\em continuous}\, morphism $\rho_\nu: \caD\wLB_\infty \lon \cE nd_{\f_V}[[\hbar]]$ of topological props.

\sip
Recall (see \S {\ref{4: Subsection on repr of Bra}}) that there is a canonical representation of the 2-coloured prop $\cB ra^{or}$  in the pair of vector spaces $(\fg, \f_V)$  which associates to a graph
$\Ga$ from $\cB\cG  ra^{or}(k;m,n)$  a linear map (\ref{4: formular for rho_Ga for Bra}),
$$
\Ba{rccc}
\rho_\Ga: & \fg_V^{\ot k}\ot \f_V^{\ot n} & \lon & \f_V^{\ot m}\\
           & \ga_1\ot\ldots\ot \ga_k\ot f_1\ot\ldots\ot f_n &\lon& \rho_\Ga(\ga_1\ot\ldots\ot \ga_k\ot f_1\ot\ldots\ot f_n).
\Ea
$$
According to Lemma {\ref{4: Lemma for d}},  a generator of $\caD\wLB_\infty(m,n)$ can be understood as  a graph $\Ga^{skew}$ obtained by skewsymmetrization of labels of internal vertices of some graph $\Ga$ from $\cB \cG ra(k;m,n)$. Then the formula
$$
\Ba{rccc}
\Phi_\Ga^\nu:=\rho(\Ga^{skew}): & \f_V^{\ot n} & \lon & \f_V^{\ot m}[[\hbar]]\\
           &  f_1\ot\ldots\ot f_n &\lon& \frac{\hbar^k}{k!}\rho_\Ga(\nu\ot\ldots\ot \nu\ot f_1\ot\ldots\ot f_n)
\Ea
$$
gives us the required continous representation of  $\caD\wLB_\infty$ in $\f_V[[\hbar]]$. Comparing this formula with (\ref{3: qa-qua}), we see that any Maurer-Cartan element $\ga$ of the $\caL ie_\infty$ algebra  $\mathsf{BGC^{or}}$ can indeed be  understood as a morphisms of one-coloured dg props
$\ga': \cA ss\cB_\infty \stackrel{\ga}{\lon} \caD\wLB_\infty$
such that the $\cA ss\cB_\infty$ algebra structure in $\f_V[[\hbar]]$ induced from the representation
$\nu$ in accordance with the formality map $\ga\in \mathsf{BGC^{or}}$ and formula  (\ref{3: qa-qua}) coincides precisely with the composition
$$
\cA ss\cB_\infty \stackrel{\ga'}{\lon} \cB ra^{or} \stackrel{\rho_\nu}{\lon}  \cE nd_{\f_V}[[\hbar]].
$$
 Put another way, the identification of $\caL ie_\infty$ algebras $\mathsf{BGC^{or}}$ and
 $\Def(\cA ss\cB \stackrel{i_\circ}{\lon} \caD\wLB_\infty   )$ is no more than a universal (graph) incarnation of the main observation in \S {\ref{4: observation on formality and sh Lie alg}} that formality maps give us quantizations of strongly homotopy Lie bialgebras.

 \sip

 Similarly, morphisms of (non-differential) props,
$$
\cA ss\cB \lon \caD\wLB
$$
can be understood as universal quantizations of Lie bialgebras.

\subsection{Existence Theorem}\label{5: Existence Theorem for formalities} {\em For any Drinfeld associator $\fA$ there is an associated formality map $F_\fA\in \cM\cC_{\mathrm{form}}$. In particular, the set of universal formality
maps is non-empty.}
\begin{proof}
The Etingof-Kazhdan theorem \cite{EK} (see also \cite{EE}) says that for any Drinfeld associator $\fA$ there
there is a morphism of props
$$
f_\fA: \Assb \lon \caD\wLB
$$
such that
\Beqrn
f_\fA\left(\begin{xy}
 <0mm,0.66mm>*{};<0mm,3mm>*{}**@{.},
 <0.39mm,-0.39mm>*{};<2.2mm,-2.2mm>*{}**@{.},
 <-0.35mm,-0.35mm>*{};<-2.2mm,-2.2mm>*{}**@{.},
 <0mm,0mm>*{\circ};<0mm,0mm>*{}**@{},
   <0.39mm,-0.39mm>*{};<2.9mm,-4mm>*{^2}**@{},
   <-0.35mm,-0.35mm>*{};<-2.8mm,-4mm>*{^1}**@{},
\end{xy}\right)
&=&
\Ba{c}\resizebox{8mm}{!}{ \xy
(-3,0)*{_1},
(3,0)*{_2},
 (0,7)*{\circ}="a",
(-3,2)*{\circ}="b_1",
(3,2)*{\circ}="b_2",
 \endxy}\Ea \ + \ \frac{1}{2}\hspace{-4mm} \Ba{c}\resizebox{12mm}{!}{
\xy
(-5,0)*{_1},
(5,0)*{_2},
(0,13)*{\circ}="0",
 (0,7)*{\bu}="a",
(-5,2)*{\circ}="b_1",
(5,2)*{\circ}="b_2",
(-8,-2)*{}="c_1",
(-2,-2)*{}="c_2",
(2,-2)*{}="c_3",
\ar @{->} "a";"0" <0pt>
\ar @{<-} "a";"b_1" <0pt>
\ar @{<-} "a";"b_2" <0pt>
\endxy}
\Ea\
\ \ + \ \ \ \mbox{terms\ with\ $\geq 2$\ black (internal) vertices}
 \\
f_\fA\left(\begin{xy}
 <0mm,-0.55mm>*{};<0mm,-3.5mm>*{}**@{.},
 <0.5mm,0.5mm>*{};<2.2mm,2.2mm>*{}**@{.},
 <-0.48mm,0.48mm>*{};<-2.2mm,2.2mm>*{}**@{.},
 <0mm,0mm>*{\circ};<0mm,0mm>*{}**@{},
 <0.5mm,0.5mm>*{}  ;<2.7mm,2.8mm>*{^2}**@{},
 <-0.48mm,0.48mm>*{};<-2.7mm,2.8mm>*{^1}**@{},
 \end{xy}\right)
 &=&
 \Ba{c}\resizebox{8mm}{!}{ \xy
(-3,9)*{^1},
(3,9)*{^2},
 (0,2)*{\circ}="a",
(-3,7)*{\circ}="b_1",
(3,7)*{\circ}="b_2",
 \endxy}\Ea \ + \ \frac{1}{2}
 \hspace{-4mm} \Ba{c}\resizebox{12mm}{!}{
\xy
(-5,7)*{_1},
(5,7)*{_2},
(0,-6)*{\circ}="0",
 (0,0)*{\bu}="a",
(-5,5)*{\circ}="b_1",
(5,5)*{\circ}="b_2",
(-8,9)*{}="c_1",
(-2,9)*{}="c_2",
(2,9)*{}="c_3",
\ar @{<-} "a";"0" <0pt>
\ar @{->} "a";"b_1" <0pt>
\ar @{->} "a";"b_2" <0pt>
\endxy}
\Ea
\ \ \ +\ \ \  \mbox{terms\ with\ $\geq 2$\ black (internal) vertices}
\Eeqrn
The theorem is proven once we show that there exists a morphism of dg props
$$
F_\fA: \Assb_\infty \lon \cB^{or}\equiv \caD\wLB_\infty
$$
which makes the following diagram commutative,
$$
 \xymatrix{
\Assb_\infty\ar[r]^{F_\fA}\ar[d]_p  & \caD\wLB_\infty\ar[d]^\pi\\
 \Assb \ar[r]_{f_\fA} &
 \caD\wLB}
$$
and satisfies the conditions
 \Beq\label{5: Boundary cond for formality map}
\pi_1\circ F_\fA\left(\Ba{c}\resizebox{13mm}{!}{ \xy
 (0,7)*{\overbrace{\ \ \ \  \ \ \ \ \ \ \ \ \ \ }},
 (0,9)*{^m},
 (0,3)*{^{...}},
 (0,-3)*{_{...}},
 (0,-7)*{\underbrace{  \ \ \ \  \ \ \ \ \ \ \ \ \ \ }},
 (0,-9)*{_n},
 (0,0)*{\circ}="0",
(-7,5)*{}="u_1",
(-4,5)*{}="u_2",
(4,5)*{}="u_3",
(7,5)*{}="u_4",
(-7,-5)*{}="d_1",
(-4,-5)*{}="d_2",
(4,-5)*{}="d_3",
(7,-5)*{}="d_4",
\ar @{.} "0";"u_1" <0pt>
\ar @{.} "0";"u_2" <0pt>
\ar @{.} "0";"u_3" <0pt>
\ar @{.} "0";"u_4" <0pt>
\ar @{.} "0";"d_1" <0pt>
\ar @{.} "0";"d_2" <0pt>
\ar @{.} "0";"d_3" <0pt>
\ar @{.} "0";"d_4" <0pt>
\endxy}\Ea\right)=
\frac{1}{m!n!}\Ba{c}\resizebox{16mm}{!}{\xy
(0,7.5)*{\overbrace{\ \ \ \ \ \ \ \ \  \ \ \ \ \ \ \ \ \ \ }},
 (0,9.5)*{^m},
 (0,-7.5)*{\underbrace{\ \ \ \ \ \ \ \ \ \ \ \  \ \ \ \ \ \ }},
 (0,-9.9)*{_n},
  (-6,5)*{...},
  (-6,-5)*{...},
 (-3,5)*{\circ}="u1",
  (-3,-5)*{\circ}="d1",
  (-6,5)*{...},
  (-6,-5)*{...},
  (-9,5)*{\circ}="u2",
  (-9,-5)*{\circ}="d2",
 (3,5)*{\circ}="u3",
  (3,-5)*{\circ}="d3",
  (6,5)*{...},
  (6,-5)*{...},
  (9,5)*{\circ}="u4",
  (9,-5)*{\circ}="d4",
 (0,0)*{\bullet}="a",
\ar @{->} "d1";"a" <0pt>
\ar @{->} "a";"u1" <0pt>
\ar @{->} "d2";"a" <0pt>
\ar @{->} "a";"u2" <0pt>
\ar @{->} "d3";"a" <0pt>
\ar @{->} "a";"u3" <0pt>
\ar @{->} "d4";"a" <0pt>
\ar @{->} "a";"u4" <0pt>
\endxy}\Ea
\Eeq
for all $m+n\geq 4$, $m,n\geq 1$. Here $\pi_1: \caD\wLB_\infty \rar \caD\wLB_\infty^{(1)}$ is the projection to the subspace spanned by graphs with precisely one black (internal) vertex.

\sip

One can construct the required morphism $F$ by a standard induction on the degree of the generators of $\Assb_\infty$.
As the map $\pi$ is a surjective quasi-isomorphism,
there exists cycles $B_2^1$ and $B_1^2$ in $\caD\LB_\infty$ such that
$$
\pi(B_2^1)= f_\fA\left(\begin{xy}
 <0mm,0.66mm>*{};<0mm,3mm>*{}**@{.},
 <0.39mm,-0.39mm>*{};<2.2mm,-2.2mm>*{}**@{.},
 <-0.35mm,-0.35mm>*{};<-2.2mm,-2.2mm>*{}**@{.},
 <0mm,0mm>*{\circ};<0mm,0mm>*{}**@{},
\end{xy}\right)\ \ \ \  ,\ \ \ \ \
\pi(B_1^2)=f\left(\begin{xy}
 <0mm,-0.55mm>*{};<0mm,-3.5mm>*{}**@{.},
 <0.5mm,0.5mm>*{};<2.2mm,2.2mm>*{}**@{.},
 <-0.48mm,0.48mm>*{};<-2.2mm,2.2mm>*{}**@{.},
 <0mm,0mm>*{\circ};<0mm,0mm>*{}**@{},
 \end{xy}\right).
$$
We set the values of $F$ on degree zero generators to be given by
$$
F_\fA\left(\begin{xy}
 <0mm,0.66mm>*{};<0mm,3mm>*{}**@{.},
 <0.39mm,-0.39mm>*{};<2.2mm,-2.2mm>*{}**@{.},
 <-0.35mm,-0.35mm>*{};<-2.2mm,-2.2mm>*{}**@{.},
 <0mm,0mm>*{\circ};<0mm,0mm>*{}**@{},
\end{xy}\right):= B_2^1\ \ \ \ \ , \ \ \ \ \
F_\fA\left(\begin{xy}
 <0mm,-0.55mm>*{};<0mm,-3.5mm>*{}**@{.},
 <0.5mm,0.5mm>*{};<2.2mm,2.2mm>*{}**@{.},
 <-0.48mm,0.48mm>*{};<-2.2mm,2.2mm>*{}**@{.},
 <0mm,0mm>*{\circ};<0mm,0mm>*{}**@{},
 \end{xy}\right):= B_1^2.
$$

\sip


Assume now that values of $F$ are defined on the generators of $\Assb_\infty$ of degree
$\geq -N+1$. Let us define $F$ on generators of degree $-N$, that is, on $(m,n)$ corollas,
$c_n^m:=\Ba{c}\resizebox{8mm}{!}{
 \xy
 (0,7)*{\overbrace{\ \ \ \  \ \ \ \ \ \ \ \ \ \ }},
 (0,9)*{^m},
 (0,3)*{^{...}},
 (0,-3)*{_{...}},
 (0,-7)*{\underbrace{  \ \ \ \  \ \ \ \ \ \ \ \ \ \ }},
 (0,-9)*{_n},
 (0,0)*{\circ}="0",
(-7,5)*{}="u_1",
(-4,5)*{}="u_2",
(4,5)*{}="u_3",
(7,5)*{}="u_4",
(-7,-5)*{}="d_1",
(-4,-5)*{}="d_2",
(4,-5)*{}="d_3",
(7,-5)*{}="d_4",
\ar @{.} "0";"u_1" <0pt>
\ar @{.} "0";"u_2" <0pt>
\ar @{.} "0";"u_3" <0pt>
\ar @{.} "0";"u_4" <0pt>
\ar @{.} "0";"d_1" <0pt>
\ar @{.} "0";"d_2" <0pt>
\ar @{.} "0";"d_3" <0pt>
\ar @{.} "0";"d_4" <0pt>
\endxy}\Ea
$
satisfying $m+n-3=N$.  Note that $\delta e_n^m$ is a linear combination of graphs whose vertices are decorated by
corollas of degrees
$\geq -N+1$ (as  $\delta$ increases degree by $+1$).
By induction, ${F}(\delta  e_n^m)$
is a well-defined closed element in $\cB ra^{or}$. As $f\circ p(e_n^m)=0$, we have
$\pi({F}(\delta  e_n^m))=0$. Since the surjection $\pi$ is a
quasi-isomorphism, the element ${F}(\delta  e_n^m)$ must be exact. Thus there exists $B_n^m\in \cB ra^{or}$ such that
$$
\delta_\bu {B_n^m}= {\cF}(\delta e_n^m).
$$
We set ${F}(e_n^m):= B_n^m$.

\sip

To complete the inductive construction of ${F}$ one has to show that
$\pi_1\circ F(e_n^m)=\frac{1}{m!n!}
\Ba{c}\resizebox{8mm}{!}{
\xy
(0,7.5)*{\overbrace{\ \ \ \ \ \ \ \ \  \ \ \ \ \ \ \ \ \ \ }},
 (0,9.5)*{^m},
 (0,-7.5)*{\underbrace{\ \ \ \ \ \ \ \ \ \ \ \  \ \ \ \ \ \ }},
 (0,-9.9)*{_n},
  (-6,5)*{...},
  (-6,-5)*{...},
 (-3,5)*{\circ}="u1",
  (-3,-5)*{\circ}="d1",
  (-6,5)*{...},
  (-6,-5)*{...},
  (-9,5)*{\circ}="u2",
  (-9,-5)*{\circ}="d2",
 (3,5)*{\circ}="u3",
  (3,-5)*{\circ}="d3",
  (6,5)*{...},
  (6,-5)*{...},
  (9,5)*{\circ}="u4",
  (9,-5)*{\circ}="d4",
 (0,0)*{\bullet}="a",
\ar @{->} "d1";"a" <0pt>
\ar @{->} "a";"u1" <0pt>
\ar @{->} "d2";"a" <0pt>
\ar @{->} "a";"u2" <0pt>
\ar @{->} "d3";"a" <0pt>
\ar @{->} "a";"u3" <0pt>
\ar @{->} "d4";"a" <0pt>
\ar @{->} "a";"u4" <0pt>
\endxy}\Ea
$.
This step is the same as in \cite{Me1} and we omit the details.
\end{proof}


\subsection{Deformation complex of a formality map}
Let $F_\fA$ be any formality map. By composing derivations of the {\em prop}\, $\LB_\infty$ with $F_\fA$ we obtain a morphism of complexes,
\Beq\label{5: s from Der Lieb to Def}
s: \Der(\wLB_\infty)_{\mathrm{prop}}[-1]\equiv \Def\left(\wLB_\infty \stackrel{\Id}{\lon}\wLB_\infty\right)_{\mathrm{prop}} \lon \Def\left(\Assb_\infty \stackrel{F_\fA}{\lon} \caD\wLB_\infty \right)
\Eeq

\subsubsection{\bf Proposition}\label{5: Prop on map s} {\em The map $s$ in (\ref{5: s from Der Lieb to Def}) is a quasi-isomorphism}.
\begin{proof} The complex $\Der(\wLB_\infty)$ of {\em properadic}\, derivations is described
very explicitly in terms of {\em connected}\, graphs in \cite{MW2}; the complex $\Der(\LB_\infty)_{\mathrm{prop}}$ of derivations of the {\em prop}\, $\LB_\infty$ can be described by the same graphs but with the {\em connectedness}\, assumption dropped.  Both complexes in (\ref{5: s from Der Lieb to Def}) admit filtrations by the number of edges in the graphs, and the map $s$ preserves these filtrations, and hence
induces a morphism of the associated spectral sequences,
$$
s_r: (\cE_r\Der(\wLB_\infty)_{\mathrm{prop}}[-1], d_r) \lon \left(\cE_r\Def(\Assb_\infty \stackrel{F_\fA}{\lon} \caD\wLB_\infty ), \delta_r\right), \ \ \ r\geq 0,
$$
The induced differential $d_0$ on the initial page of the spectral sequence of the l.h.s.\ is trivial, $d_0=0$. The induced differential  on the initial page of the spectral sequence of the r.h.s. is not trivial and is determined by the following
summand
$$
\Ba{c}\resizebox{8mm}{!}{ \xy
%
 (0,3)*{\circ}="a",
(-3,-2)*{\circ}="b_1",
(3,-2)*{\circ}="b_2",
(-5,-2)*{}="x",
(5,-2)*{}="y",
(-5,3)*{}="X",
(5,3)*{}="Y",
\ar @{.} "x";"y" <0pt>
\ar @{.} "X";"Y" <0pt>
 \endxy}\Ea\ \ \ \ \ + \ \ \ \ \
 \Ba{c}\resizebox{8mm}{!}{ \xy
%
 (0,-2)*{\circ}="a",
(-3,3)*{\circ}="b_1",
(3,3)*{\circ}="b_2",
(-5,3)*{}="x",
(5,3)*{}="y",
(-5,-2)*{}="X",
(5,-2)*{}="Y",
\ar @{.} "x";"y" <0pt>
\ar @{.} "X";"Y" <0pt>
 \endxy}\Ea
$$
 in $F_\fA$. Hence the differential $\delta_0$ acts only on white vertices of graphs from $ \Def(\Assb \stackrel{F_\fA}{\lon} \caD\wLB_\infty )$ by splitting them in two white vertices
and redistributes outgoing or incoming edges and all possible ways. The cohomology
$$
E_1\Def\left(\Assb_\infty \stackrel{F_\fA}{\lon} \caD\wLB_\infty \right)=H^\bu\left(E_0\Def(\Assb_\infty \stackrel{F_\fA}{\lon} \caD\wLB_\infty ), \delta_0\right)
$$
is spanned by graphs all of whose white vertices are precisely univalent (see Theorem 3.2.4 in \cite{Me2} or Appendix A in \cite{Wi}) and hence is isomorphic to $\Der(\wLB_\infty)_{\mathrm{prop}}[-1]$ as a graded vector space.
By boundary condition (\ref{5: Boundary cond for formality map}), the induced differential $\delta_1$ in $E_1\Def(\Assb \stackrel{F_\fA}{\lon} \caD\wLB_\infty )$ agrees with the induced differential $d_1$ in
$\cE_1\Der(\wLB_\infty)_{\mathrm{prop}}[-1]$ so that the morphism of the next pages of the spectral sequences,
$$
s_1: (\cE_1\Der(\wLB_\infty)_{\mathrm{prop}}[-1], d_1)\simeq \Der(\wLB_\infty)_{\mathrm{prop}}[-1] \lon \left(\cE_1\Def(\Assb_\infty \stackrel{F_\fA}{\lon} \caD\wLB_\infty ), \delta_1\right)
$$
is an isomorphism. By the Comparison Theorem, the morphism $s$ is a quasi-isomorphism.
\end{proof}

\subsubsection{\bf Corollary}\label{5: Corollary on GCor and Def(assb to Dlie)} {\em For any formality morphism $F_\fA$ the associated deformation complex comes equipped with a canonical morphism
 of complexes
 $$
 \mathsf{fGC}_3^\uparrow \lon \Def\left(\Assb_\infty \stackrel{F_\fA}{\lon} \caD \wLB_\infty\right)[1]
 $$
 which is a quasi-isomorphism. In particular,
  $$
  H^{\bu+1}\left(\Def(\Assb_\infty \stackrel{F_\fA}{\lon} \caD\LB_\infty )\right)= H^\bu(\mathsf{fGC}^\uparrow_3).
  $$
}
\begin{proof} There is a morphism of dg Lie algebras (\ref{4: morphism from fGC_3 to derLieb-prop}) which is a quasi-isomorphism. Hence the claim follows from
Proposition {\ref{5: Prop on map s}}.
\end{proof}

\subsection{Proof of the Main Theorem {\ref{1: Main Theorem}}}

Claim (i) follows from Theorem {\ref{5: Existence Theorem for formalities}}.

\sip

Claim (ii) follows from the identification of the deformation complex
$\Def(\Assb_\infty \stackrel{F_\fA}{\lon} \caD\wLB_\infty )$ with the subcomplex $\mathsf{BGC}_3^{or}$ of the complex $\Def(\cQ ua_\infty \stackrel{f}{\rar} \cB ra^{or})$.

\mip

Claim (iii) is the Corollary {\ref{5: Corollary on GCor and Def(assb to Dlie)}}.
\mip

Claim (iv) follows from Corollary {\ref{5: Corollary on GCor and Def(assb to Dlie)}}
and the fact (\ref{4: Formula for H^0 DerLieb-prop}) that $H^0(\mathsf{fGC}^\uparrow_3)=\grt$.

\

\def\cprime{$'$}

\end{document}